\documentclass[11pt,reqno]{amsart}

\usepackage{fullpage}
\usepackage{amsthm, amssymb, amstext, amscd, amsfonts, amsxtra, latexsym, amsmath, comment, tikz,
	mathrsfs, mathtools, esint, stmaryrd, cancel, mathtools}
\usepackage{color}
\usepackage{fancybox}
\usepackage[english]{babel}
\usepackage[latin1]{inputenc}
\usepackage{enumitem}
\usepackage{cancel}

\usepackage{graphicx}
\usepackage{mdframed}
\usepackage{stackengine}
\usepackage{scalerel}
\usepackage{bbm}
\usepackage{scalerel,stackengine}
\usepackage{graphicx} 
\usepackage{bm}
\usetikzlibrary{calc}

\stackMath
\newcommand\reallywidehat[1]{%
	\savestack{\tmpbox}{\stretchto{%
			\scaleto{%
				\scalerel*[\widthof{\ensuremath{#1}}]{\kern-.6pt\bigwedge\kern-.6pt}%
				{\rule[-\textheight/2]{1ex}{\textheight}}
			}{\textheight}%
		}{0.5ex}}%
	\stackon[1pt]{#1}{\tmpbox}%
}
\newlength{\parenheight}
\newlength{\parendepth}
\newlength{\parendrop}

\newcommand{\paren}[4]{%
	\settoheight{\parenheight}{\(#4 #2\)}%
	\settodepth{\parendepth}{\(#4 #2\)}
	\addtolength{\parendepth}{.5ex}
	\addtolength{\parenheight}{-.5ex}
	\addtolength{\parenheight}{\parendepth}
	\addtolength{\parendepth}{-.5\parenheight}
	\setlength{\parendrop}{-.5\parenheight}
	\addtolength{\parendrop}{.5ex}
	\raisebox{-\parendepth}{\(#4
		\left#1%
		\rule[\parendrop]{0pt}{\parenheight}%
		\right.\)}
	#2
	\raisebox{-\parendepth}{\(#4
		\left.%
		\rule[\parendrop]{0pt}{\parenheight}%
		\right#3\)}
}

\def\myleft#1#2\myright#3{%
	\mathchoice{%
		\paren{#1}{#2}{#3}{\displaystyle}%
	}{%
		\paren{#1}{#2}{#3}{\textstyle}%
	}{%
		\paren{#1}{#2}{#3}{\scriptstyle}%
	}{%
		\paren{#1}{#2}{#3}{\scriptscriptstyle}%
	}%
}

\numberwithin{equation}{section}
\newtheorem{theorem}{Theorem}
\numberwithin{theorem}{section}
\newtheorem{lemma}[theorem]{Lemma}
\newtheorem{proposition}[theorem]{Proposition}
\newtheorem{corollary}[theorem]{Corollary}
\newtheorem{assumption}[theorem]{Assumption}

\newtheorem{example}[theorem]{Example}
\newtheorem{remark}[theorem]{Remark}

\theoremstyle{plain}

\theoremstyle{definition}
\newtheorem{definition}[theorem]{Definition}

\theoremstyle{remark}

\usepackage{enumitem}
\setlist[enumerate]{leftmargin=*, listparindent=\parindent, parsep=0pt, font=\upshape, label=(\arabic*)} 
\setlist[itemize]{leftmargin=*} 

\newcommand{\mat}[1]{\left( \begin{matrix} #1 \end{matrix} \right)}

\newcommand{\pabcd}{\left(\begin{smallmatrix}a&b\\c&d\end{smallmatrix}\right)}

\newcommand{\Z}{\mathbb{Z}}
\newcommand{\N}{\mathbb{N}}
\newcommand{\Q}{\mathbb{Q}}
\newcommand{\U}{\mathcal{U}}
\newcommand{\R}{\mathbb{R}}
\newcommand{\C}{\mathbb{C}}
\renewcommand{\L}{\mathbb{L}}

\newcommand{\sgn}{\operatorname{sgn}}

\newcommand{\Mat}{\operatorname{Mat}}
\newcommand{\reg}{\operatorname{reg}}
\newcommand{\supp}{\operatorname{supp}}
\newcommand{\Id}{\operatorname{Id}}

\renewcommand{\H}{\mathbb{H}}
\newcommand{\SL}{\text{\rm SL}}

\newcommand{\bb}[1]{\boldsymbol{#1}}

\newcommand{\sm}{\setminus}

\makeatletter
\newcommand{\vast}{\bBigg@{3.5}}
\newcommand{\Vast}{\bBigg@{4.5}}
\newcommand{\VVast}{\bBigg@{6}}
\makeatother

\renewcommand{\b}[1]{\boldsymbol{#1}}

\newcommand{\sangle}[1]{\langle #1\rangle}

\newcommand{\wh}{\widehat}
\def\del{  \partial}

\renewcommand{\pmod}[1]{\  \,  \left( \rm{mod} \,  #1 \right)}

\newcommand{\xol}[1]{\xoverline{#1}}
\newcommand{\ol}[1]{\overline{#1}}

\def\a{\alpha}
\def\b{\beta}
\def\d{\delta}
\def\e{\varepsilon}
\def\De{\Delta}
\def\k{\kappa}
\def\l{\lambda}

\def\th{\theta}
\def\TH{\Theta}

\def\vf{\varphi}

\def\n{\nu}
\def\s{\sigma}
\def\t{\tau}

\newcommand{\Mf}{\mathfrak{M}}

\newcommand{\calP}{\mathcal{P}}
\newcommand{\calM}{\mathcal{M}}

\newcommand{\calC}{\mathcal{C}}
\newcommand{\calE}{\mathcal{E}}
\newcommand{\calS}{\mathcal{S}}

\newcommand{\FF}{\mathcal{F}}

\newcommand{\GL}{{\text {\rm GL}}}

\newcommand{\g}{\gamma}

\def\H{\mathbb{H}}

\renewcommand{\pmod}[1]{\  \,  \left( \mathrm{mod} \,  #1 \right)}

\allowdisplaybreaks

\newcommand{\xoverline}[1]{\mkern 1.5mu\overline{\mkern-1.5mu#1\mkern-1.5mu}\mkern 1.5mu}

\begin{document}

\title[Modularity of counting functions of convex polygons]{Modularity of counting functions of convex planar polygons with rationality conditions}
\author{Kathrin Bringmann}
\author{Jonas Kaszian}
\author{Jie Zhou}

\address{University of Cologne, Department of Mathematics and Computer Science, Weyertal 86-90, 50931 Cologne, Germany}
\email{kbringma@math.uni-koeln.de}
\email{jkaszian@math.uni-koeln.de}

\address{Jie Zhou, Yau Mathematical Sciences Center, Tsinghua University, Beijing 100084, P.\,R. China}
\email{jzhou2018@mail.tsinghua.edu.cn}


\makeatletter
\@namedef{subjclassname@2020}{%
 \textup{2020} Mathematics Subject Classification}
\makeatother

\maketitle

\begin{abstract}
	We study counting functions of planar polygons arising from homological mirror symmetry of elliptic curves. We first analyze the signature and rationality of the quadratic forms corresponding to the signed areas of planar polygons. Then we prove the convergence, meromorphicity, and mock modularity of the counting functions of convex planar polygons satisfying certain rationality conditions on the quadratic forms.
\end{abstract}

\section{Introduction and statement of results}

Since the work of \cite{Pol:2005, Pol:20011} there has been growing interest in indefinite theta functions arising from generating functions of genus zero open Gromov--Witten invariants and homological mirror symmetry of elliptic curves and of related manifolds. These generating functions are concrete objects due to the simplicity of elliptic curves: they are simply counting functions of planar polygons. Such a counting function of planar polygons with $N$ sides and  fixed angles has the form
\begin{equation*}
	\Theta_{Q,\chi}(\bb{z};\tau)= \sum_{\bb{n}\in\mathbb{Z}^{N-2}} \chi\left(\bb{n}+\frac{\bb{y}}{v}\right) q^{Q (\bb{n})}e^{2\pi i B\left (\bb{n},\,\bb{z}\right)}\,,~ \tau=u+iv\in \mathbb{H}\,,~\bm{z}=\bm{x}+i\bm{y}\in\C^{N-2}\,,~q:=e^{2\pi i \tau}.
\end{equation*}
Geometrically, $\bm n+\frac{\bm y}{v}$ corresponds to the side lengths of the polygons being enumerated, $Q(\bm n)$ arises from its signed area, and $B(\bm n,\bm z)$ from its weighted perimeter. The term $\chi(\bm n+\frac{\bm y}{v})$ is determined by the Euler characteristic of the moduli space of genus zero holomorphic disks in the elliptic curve. Its support only contains points in $\Z^{N-2}$ that give rise to polygons with non-negative areas instead of the full lattice. It is expected from homological mirror symmetry \cite{Kon:1994} that these generating functions exhibit modular properties. In a series of works \cite{BKR, BKZ, BRZ, Lau:2015, Pol:20012, Pol:20013, Pol:2005, Pol:2000, Pol:20011} it was proved that this is indeed the case for the enumeration of certain types of planar $N$-gons for $N\le6$. In this paper, we systematically extend the previous studies of individual cases to a large class of counting functions, namely those of convex planar polygons. Notably, the previous studies used rationality satisfied by quadratic forms and defining normal vectors. Since such rationality is difficult to access for the more general geometric objects, so complications arise in our setting.

For positive definite rational quadratic forms $Q$, such generating functions (with $\chi=1$) are (positive definite) theta functions, thus modular forms of weight $\frac N2-1$ (or Jacobi forms if one includes the elliptic variable $\bm z$). Since $q^{Q(\bm n)}$ produces exploding terms for indefinite quadratic forms, one can restrict the summation to a suitable cone to obtain a convergent series, which usually breaks modularity. Through Zwegers' thesis \cite{Zw} and later extensions to arbitrary signature  \cite{ABMP,FK,Na}, 
modularity properties can be recovered by adding a real-analytic function yielding a ``completion'' that is modular of weight $\frac N2-1$ and whose non-holomorphic differential properties depend on the signature of the quadratic form. More explicitly, we write the indicator function of the cone as a linear combination of expressions of the form $\prod_{j=1}^N\sgn(B(\bm{c_j},\bm a))$, where $C=(\bm{c_1},\dots,\bm{c_N})$ are the normal vectors of the hyperplanes bounding the cone. To obtain the completion, one exchanges each sign product by a smoothed replacement $E_{Q,C}$, called the generalized error function (see \eqref{DefGenError}). Note that $E_{Q,C}(\bm a)$ approximates the sign product with suitable exponential accuracy as $||\bm a||\to\infty$, where $||\cdot||$ is the Euclidean norm, to ensure convergence of the resulting theta function despite the growth of $q^{Q(\bm a)}$ in regions with $Q(\bm a)<0$. Furthermore, the generalized error function satisfies Vign\'eras' differential equation in Theorem \ref{Vigneras} of \cite{Vi} yielding modularity. Since the generalized error functions used in the completion contain a factor $\sqrt v$, the completion is usually not holomorphic. However, applying non-holomorphic derivatives simplifies the differential properties leading to the concept of {\it depth}. Following Zagier--Zwegers, recursively define the space $\Mf_k^d$ of mock modular forms of weight $k$ and depth $d$, and the space of their modular completions $\wh\Mf_k^d$, where mock modular forms of depth zero are ordinary modular forms ($\wh\Mf_k^0=\Mf_k^0=\Mf_k$), as follows: 
\begin{definition}
	A {\it mock modular form} $h:\H\to\C$ of {\it weight} $k$ and {\it depth} $d\ge1$ is the ``holomorphic part'' (see Theorem \ref{thm:Modularitysimplicialwithnoisotropicray}) of a real-analytic function $\wh h$, called {\it completion}, that transforms like a weight $k$ modular form and satisfies
	\begin{equation*}
		\frac{\del   }{\del \bar{\tau}}\wh{h} \in \bigoplus_j  v^{r_j} \wh{\mathfrak{M}}^{d-1}_{k + r_j}  \otimes \overline{\mathfrak{M}_{2+r_j}}.
	\end{equation*}
\end{definition}

The presence of isotropic vectors in the summed cone of the counting functions, which geometrically correspond to degenerations of polygons into line segments,
leads to more interesting phenomena.
A detailed analysis of the growth behaviour of the quadratic form
close to such points is thus included to determine convergence and meromorphicity of  the counting functions. Furthermore, new techniques on
generalized error functions for negative semi-definite quadratic spaces
are introduced for the study of modularity.
Applying these results and the limiting techniques, we obtain the following modularity result.

\begin{theorem}\label{ModularityOfCoutingFunction}
	Assume the positivity condition \eqref{eqnpositivityonangles} and the rationality condition \eqref{eqnrelaxedconditionintermsangles}. The regularized counting function $\TH_{Q,\chi_C^{\reg}}$, with $\chi_C^{\reg}$ given in  \eqref{RegularizedChi}, is a linear combination of mock modular forms of weight at most $\frac N2-1$ and depth at most $N-3$.
\end{theorem}
The following example demonstrates what the generating function and the corresponding modular completion look like.

\begin{example}\label{extrapzoidintro}
	The enumeration of isosceles trapezoids is given by  \cite{BKZ,Pol:2000}
	\[
		f(\bm z) := \sum_{\bm n\in\Z^2} \chi\left(\bm n+\frac{\bm y}{v}\right) q^{\frac12\left(n_1^2-n_2^2\right)} e^{2\pi i(n_1z_1-n_2z_2)},
	\]
	where
$
		\chi(\bm w) = \frac12(\sgn(w_1-w_2)+\sgn(w_2))
$,
	with $\sgn(x):=\frac{x}{|x|}$ for $x\ne0$, $\sgn(0):=0$.
In this case, we have 
\[
Q(\bb{n})= \frac12\left(n_1^2-n_2^2\right),\quad 
B(\bb{n}, \bb{z})=n_1z_1-n_2z_2,\quad \text{for $\bb{n}=(n_{1},n_{2})\in \mathbb{Z}^{2},$ $\bb{z}=(z_{1},z_{2})\in\mathbb{C}^{2}$}.
\]
	Assume that $z_1,z_2$ satisfy $0<y_1-y_2<v$, $0<y_2<v$. Denote $\bm{c_1}=(1,1)^T$, $\bm{c_2}=(0,1)^T$. Then the above generating function can be rewritten as
\begin{equation*}
	f(\bm{z}) = {1\over2} \sum_{\bb{n}\in\mathbb{Z}^2} \left(\sgn(B(\bm{c_1},\bb{n}))-\sgn(B(\bm{c_2},\bb{n}))\right) q^{{1\over2}\left(n_1^2-n_2^2\right)} e^{2\pi i(n_1z_1-n_2z_2)}
\end{equation*}
and is related to so-called Appell--Lerch functions. 
One then has $\chi_C^{\reg}=\chi$ according to \eqref{RegularizedChi}.
Using the results in \cite{Zw}, the modular completion of $	f(\bm{z})$ is 
\begin{equation*}
	\widehat{f}(\bm{z}) = {1\over2} \sum_{\bb{n}\in\mathbb{Z}^2} \left(\sgn(B(\bm{c_1},\bb{n}))-E\left({B(\bm{c_2},\bb{n})\over\sqrt{-Q(\bm{c_2})}} v^{1\over2}\right)\right) q^{{1\over2}\left(n_1^2-n_2^2\right)} e^{2\pi i(n_1z_1-n_2z_2)}\,,
\end{equation*}
where $E(x):=2\int_{0}^{x}e^{-\pi t^2}dt$ is the error function.
\end{example}

The paper is organized as follows. In Section \ref{secsignedarea}, we study properties of the signed area of planar polygons. We analyze the signature, positivity, and rationality of the corresponding quadratic form. In Section \ref{secdefinitionofcountingfunctions}, we prove convergence of the counting functions of convex planar polygons satisfying the positivity and rationality conditions given in Section \ref{secsignedarea}. This requires an examination of the structure of the isotropic vectors contained in the range of summation in the counting functions.  We also consider the correspondence between counting functions of degenerate polygons and restrictions of theta functions on faces of cones. In Section \ref{sectools}, we develop some tools for the study of modularity of indefinite theta functions for negative semi-definite quadratic spaces, such as generalized error functions. In Section \ref{secmodularity}, we prove modularity of the counting functions of convex planar polygons studied in Section \ref{secdefinitionofcountingfunctions}. In Appendix \ref{secappenedix}, we discuss deformations of polygons with fixed angles.

\section*{Acknowledgments}

We thank the referee for helpful comments. This research was supported by the Deutsche Forschungsgemeinschaft (German Research Foundation) under the Collaborative Research Centre / Transregio (CRC/TRR 191) on Symplectic Structures in Geometry, Algebra, and Dynamics. This result is part of a project that has received funding from the European Research Council (ERC) under the European Union's Horizon 2020 research and innovation programme (Grant agreement No. 101001179). Part of J. Zhou's work was done while he was a postdoc at the Mathematical Institute of University of Cologne and was supported by German Research Foundation Grant CRC/TRR 191. J. Zhou was also supported by the national key research and development program of China (No. 2022YFA1007100), a start-up grant at Tsinghua University, the Young overseas high-level talents introduction plan of China.

\section{Signed area of planar polygons}
\label{secsignedarea}

\subsection{Elementary properties of the signed area of planar polygons}

For  $N\geq 3$ consider $N$-gons with non-zero angles $\theta_{j}$ and non-zero side lengths $\lambda_j$ ($1\leq j\leq N$).
The vertices and the correspondingly sides are labelled clockwisely and cyclically $\pmod{N}$.
The $k$-th side vector has length $\lambda_{k}$ and the angle at the vertex labelled by $k$ from the $(k-1)$-th side vector to the $k$-th side vector is $\theta_{k}\in (-\pi,0)\cup (0,\pi)$;
see Figure \ref{figure:orientedpolygon} below. The angles satisfy $\sum_{k=1}^{N}\theta_k=2\pi$ and, since the $N$-gon is closed, we have with $\varphi_k:=\sum_{a=1}^{k}\th_a$,
\begin{align}\label{eqnrankcondition}
\sum_{k=1}^{N} \lambda_k e^{-i\varphi_k} =0.
\end{align}
\begin{figure}[h]
	\centering
	\includegraphics[scale=0.6]{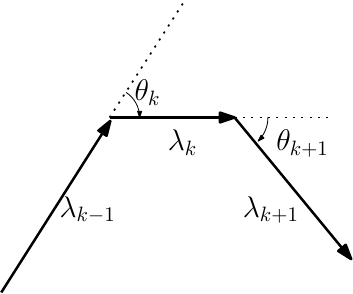}
	\caption{The outer angle at the vertex labelled by $k$ of the polygon is $\theta_{k}$, for $1\leq k \leq N$.}
	\label{figure:orientedpolygon}
\end{figure}

\begin{definition}
	An $N$-gon is {\it simple} if no non-consecutive edges intersect. Otherwise it is {\it complex}.
\end{definition}

If a simple $N$-gon further satisfies $\th_k\in(0,\pi)$ for $1\le k\le N$, then it is convex. If $\th_k\in(-\pi,0)$ for some $k$, then the $N$-gon is concave at the vertex labelled by $k$. Note that an $N$-gon satisfying $\l_k=0$ for some $k$ is in fact an $(N-1)$-gon. We study the signed area of simple polygons defined by using the cross product; see Appendix \ref{secappenedix} for discussions on complex polygons.
\begin{figure}[h]
	\centering
	\includegraphics[scale=0.6]{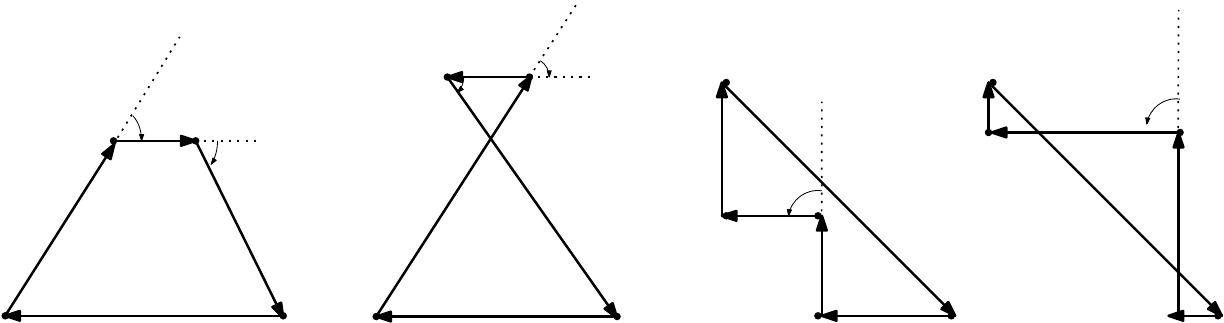}
	\caption{From left to right: the polygons are simple and convex, complex, simple, complex.}
	\label{figure:simpleandcomplexpolygons}
\end{figure}

\begin{lemma}\label{la:quadForm}
	The signed area of the simple $N$-gon is given by ($\bm\l=(\l_1,\dots,\l_{N-2})^T$)
	\begin{align*}
		Q(\bm{\lambda})=-\frac12\sum_{1\leq j\leq N-2} \hspace{-0.2cm} { \sin (\varphi_{N}-\varphi_{j}  ) \sin (\varphi_{N-1}-\varphi_{j})
		\over \sin (\theta_{N})}
		\lambda_{j}^2
		-\hspace{-0.3cm}\sum_{1 \leq  j<k\leq N-2}\hspace{-0.3cm}
		{ \sin (\varphi_{N}-\varphi_{j}  ) \sin (\varphi_{N-1}-\varphi_{k})
		\over \sin (\theta_{N})}
		\lambda_{j} \lambda_{k}.
	\end{align*}
\end{lemma}

\begin{proof}
	One connects the vertex labelled by $1$ with the rest of the vertices to cut the $N$-gon into $N-2$ triangles. Computing the signed area of each triangle using the cross product and summing yields
	\begin{equation}\label{eqnarea1}
		-\frac12\sum_{k=1}^{N-2} \left(\sum_{j=1}^k \l_je^{-i\vf_j}\right) \left(\l_{k+1}e^{-i\vf_{k+1}}\right) = \frac12\hspace{-.12cm}\sum_{1\le j<k\le N-1}\hspace{-.12cm} (\sin(\vf_k)\cos(\vf_j)-\sin(\vf_j)\cos(\vf_k))\l_j\l_k.
	\end{equation}
	Simplifying this gives an alternate form for the signed area 
	\begin{equation}\label{eqnareaform}
		\frac12\sum_{1\le j<k\le N-1} \sin(\vf_k-\vf_j)\l_j\l_k.
	\end{equation}
	Furthermore, one can easily solve for $\lambda_{N-1},\lambda_{N}$ by considering the real and imaginary parts of \eqref{eqnrankcondition}
	\begin{align}\label{eqnlambdaN-1}
		\lambda_{N-1}=  -\sum_{j=1}^{N-2}{ \sin (\varphi_{N}-\varphi_{j}  )\over \sin (\theta_{N})}\lambda_{j}\,,\quad \lambda_{N}=\sum_{j=1}^{N-2}  {\sin (\varphi_{N-1}-\varphi_{j} ) \over \sin (\theta_{N})}\lambda_{j}.
	\end{align}
	We split off the terms with $k=N-1$ in \eqref{eqnarea1} and plug in $\lambda_{N-1}$ and $\varphi_{N-1}=2\pi -\theta_{N}$ to obtain
	\begin{multline*}
		Q(\bm\l) =	-\sum_{1\le j<k\le N-2} \sin(\vf_j)\cos(\vf_k)\l_j\l_k - \sum_{1\le j<k\le N-2} \sin(\vf_k)\sin(\vf_j)\frac{\cos(\th_N)}{\sin(\th_N)}\l_j\l_k\\
		- \frac12\sum_{1\le j\le N-2} \frac{\sin(\vf_j)}{\sin(\th_N)}\sin(\vf_j+\th_N)\l_j^2.
	\end{multline*}
	Combining the first two sums, using a trigonometric identity, gives the claim.
\end{proof}

We next extend the signed area formula in Lemma \ref{la:quadForm} to $\l_j\in\mathbb{R}$ ($1\le j\le N-2$) and denote the resulting quadratic form by $Q$. For simplicity, we employ the following definitions:
\begin{align}\nonumber
	\bm{v}&:= \left(-\frac{\sin(\vf_N-\vf_j)}{\sin(\th_N)}\right)_{1\leq j\leq N-2},\quad
	\bm{w} \hspace{-.1cm}:= \left(\frac{\sin(\vf_{N-1}-\vf_j)}{\sin(\th_N)}\right)_{1\leq j\leq N-2}\,,
	\label{eqndfnofNlambda}\\
	\quad \l_{N-1}& := \bm{v}^T\cdot\bm\l\,,\quad  \hspace{33mm}\l_N := \bm{w}^T\cdot\bm\l.
\end{align}
We also use the standard basis vectors $\bm{e}_{\bm j}:=(\delta_{j,1},\delta_{j,2},\dots, \delta_{j,N-2})^T$ for $1\leq j\leq N-2$. The symmetric matrix $A=(A_{jk})_{1\leq j,k\leq N-2}$ corresponding to $Q$ is given by
\begin{equation}\label{eqndefinitionofA}
	A_{jk}:=\sin (\theta_{N}) v_{\min(j,k)} w_{\max(j,k)}.
\end{equation}

The following results, which follow from geometric intuition (that the signed area  is independent of the choice of the reference vertex), can be proved directly using Lemma \ref{la:quadForm} and equation \eqref{eqnrankcondition}.

\begin{lemma}\label{lemcyclicsymmetry}
	\ \begin{enumerate}[leftmargin=*,label=\rm{(\arabic*)}]
		\item Any $N-2$ consecutive parameters (enumerated cyclically) from $\l_k$ $(1\le k\le N)$ provide a set of independent parameters for the signed area.

		\item The signed area formula is invariant under cyclic permutations on the labellings $\{1,2,\dots,N\}$, i.e., for a polygon with side lengths $(\lambda_1,\lambda_2,\dots, \lambda_N)$ and a cyclic permutation $\sigma \in S_{N}$ we have
		\begin{align*}
			Q(\bm{\lambda})=&-\frac12\sum_{1\leq j\leq N-2} { \sin \left(\varphi_{\sigma(N)}-\varphi_{\sigma(j)}  \right) \sin \left(\varphi_{\sigma(N-1)}-\varphi_{\sigma(j)}\right)
			\over \sin (\varphi_{\sigma(N)}-\varphi_{\sigma(N-1)})}
			\lambda_{\sigma(j)}^2
			\\&-\sum_{1 \leq  j<k\leq N-2}
			{ \sin \left(\varphi_{\sigma(N)}-\varphi_{\sigma(j)}  \right) \sin \left(\varphi_{\sigma(N-1)}-\varphi_{\sigma(k)}\right)
			\over \sin \left(\varphi_{\sigma(N)}-\varphi_{\sigma(N-1)}\right)}
			\lambda_{\sigma(j)} \lambda_{\sigma(k)}.
		\end{align*}
	\end{enumerate}
\end{lemma}

\begin{remark}\label{remdependenceofparameters}
	While any $N-2$ consecutive side lengths can be chosen to provide independent parameters for the signed area, this is not necessarily true for non-consecutive ones as can be seen by considering the coefficients of $\lambda_{N-1},\lambda_N$ in terms of $\lambda_j\,(1\leq j\leq N-2)$. A direct calculation shows that the dependence occurs exactly if there exist $1\leq j,k\leq N$ such that $\sin(\varphi_k-\varphi_j)=0$. Geometrically this means that the $j$-th side of the polygon is parallel to the $k$-th side.
\end{remark}

\subsection{Non-negativity of signed area functions}

We next determine conditions that ensure a positive area or equivalently, positive values of the quadratic form $Q$. While simple polygons have a positive area, checking simplicity in terms of the angles and side lengths is not convenient. In this paper, we only consider polygons satisfying the following positivity condition
\begin{equation}\label{eqnpositivityonangles}
	\th_k\in(0,\pi)\quad \text{for}\quad  1\leq k \leq N.
\end{equation}
We obtain the following result, using the approach of Lemma \ref{la:quadForm} and induction.

\begin{corollary}\label{corpositivityofarea}
	
Assume \eqref{eqnpositivityonangles} and either $\lambda_{k}>0$ or
$\lambda_{k}<0$ for all  $ 1\leq k \leq N$. Then $Q(\bm{\lambda})>0$.
\end{corollary}

\begin{remark}\label{rempositivietyofarea}
	\ \begin{enumerate}[label=\textnormal{(\arabic*)}]
		\item We use \eqref{eqnpositivityonangles} instead of the more common convexity, since we consider deformed polygons with fixed angles but arbitrary values of side lengths $\l_k$. A deformed polygon is not necessarily convex. For example, the second polygon in Figure \ref{figure:simpleandcomplexpolygons} is a deformation of the first one, it satisfies \eqref{eqnpositivityonangles} with some negative side length and is not convex; see Appendix \ref{secappenedix} for further discussions.
		
		\item Without \eqref{eqnpositivityonangles}, the positivity of the side lengths $\l_k$ alone does not imply the positivity of the signed area; see the fourth polygon in Figure \ref{figure:simpleandcomplexpolygons} as an example.
	\end{enumerate}
\end{remark}
	
By Corollary \ref{corpositivityofarea} and continuity, $Q(\bm{\lambda})\ge0$ if $\l_k\ge0$ or $\l_k\le0$ for all $1\le k\le N$. 

\subsection{Integrality of signed area functions}
\label{secintegralityofsignedarea}

Next we consider some integrality conditions that we impose on the polygons to ensure convergence and ellipticity of the resulting counting functions. For this we extend the definition of $v_{k}, w_{k}$ to $k\in \{N-1,N\}$ according to \eqref{eqndfnofNlambda} to define
\[
	\calM:= \begin{pmatrix}
	v_{1} &\cdots& v_{N-2}& -1 & 0\\
	w_{1} &\cdots& w_{N-2} & 0 & 1
	\end{pmatrix}
	.
\]
Then \eqref{eqnrankcondition} is equivalent to
\begin{align}\label{eqnsymmetriclinearrelation}
	\calM
	\begin{pmatrix}
	\lambda_{1} &
	\cdots &
	\lambda_{N}
	\end{pmatrix}^T=0.
\end{align}
By Lemma \ref{lemcyclicsymmetry} any two consecutive columns of $\calM$ are linearly independent.
In fact, using 
\begin{align}\label{eqnelementaryidentity}
	\sin(\alpha+\beta+\gamma)\sin(\gamma) +\sin(\alpha)\sin(\beta)=\sin(\alpha+\gamma)\sin(\beta+\gamma)\,,
\end{align}
we have
\begin{equation*}\label{eqnconsecutivedeterminantofM}
	\det  \begin{pmatrix}
	v_{k} & v_{k+1}\\
	w_{k} & w_{k+1}
	\end{pmatrix}
	={\sin (\theta_{k+1})\over \sin (\theta_{N})}.
\end{equation*}
Let
\begin{equation*}
	\calP = \mat{a_1&0\\0&a_2} \in \GL_2(\R),\qquad \calS=
	\begin{pmatrix}
	\varepsilon_{1} & 0 & \cdots &0\\
	0 & \varepsilon_{2} & \cdots &0\\
	&\cdots&&0\\
	0 & 0 &0& \varepsilon_{N}
	\end{pmatrix}\in \mathrm{GL}_{N}(\mathbb{R}).
\end{equation*}
Assume that $(\theta_1, \theta_2,\dots, \theta_N)$ are such that, with $\bm e_{\bm k}\, (1\leq k \leq N)$ the standard basis vectors in $\R^N$,
\begin{equation}\label{eqnconditionintermsofstandardform}
	\calP^{-1}\calM\calS
	\begin{pmatrix}
	\bm{e}_{\bm{k}}& \bm{e}_{\bm{k+1}}\end{pmatrix}
	\in \mathrm{SL}_{2}(\mathbb{Z})\,~
	\text{for all} \,1\leq k \leq N.
\end{equation}
Here again we use the cyclic ordering in the sub-indices with, for example, $\bm{e_{N+1}}:=\bm{e_1}$.

\begin{lemma}\label{lemconditiononangles}
	The following are equivalent:
	\begin{enumerate}[label=\rm{(\arabic*)}]
		\item There exist $\e_k\ne0\,(1\le k\le N)$ such that for any solution to \eqref{eqnsymmetriclinearrelation} with $\l_k=n_k\e_k$, $n_k\in\Z$ ($k\in\{\s(1),\s(2),\dots,\s(N-2)\}$) implies that $n_k\in\Z$ for $k\in\{\s(N-1),\s(N)\}$.
		
		\item There exist $\calP$ and $\calS$ such that \eqref{eqnconditionintermsofstandardform} holds.
		
		\item For $1\le k\le\frac N2$, set $b_{2k}:=\frac{\sin(\th_1)}{\sin(\th_2)}\frac{\sin(\th_3)}{\sin(\th_4)}\cdots \frac{\sin(\th_{2k-1})}{\sin(\th_{2k})}$. There exists $t\in\R\setminus\{0\}$ such that
		\begin{align}\label{eqnconditionconsistency}
			&t = b_{N-1}\quad  \text{if } N \equiv 1\pmod 2,\qquad b_{N} = 1\quad  \text{if } N \equiv 0\pmod 2\,,\\
			\nonumber
			\text{and}\hspace{2cm} &\frac{1}{b_{2k}}\frac{\sin(\varphi_N-\varphi_{2k+1})}{\sin(\th_{2k+1})}\in \mathbb{Z}\,,\quad \frac{b_{2k}}{t}\frac{\sin(\varphi_N-\varphi_{2k})}{ \sin(\th_N)}\in \mathbb{Z}\,,\\
			\label{eqnconditionintegralentries}
			&-\frac{1}{b_{2k}}\frac{t\sin(\varphi_{N-1}-\varphi_{2k+1})}{\sin(\th_{2k+1})}\in \mathbb{Z}\,,\quad -b_{2k}\frac{\sin(\varphi_{N-1}-\varphi_{2k})}{\sin(\th_N)}\in \mathbb{Z}.
		\end{align}
	\end{enumerate}
\end{lemma}

\begin{proof}
	Comparing \eqref{eqnsymmetriclinearrelation}  with \eqref{eqnconditionintermsofstandardform}, it is clear that $(1)$ and $(2)$ are equivalent. Note that if the image of, say, $\lambda_{\sigma(N-1)}$ is not an affine transformation of $\mathbb{Z}$, then it is dense in $\mathbb{R}$.
	
	To prove that $(2)$ implies $(3)$, note that  $a_1=\pm\e_{N-1}$ and $a_2=\mp \e_N$. For notational simplicity, we choose $a_1=-\e_{N-1}$ and $a_2=\e_N$. We let
	\begin{align}\label{eqnvariablesforintegralitycondition}
		x_{k}:=-{\varepsilon_{k}\over \varepsilon_{N-1} }\quad(1\leq k\leq N-2)\,,\quad\quad t:=-{\varepsilon_{N-1}\over \varepsilon_{N} }.
	\end{align}
	The integrality in \eqref{eqnconditionintermsofstandardform} is equivalent to $x_{k}v_{k}\,, t x_{k}w_{k}\in \mathbb{Z}$ for $ 1\leq k \leq N.$ The unimodularity in \eqref{eqnconditionintermsofstandardform} gives $x_kx_{k+1} t \frac{\sin(\theta_{k+1})}{\sin(\theta_N)}=1$, which we recursively solve to obtain
	\[
		x_{2k+1}={x_1\over b_{2k}}{\sin (\theta_{1})\over \sin (\theta_{2k+1})}\,,\quad x_{2k}={b_{2k}\over t x_{1}}{\sin (\theta_{N})\over \sin (\theta_{1})},\,\quad x_1v_1 = tx_{N-2}w_{N-2}=-1.
	\]
	Inserting $x_1=-\frac{1}{v_1}$, we obtain
	\begin{equation*}
		x_{2k+1} u_{2k+1}= -{1\over b_{2k}}{\sin (\theta_{1})\over \sin (\theta_{2k+1})v_{1}}u_{2k+1}\,, \quad x_{2k} u_{2k}=- {v_{1} b_{2k}\over t }{\sin (\theta_{N})\over \sin (\theta_{1})}u_{2k}\,,\quad \bm{u}\in \{\bm{v}, t\bm{w}\}.
	\end{equation*}
	Applying $tw_{N-2}x_{N-2}=-1$ gives  \eqref{eqnconditionconsistency}.
	Plugging in the above expressions for $x_{2k+1},x_{2k}$, the entries of the matrix $\U=\calP^{-1}\calM\calS$ simplify as
	\begin{align*}\label{eqnconditionintermssolutions}
		\U_{1,2k+1}&=\frac{1}{b_{2k}}\frac{\sin(\varphi_N-\varphi_{2k+1})}{\sin(\th_{2k+1})}\,,\quad \U_{1,2k}=\frac{b_{2k}}{t}\frac{\sin(\varphi_N-\varphi_{2k})}{ \sin(\th_N)}\,, \\
		\U_{2,2k+1}&=-\frac{1}{b_{2k}}\frac{t\sin(\varphi_{N-1}-\varphi_{2k+1})}{\sin(\th_{2k+1})}\,,\quad \U_{2,2k}=-b_{2k}\frac{\sin(\varphi_{N-1}-\varphi_{2k})}{\sin(\th_N)}.\notag
	\end{align*}
	The integrality condition in \eqref{eqnconditionintermsofstandardform}  then gives \eqref{eqnconditionintegralentries}.
	The above argument that $(2)$ implies that $(3)$ can be reversed, yielding that $(3)$ implies $(2)$.
\end{proof}

\begin{example}\label{exconditiononangles}
	For $N=3$, \eqref{eqnconditionconsistency} and \eqref{eqnconditionintegralentries} are trivially true. For $N=4$, one directly checks that these conditions are satisfied as long as the consistency constraint $b_4=\frac{\sin(\th_1)\sin(\th_3)}{\sin(\th_2)\sin(\th_4)}=1$ holds. By \eqref{eqnelementaryidentity}, this becomes $\sin(\th_2+\th_3)\sin(\th_3+\th_4)=0$.
	It is satisfied if and only if the $4$-gon has parallel sides.
	For $N=5$, the conditions become
	\[
		\frac{\sin(\th_k+\th_{k+1})\sin(\th_{k+1}+\th_{k+2})}{\sin(\th_k)\sin(\th_{k+2})} \in \Z\qquad \text{for } \,1\leq k \leq 5.
	\]
	For example, this is satisfied if $\th_1+\th_2=\pi=\th_2+\th_3$.
\end{example}

It is sometimes necessary to work with the fixed set of independent parameters $\l_k$ ($1\le k\le N-2$). This naturally leads to the following more relaxed condition  than \eqref{eqnconditionintermsofstandardform}, namely
\begin{equation*}
	\U = \calP^{-1}\calM\calS \in \Mat_{2\times N}(\Z).
\end{equation*}
That the above equation is solvable means geometrically that any polygon with $\bm\l=(\l_1,\dots,\l_{N-2})$, where $\l_k\in\Z\e_k,\, 1\leq k\leq N-2$, satisfies $\l_k\in\Z\e_k,k\in\{N-1,N\}$. In fact, recall from \eqref{eqnsymmetriclinearrelation} that
\begin{equation}\label{eqnMSintegralentries}
	\mat{\frac{\l_{N-1}}{\e_{N-1}}\\\frac{\l_{N-2}}{\e_N}} = \mat{\frac{v_1}{\e_{N-1}}&\cdots&\frac{v_{N-2}}{\e_{N-1}}\\ \frac{w_1}{\e_N}&\cdots&\frac{w_{N-2}}{\e_N}} \mat{\l_1\\\vdots\\\l_{N-2}} = \mat{\frac{v_1\e_1}{\e_{N-1}}&\cdots&\frac{v_{N-2}\e_{N-2}}{\e_{N-1}}\\ \frac{w_1\e_1}{\e_N}&\cdots&\frac{w_{N-2}\e_{N-2}}{\e_N}} \mat{n_1\\\vdots\\n_{N-2}}.
\end{equation}
The relaxed condition means that the $2\times(N-2)$ matrix in the last expression has integral entries.

Before proceeding we need the following.

\begin{definition}
	A matrix with integral entries is {\it primitive} if its rows and columns are primitive vectors.
\end{definition}

\begin{lemma}\label{lemrelaxedcondition}
	Let
	\[
		M= \begin{pmatrix}
		v_{1} &\cdots& v_{N-2}\\
		w_{1} &\cdots& w_{N-2}
		\end{pmatrix}\,,\quad
		S=
		\begin{pmatrix}
		\varepsilon_{1} & 0 & \cdots &0\\
		0 & \varepsilon_{2} & \cdots &0\\
		&\cdots&&0\\
		0 & 0 &0& \varepsilon_{N-2}
		\end{pmatrix}.
	\]
	Then the following are equivalent:
	\begin{enumerate}[label=\rm{(\arabic*)}]
		\item We have
		\begin{align}\label{eqnrelaxedconditionintermsangles}
			\frac{\sin(\vf_N-\vf_k)}{\sin(\vf_N-\vf_j)}\frac{\sin(\vf_{N-1}-\vf_j)}{\sin(\vf_{N-1}-\vf_k)} \in \Q \cup \{\infty\} \quad\text{for all}\quad 1 \le j,k \le N-2.
		\end{align}
		
		\item There exist $\calP,\calS$ such that
		\begin{equation}\label{eqnrelaxedconditionintermsofstandardform}
			\U = \calP^{-1}\calM\calS \in \Mat_{2\times N}(\Z).
		\end{equation}

		\item There exist $\calP,S$ such that
		\begin{align}\label{eqnrelaxedprimitiveconditionintermsofstandardform}
			U=\calP^{-1}M{S}\in  \mathrm{Mat}_{2\times (N-2)}(\mathbb{Z})\,~\text{is primitive}\,,\quad S^{T}A S\in \mathrm{Mat}_{(N-2)\times (N-2)}(\mathbb{Z}).
		\end{align}
		Here the matrix $A$ corresponds to the quadratic form $Q$ given in \eqref{eqndefinitionofA}.
	\end{enumerate}
\end{lemma}

\begin{proof}
	Note that \eqref{eqnrelaxedconditionintermsangles} is equivalent to $\frac{w_k}{v_k}(\frac{w_j}{v_j})^{-1}\in\Q\cup\{\infty\}$, for all $1\le j,k\le N-2$. Clearly $(1)$ and $(2)$ are equivalent. From the definitions of $M,S,\calM,\calS$, and \eqref{eqnMSintegralentries}, $(3)$ implies $(2)$. To prove that $(2)$ implies $(3)$, note that \eqref{eqnrelaxedconditionintermsofstandardform} is an equation for the variables in \eqref{eqnvariablesforintegralitycondition} and $\e_{N-1}$. For a solution to \eqref{eqnrelaxedconditionintermsofstandardform}, one can scale the set of variables in  \eqref{eqnvariablesforintegralitycondition} by integers suitably to meet \eqref{eqnrelaxedprimitiveconditionintermsofstandardform}, determining them up to sign. In particular, for such a solution, inspecting the entries of $\U$
	\[
		{ v_{j} \varepsilon_{j} \over \varepsilon_{N-1}}   { w_{k} \varepsilon_{k}\over \varepsilon_{N}} = v_{j} x_{j}\cdot   tw_{k} x_{k}\in \mathbb{Z}\,,\quad 1\leq j,k\leq N-2.
	\]
	Recall the definition of $A$ given in \eqref{eqndefinitionofA}. Then we have
	\begin{equation*}
		\left(S^TAS\right)_{jk} = \e_jA_{jk}\e_k =
		\begin{cases}
			\sin(\th_N)v_j\e_jw_k\e_k=\sin(\th_N)\frac{\e_{N-1}^2}{t}v_jx_j\cdot tw_kx_k &\text{if } 1\leq j\le k\le N-2,\\
			\sin(\th_N)w_j\e_jv_k\e_k=\sin(\th_N)\frac{\e_{N-1}^2}{t}v_kx_k\cdot tw_jx_j &\text{if } 1\leq k\le j\le N-2.
		\end{cases}
	\end{equation*}
	We take a solution to \eqref{eqnrelaxedconditionintermsofstandardform} that satisfies $t>0$ and set $\e_{N-1}=(\frac{t}{\sin(\th_N)})^\frac12$, proving the claim.
\end{proof}

\begin{remark}\label{remasymmetry}
	If one exchanges the parameters with another set of $N-2$ consecutive parameters as independent ones with the same set of numbers $\{\e_k\}_{1\le k\le N}$, then, by \eqref{eqnsymmetriclinearrelation}, the integrality of the $N-2$ numbers $\l_k$ only implies rationality of the remaining two. Equivalently, the same set of polygons determined by $\bm\l\in S\Z^{N-2}$ is described by a scaled sub-lattice of $\Z^{N-2}$ (instead of a full one) in terms of other $N-2$ consecutive parameters. Thus the cyclic symmetry (such as the one exhibited in Lemma \ref{lemcyclicsymmetry}) among the indices $1\leq k\leq N$ still exists, as long as only rationality structure is concerned. That the cyclic symmetry respects rationality is argued above from the  point of view of linear transformation via \eqref{eqnrelaxedconditionintermsofstandardform}, but it can also be checked using \eqref{eqnrelaxedconditionintermsangles} and \eqref{eqnelementaryidentity}.
\end{remark}

Hereafter, we assume that \eqref{eqnrelaxedconditionintermsangles} holds. For definiteness, we fix a solution $\{\e_k\}_{1\le k\le N}$ satisfying \eqref{eqnrelaxedprimitiveconditionintermsofstandardform}. We also omit the notations for $\e_k$ if it is clear from the context.

\section{Convergence of counting functions of convex planar polygons}\label{secdefinitionofcountingfunctions}

\subsection{Counting functions of convex polygons} 

Following \cite{BKZ,Pol:2005,Pol:20011}, we define
\begin{equation}\label{eqncountingfunctionwithparameterlambda}
	\TH_{Q,\chi}(\bm z;\t) := q^{-Q(\bm\a)}e^{-2\pi iB(\bm\a,\bm\b)}\sum_{\bm n\in\Z^{N-2}+\bm\a} \psi(\bm n)(\sgn(n_{N-2}))^{N-3}q^{Q(\bm n)}e^{2\pi iB(\bm\b,\bm n)}.
\end{equation}
Here $\bb{z}=:\bb{\alpha}\tau+\bb{\beta}$ with $\bb{\alpha},\bb{\beta}\in \R^{N-2}$, $B(\bm{a},\bm{b}):=Q(\bm{a}+\bm{b})-Q(\bm{a})-Q(\bm{b})$ is the bilinear form associated to $Q$, $\sgn(x):=\frac{x}{|x|}$ for $x\ne0$, $\sgn(0):=0$, 
and $\psi$ is the characteristic function
for 
\[
	\left\{\bm\l\in(0,\infty)^{N-2} : \l_{N-1},\l_N>0\right\} \cup \left\{\bm\l\in(-\infty,0)^{N-2} : \l_{N-1},\l_N<0\right\} \subseteq \R^{N-2}
\]
with $\l_{N-1}:=\bm v^T\cdot\bm\l$, $\l_N:=\bm w^T\cdot\bm\l$. For use below, we also define $z_{N-1}:=\bm v^T\cdot\bm z$, $z_N:=\bm w^T\cdot\bm z$. Geometrically, \eqref{eqncountingfunctionwithparameterlambda} represents the fact that the summation is over those polygons with fixed angles $\th_k$ ($1\leq k\leq N$) whose side lengths are related to the reference side lengths $\bm\a$ by shifts by integral multiples of some fixed vectors; see \cite{BKZ}, as well as Appendix \ref{secappenedix}, for detailed discussions. Define
\[
	\chi(\bm n) := \psi(\bm n)\sgn(n_{N-2})^{N-1}.
\]
In terms of $\l_k=n_k$, $1\le k\le N$, one has
\begin{align}\label{eqnchifunction}
	\chi\left(\bm\lambda\right) =\prod_{k=1}^{N} H(\lambda_{k})- (-1)^{N}\prod_{k=1}^{N} H(-\lambda_{k})\,,
\end{align}
with $H$ the {\it Heaviside function}, i.e.,
\begin{equation*}
	H(t) :=
	\begin{cases}
		1 & \text{if $t>0$},\\
		0 & \text{otherwise}.
	\end{cases}
\end{equation*}
We rewrite \eqref{eqncountingfunctionwithparameterlambda} 
\begin{equation*}\label{eqncountingfunctionwithparametern}
	\Theta_{Q,\chi}(\bb{z};\tau)=\sum_{\bb{n}\in\mathbb{Z}^{N-2}} \chi\left(\bb{n}+\frac{\bb{y}}{v}\right) q^{Q (\bb{n})}e^{2\pi i B\left (\bb{n},\,\bb{z}\right)}.
\end{equation*}
The series in \eqref{eqncountingfunctionwithparameterlambda} is divergent if the summation range contains $\bm n$ with $Q(\bm n)<0$. This does not occur if the positivity condition on angles $\th_k\in(0,\pi)$ $(1\le k\le N)$ is imposed. In fact, by \eqref{eqnchifunction} in the summation range either $\l_k\ge0$ or $\l_k\le0$ for all $1\le k\le N$. Then one applies Lemma \ref{la:quadForm} and Corollary \ref{corpositivityofarea}. In either case, the corresponding polygons being counted are actually convex.

\begin{remark}\label{remothersignatures}
	Without \eqref{eqnpositivityonangles}, $Q$ can be negative as explained in Remark \ref{rempositivietyofarea} {\rm(2)}. To avoid the resulting divergence of $\TH_{Q,\chi}$, one has to restrict the summation range in \eqref{eqncountingfunctionwithparameterlambda} to its intersection with the range  $\{\bm x\in\R^{N-2}:Q(\bm x)\ge0\}$. We will return to this topic in future investigations.
\end{remark}

\subsection{Structure of isotropic vectors}

The counting function of convex polygons is an indefinite theta function associated to the quadratic form $Q$. Define $\De:=\{\bm\l\in\R^{N-2}:\chi(\bm\l)\ne0\}$. Then $\xoverline\De$ is a polyhedron, which is a cone over a compact polytope. By Corollary \ref{corpositivityofarea}, one has $Q(\bm x)>0$ for $\bm x\in\De$. However, this does not imply convergence of $\TH_{Q,\chi}$, since $\De$ is non-compact and its closure might contain vectors in $\{\bm x\in\R^{N-2}:Q(\bm x)=0\}$. The discussion on the convergence of $\TH_{Q,\chi}$ hence requires the study of isotropic vectors contained in the closure $\xoverline\De$.

\begin{lemma}\label{la:quadFormsignature}
	If $\th_k\in(0,\pi)$ for $1\leq k\leq N$, then the quadratic form $Q$ has signature $(1,N-3)$.
\end{lemma}
	
\begin{proof}
	We prove the claim by induction. The cases $N\in\{3,4\}$ are clear by direct computations on the corresponding quadratic forms. Assume that the statement holds for $1\leq n\le N-1$ with $N-1\ge4$. We show that it then also holds for $n=N$. First it is easy to prove by contradiction  that there exists some $k$ such that $\th_k+\th_{k+1}<\pi$. By a relabelling of the vertices if needed, we can assume that $k=2$ for definiteness. Then we make the following  change of variables
	\begin{equation}\label{eqnchangeofsides}
		\mu_1 = \lambda_1+\frac{\sin(\th_3)}{\sin(\th_2+\th_3)} \lambda_2,\qquad
		\mu_3 = \lambda_3+\frac{\sin(\th_2)}{\sin(\th_2+\th_3)} \lambda_2,\qquad
		\mu_k = \lambda_k,\quad 4\le k\le N-2.
	\end{equation}
	Denote $\bm\mu=(\mu_1,\mu_3,\dots,\mu_{N-2})^T$. One can check by a straightforward computation that
	\begin{equation}\label{eqninduction}
		\bm\l^T A(\th_1,\th_2,\dots,\th_N)\bm\l = -\frac{\sin(\th_2)\sin(\th_3)}{2\sin(\th_2+\th_3)}\l_2^2 + \bm\mu^TA(\th_1,\th_2+\th_3,\th_4,\dots,\th_N)\bm\mu.
	\end{equation}
	In fact, recall from the proof of Lemma \ref{la:quadForm} that one has from \eqref{eqnareaform}
	\[
		\bm\l^TA(\th_1,\th_2,\dots,\th_N)\bm\l = \frac12\sum_{1\le j<k\le N-1} \sin(\vf_k-\vf_j)\l_j\l_k\,,
	\]
	where $\lambda_{N-1}$ is given by \eqref{eqnlambdaN-1}. The same computations yield
	\[
		\bm\mu^TA(\th_1,\th_2+\th_3,\th_4,\dots,\th_N)\bm\mu = \frac12\sum_{\substack{1\le j<k\le N-1\\j,k\ne 2}} \sin(\vf_k-\vf_j)\mu_j\mu_k\,,
	\]
	with $\mu_{N-1}$ similarly given by \eqref{eqnlambdaN-1} $\mu_{N-1}=\sum_{j\in\{1,\cdots, N-2\}\sm\{2\}}\frac{\sin(\vf_j)}{\sin(\th_N)}\mu_j$. Plugging \eqref{eqnchangeofsides} into the above expression for $\mu_{N-1}$, we have
	\begin{align*}
		\mu_{N-1} &= { \sin (\varphi_{1})  \over \sin (\theta_{N})}\lambda_{1} + {\sin (\varphi_1)\sin(\theta_3)+\sin (\varphi_3)\sin(\theta_2)\over  \sin(\theta_{N})\sin(\theta_{2}+\theta_{3})}\lambda_{2} + { \sin (\varphi_{3})  \over \sin (\theta_{N})}\lambda_{3}+\sum_{4\leq j\leq N-2}{ \sin (\varphi_{j})  \over \sin (\theta_{N})}\lambda_{j}.
	\end{align*}
	Applying \eqref{eqnelementaryidentity}, one has
	\begin{align*}
		\mu_{N-1} &= { \sin (\varphi_{1})  \over \sin (\theta_{N})}\lambda_{1} + {\sin (\varphi_2)\over  \sin(\theta_{N})}\lambda_{2} + { \sin (\varphi_{3})  \over \sin (\theta_{N})}\lambda_{3} + \sum_{4\leq j\leq N-2}{ \sin (\varphi_{j})  \over \sin (\theta_{N})}\lambda_{j}=\lambda_{N-1}.
	\end{align*}
	Expanding in terms of $\lambda_{j},1\leq j\leq N-1$, we obtain
	\begin{align*}
		&\bm\mu^TA(\th_1,\th_2+\th_3,\th_4,\dots,\th_N)\bm\mu\\
		&\hspace{1cm}= \frac12\sum_{\substack{1\le j<k\le N-1\\j,k\notin\{1,2,3\}}} \sin(\vf_k-\vf_j)\l_j\l_k + \frac12\sin(\vf_3-\vf_1)\l_1\l_3 + \frac12\sin(\th_2)\l_1\l_2 + \frac12\sin(\th_3)\l_2\l_3\\
		&\hspace{1cm}\quad+ \frac{\sin(\th_2)\sin(\th_3)}{2\sin(\th_2+\th_3)}\l_2^2 + \frac12\sum_{4\le k\le N-1} \sin(\vf_k-\vf_1)\l_1\l_k + \frac12\sum_{4\le k\le N-1} \sin(\vf_k-\vf_3)\l_3\l_k\\
		&\hspace{1cm}\quad+ \frac12\sum_{4\le k\le N-1} \sin(\vf_k-\vf_2)\l_2\l_k.
	\end{align*}
	It follows that 
	\begin{align*}
		\bm{\mu}^{T}A(\theta_{1},\theta_{2}+\theta_{3},\theta_{4},\dots, \theta_{N})\bm{\mu}&={\sin(\theta_{2}) \sin (\theta_{3}) \over 2\sin (\theta_{2}+\theta_{3})}\lambda_{2}^{2}+\frac12\sum_{1\leq j< k\leq N-1}\sin (\varphi_{k}-\varphi_{j})\lambda_{j}\lambda_{k} \\
		&={\sin(\theta_{2}) \sin (\theta_{3}) \over 2\sin (\theta_{2}+\theta_{3})}\lambda_{2}^{2}+\bm{ \lambda}^{T}A(\theta_{1},\theta_{2},\dots, \theta_{N})\bm{ \lambda}.
	\end{align*}
	This gives \eqref{eqninduction}. By the induction hypothesis, the second term on the right-hand side of \eqref{eqninduction} corresponds to a quadratic form of signature $(1,N-4)$. It follows that the corresponding quadratic form on the left-hand side of \eqref{eqninduction} has signature $(1,N-3)$. This completes the proof.
\end{proof}

\begin{remark}
	Geometrically, \eqref{eqnchangeofsides} corresponds to cutting a triangle from a convex $(N-1)$-gon to form a convex $N$-gon. For non-convex polygons, it is not always possible to find $k$ such that $\th_k+\th_{k+1}\in(-\pi,0)\cup(0,\pi)$; see e.g., the third polygon in Figure \ref{figure:simpleandcomplexpolygons}. Furthermore, the signature of the first term on the right-hand side of \eqref{eqninduction} can be $\pm1$. Starting from $3$-gons, instead of using the cutting procedure one can use additionally the gluing procedure to form non-convex $N$-gons; see e.g., the third and forth polygons in Figure \ref{figure:simpleandcomplexpolygons}. This then implies that for general simple $N$-gons with $N\ge4$, the signature of the corresponding quadratic form $Q$ can be $(N-2-s,s)$ with $0\le s\le N-3$.
\end{remark}

We next discuss whether under the positivity condition $\th_k\in(0,\pi)$ for $1\le k\le N$, $Q$ is strictly positive in 
$\xoverline{\Delta}$. Let $\bm u^\perp:=\{\bm x\in\R^{N-2}:B(\bm x,\bm u)=0\}$ for $\bm x\in\R^{N-2}$ and
\begin{align}\label{eqnnormalvectors}
	\bm{c}_{\bm j}:=A^{-1} \bm{e}_{\bm j}\,,\quad 1\leq j\leq N-2\,,\quad
	\bm{c}_{\bm {N-1}}:=A^{-1} \bm{v}\,,\quad
	\bm{c}_{\bm{N}}:=A^{-1} \bm{w}.
\end{align}
Then we have $\lambda_{k}=B(\bm{c}_{\bm k}, \bm{\lambda})$ for $1\leq k\leq N$. The following is a preparationary lemma.

\begin{lemma}\label{lemnoncoincidenceonhyperplanes}	
	Assume that $\th_k\in(0,\pi)$ for $1\leq k\leq N$. Then the $N$ hyperplanes $\bm{c_k}^\perp$ for $1\leq k\leq N$ are distinct, except if $N=3$ or $N=4$ with $\th_k+\th_{k+1}=\pi$ for some $1\leq k\leq 4$.
\end{lemma}

\begin{remark}\label{remexceptionalcases}
	The exceptional case $N=3$ in Lemma \ref{lemnoncoincidenceonhyperplanes} corresponds to triangles whose enumeration gives Jacobi theta functions \cite{Pol:20011}. The exceptional case $N=4$ corresponds to parallelograms and trapezoids. The enumeration of parallelograms gives the Zagier series \cite{Gottsche:1998,Pol:2000,Zagier}, while that of trapezoids and other convex $4$-gons yield the Appell--Lerch sums studied by Kronecker that describe sections of rank two vector bundles on the elliptic curve as shown in \cite{Pol:20012,Pol:20013,Pol:2005,Pol:2000}. 
\end{remark}

In what follows we only focus on the cases $N\ge5$. We next study the intersection of $\{\bm\l\in\R^{N-2}:Q(\bm\l)=0\}$ with $\xoverline\De$. For $I\subseteq\{1,2,\dots,N\}$, let
\begin{align}\label{eqndfnface}
	D_0 :=& \left\{\bm{\lambda}\in \R^{N-2}: Q(\bm{\lambda})=0\right\},\quad D_+:= \left\{\bm{\lambda}\in \R^{N-2}: Q(\bm{\lambda})>0\right\},
	\\ V_I :=& \bigcap_{j\in I} \left\{\bm{\lambda}\in \R^{N-2}: \lambda_j=0\right\}=\left\{\bm{\lambda}\in \R^{N-2}: \lambda_j=0 \text{ for } j\in I \right\}\,,\quad F_{I}=V_{I}\cap \xoverline{\Delta}.\notag
\end{align}

We have the following result regarding the structure of isotropic vectors in the closure $\xoverline{\Delta}$.

\begin{proposition}\label{propisotricconestructure}
	Assume \eqref{eqnpositivityonangles} and $N\ge5$. Then
	\begin{equation}\label{eqnisotropivectors}
		D_0\cap\xoverline{\Delta} = \bigcup_{\substack{I\subseteq\{1,2,\dots,N\}\\|I|=N-2}} D_0\cap F_I.
	\end{equation}
	Furthermore, on the right-hand side we have $D_0\cap F_I\ne \{\bm{0}\}$ if and only if $I=\{1,2,\dots,N\}\setminus\{a,b\}$ with $b\ne a\pm1$ such that $\sin(\varphi_b-\varphi_a)=0$. In particular, assume that $\sin(\varphi_k-\varphi_j)\ne0$ for any $1\leq j,k\leq N$. Then $Q$ is positive on $\xoverline{\Delta}\setminus\{\bm{0}\}$.
\end{proposition}

\begin{proof}
	From Lemma \ref{la:quadFormsignature}, by diagonalizing the matrix $A$ (which does not change $Q=0$ and maps hyperplanes to hyperplanes), the intersection of $D_0$ with a cone over a polytope is a set of lines through the origin. Then the intersection of $\xoverline{\Delta}$ with the set of isotropic vectors $D_0$ is
	\[
		D_0\cap\xoverline{\Delta} = \bigcup_{\substack{I\subseteq\{1,2,\dots,N\}\\\dim(F_I)=1}} D_0\cap F_I.
	\]
	Hence we only need to consider $D_0\cap F_I$ with $|I|=N-3$ to prove \eqref{eqnisotropivectors}. By using Lemmas \ref{lemcyclicsymmetry} and \ref{lemnoncoincidenceonhyperplanes}, we can assume that  $|I\cap\{1,2,\dots,{N-2}\}|=N-4$ and $|I\cap\{N-1,N\}|=1$.
	
	Let $1\leq a <b\leq N-2$ such that $\{a,b\}\cup I=\{1,2,\dots,{N-2}\}$.
	Then, for $\bm{\l}\in F_I$,
 	\[
		Q|_{F_I}(\bm{\lambda}) = - \frac{\sin(\varphi_a)\sin(\th_N+\varphi_a)}{2\sin(\th_N)} \l_a^2 -  \frac{\sin(\varphi_b)\sin(\th_N+\varphi_b)}{2\sin(\th_N)} \l_b^2 - \frac{\sin(\varphi_a)\sin(\th_N+\varphi_b)}{\sin(\th_N)} \l_a\l_b.
	\]
	Using \eqref{eqndfnofNlambda} and $|I\cap\{N-1,N\}|= 1$, one has
	\[
		Q|_{F_I}(\bm{\lambda}) = \frac12\sin(\th_N)\l_{N-1}\l_{N} + \frac12\sin(\varphi_b-\varphi_a) \l_a\l_b=\frac12\sin(\varphi_b-\varphi_a) \l_a\l_b.
	\]
	Thus, for $\bm{\l}\in D_0\cap F_I$, we have $\bm{\l}\in D_0\cap (F_{I\cup \{a\}}\cup F_{I\cup \{b\}})$ or  $\sin(\varphi_b-\varphi_a)=0$. For the second case, recall that by \eqref{eqnpositivityonangles} one has $0<\varphi_a<\varphi_b<2\pi$, so $\sin(\varphi_b-\varphi_a)=0$ implies $\varphi_b=\varphi_a+\pi$. Using the expressions for $\bm{v},\bm{w}$ in \eqref{eqndfnofNlambda}, this further gives that the vanishing of one of $\l_{N-1},\l_{N}$ implies the vanishing of the other (and $\l_a=\l_b$). Hence we have  $\bm{\l}\in D_0\cap F_{I\cup \{N-1,N\}}$ and in general
	\begin{align*}
		D_0\cap F_I\subseteq\bigcup_{\substack{I\subseteq\{1,2,\dots,N\}\\|I|=N-2}} F_I\cap D_0.
	\end{align*}

	Take any $I$ with $I={N-2}$, say $I=\{1,2,\dots,N\}\setminus\{a,b\}$ with $a<b$. If $b=a\pm1$, enumerated cyclically, then, by Lemma \ref{lemcyclicsymmetry}, one can take $\l_j,j\in I$ as independent parameters for the signed area. From \eqref{eqnlambdaN-1}, $\l_k=0$ for all $k$ and thus $F_I=\{\bm0\}$. Otherwise, one can take $\l_a,\l_{a+1},\dots,\l_{a+N-3}$, again enumerated cyclically, as independent parameters. Using \eqref{eqnareaform} one has
	\[
		Q|_{F_I}(\bm\l) = \frac12\sin(\vf_b-\vf_a)\l_a\l_b.
	\]
	For $\l_a=0$ or $\l_b=0$ we have $\bm{\l}\in F_{I\cup\{a\}}\cup F_{I\cup\{b\}}=\{\bm{0}\}$. Hence we only need to consider $\l_a\l_b\ne0$, which implies that $D_0\cap F_I\ne \{\bm{0}\}$ if and only if $\sin(\varphi_b-\varphi_a)=0$. This finishes the proof.
\end{proof}

\begin{remark}\label{remisotropicvectorasintersection}
	Figure \ref{figure:crosssectionpolyhedron} illustrates the structure of the isotropic vectors contained in $\xoverline\De$. In the proof of Proposition \ref{propisotricconestructure}, we show that $D_0\cap F_I\ne\{\bm0\},|I|=N-3$ requires $D_0\cap F_I=D_0\cap F_J$ for some $|J|={N-2}$. 
	Geometrically, the condition in Proposition \ref{propisotricconestructure} for the existence of non-zero isotropic vectors means that the convex polygons have parallel sides, and the vanishing of $Q$ occurs exactly if the parallel sides coincide.
	\begin{figure}[h]
		\centering
		\includegraphics[scale=0.6]{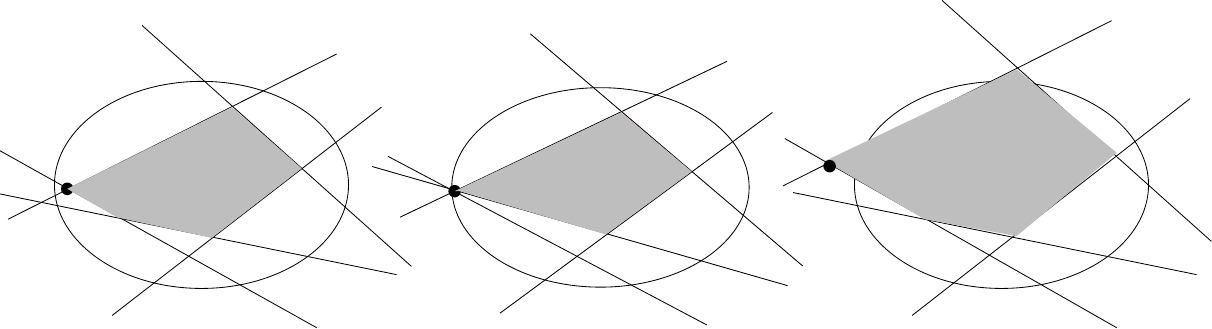}
		\caption{The shaded region is the cross section for $\xoverline{\Delta}$, the quadric is the cross section for $D_0$. The first case arises from consideration of counting functions of convex polytopes which involve no isotropic vectors, the second one from that of  convex polytopes which involve finitely many isotropic vectors, the final one from that of concave polytopes which can involve infinitely many isotropic vectors.}
		\label{figure:crosssectionpolyhedron}
	\end{figure}
\end{remark}

\begin{definition}
	We call a face $D_0\cap F_I$ an {\it isotropic ray} if it is one-dimensional. By Proposition \ref{propisotricconestructure}, this holds if $I=\{1,2,\dots,N\}\setminus\{a,b\}$ with $b\ne a\pm1$ and $\sin(\vf_b-\vf_a)=0$. There are at most $\binom{N}{N-2}-N-4$ isotropic rays.
\end{definition}

By a change of labelings if needed, we can assume such an isotropic ray  is given by
$
D_0\cap  F_{I}\,, I=\{1,2,\cdots, N\}\setminus\{\ell, N-1\}
$
for some $1\leq \ell \leq N-3$.
We also denote
\[
	\mathcal{E}_\ell := \{\bm c_{\bm j} :  j\in\{1,2,\dots,{N-2}\}\setminus\{\ell\}\},\qquad E_\ell := \operatorname{span}_\R\{\bm c_{\bm j} : j\in\{1,2,\dots,{N-2}\}\setminus\{\ell\}\} = \operatorname{span}(\mathcal{E}_\ell).
\]
According to \eqref{eqndfnofNlambda}, \eqref{eqndefinitionofA}, and \eqref{eqnnormalvectors}, there exists an isotropic ray if for some $1\le\ell\le N-3$,
\begin{equation*}
	B(\bm w,\bm{c_\ell}) = B(\bm{c_N},\bm{e_\ell}) = w_\ell = 0\,;
\end{equation*}
note that $w_{N-2}\ne0$ by the positivity condition \eqref{eqnpositivityonangles}. It is possible that $\dim(F_J):=\dim(\sangle{F_J})<\dim(V_J)$ for $J\subseteq\{1,\dots,N\}$, for example, if $J=\{j\}$, where $V_{\{j\}}$ fails to be a supporting hyperplane of $\xoverline\De$; see the middle graph of Figure \ref{figure:crosssectionpolyhedron}. We next discuss this subtlety, which leads to another characterization of isotropic rays that is useful in studying modularity of counting functions.

\begin{lemma}	\label{lemlineardepenence}
	For any $J\subseteq \{1,\dots,N\}$ with $ |J|\le N-2$,
	let $J_{*}$ be a maximal subset (under inclusion) among those $K\subseteq \{1,\dots,N\}$ satisfying $J\subseteq K$ and $F_{J}=F_{K}$.
	Then we have
	\begin{align*}
		F_J \text{ is an isotropic ray} \quad&\Leftrightarrow\quad \text{there~exists}~J_0\subsetneq J_{*},~\text{such~that}~F_{J_0} = F_{J_{*}}\\
		& \Leftrightarrow\quad \{\l_j\}_{j\in J_{*}} \text{ are linearly dependent}
		\quad \Leftrightarrow \quad\dim(V_{J_{*}}) = N-1-|J_{*}|.
	\end{align*}
\end{lemma}

\begin{remark}
	By the geometry of the polyhedron $\xoverline\De$, the existence of $J_0\subsetneq J$ with $F_J=F_{J_0}$ implies $\l_j\in\sangle{\l_k:k\in J_0}$ for $j\in J\setminus J_0$. By Lemma \ref{lemlineardepenence}, $V_J$, and thus $F_J$, must be an isotropic ray. 
	However, the converse is not true. For example, it is possible that $|J|=1$ with $F_J$ an isotropic ray and thus a desired $J_0$ does not exist; see the middle graph of Figure  \ref{figure:crosssectionpolyhedron} for illustration.
\end{remark}

\subsection{Convergence of counting functions}

The structure of isotropic rays exhibited in Proposition \ref{propisotricconestructure} allows to prove convergence of $\TH_{Q,\chi}$ using the ideas in \cite{Zw}. The key is to find a slicing summation range $\supp(\chi)\cap(\Z^{N-2}+\bm\a)\subseteq\De\cap(\Z^{N-2}+\bm\a)$ ensuring convergence.
It suffices to consider one of the two connected components of $D_0\cap\xoverline\De\setminus\{\bm0\}$, say, the one with $\l_k\ge0$. The following lemma on the normal vectors introduced in \eqref{eqnnormalvectors} is useful in finding a slicing of $\R^{N-2}$.

\begin{lemma}\label{lemnormalvectors}
	Assume the positivity condition \eqref{eqnpositivityonangles} and $N\ge5$. Suppose that $D_0\cap F_I$ is one-dimensional with $I=\{1,2,\dots,N\}\setminus\{\ell,N-1\}$ for some $1\le\ell\le N-3$. Fix a generator $\bm f=f_\ell\bm{e_\ell}$ of this isotropic ray with\footnote{For definiteness, we can take $f_\ell>0$ to be $1$.} $f_\ell>0$.
	\begin{enumerate}[label=\textnormal{(\arabic*)}]
		\item For each $\bm c_{\bm k},1\leq k\leq N$, one has $Q(\bm c_{\bm k})\le0$.
		
		\item  One has $\bm f^\perp=E_\ell$ and $\bm f\in E_\ell$.
		
		\item Choose any subset $\calE_\ell^-$ of $\calE_\ell$ with $N-4$ elements such that vectors in $\calE_\ell^-\cup\{\bm f\}$ are linearly independent. Define $E_\ell^-:=\langle\calE_\ell^-\rangle$. Then one has the orthogonal decomposition
		$
			E_\ell = \R\bm f \oplus E_\ell^-
		$
		with $Q|_{E_\ell^-}<0$. In particular, one has $Q(\bm c_{\bm j})<0$ for any $\bm c_{\bm j}\in\calE_\ell^-$.
	\end{enumerate}
\end{lemma}

\begin{proof}
	(1) Intersecting \eqref{eqnisotropivectors} with $F_{\{k\}}\setminus\bigcup_{j\ne k}F_{\{j\}}$, we get $\xoverline\De\cap F_{\{k\}}\setminus\bigcup_{j\ne k}F_{\{j\}}\subseteq D_+$. Thus, for any $1\le k\le N$, there exists $\bm\l\in D_+$ such that $B(\bm{c_k},\bm\l)=0$. Suppose that $\bm{c_k}$ satisfies $Q(\bm{c_k})>0$ and choose a vector $\bm\l$ as above. By Lemma \ref{la:quadFormsignature}, we have $B(\bm{c_k},\bm\l)>0$ if $\bm{c_k}$ and $\bm\l$ lie in the same component of $\{\bm a\in\R^{N-2}:Q(\bm a)\ge0\}$ and $B(\bm{c_k},\bm\l)<0$ otherwise. Either case contradicts the condition that $B(\bm{c_k},\bm\l)=0$.\\
	(2) By Proposition \ref{propisotricconestructure}, $B(\bm{c_j},\bm f)=0$ for $\bm{c_j}\in\calE_\ell$. For dimension reasons $\bm f^\perp=E_\ell$. Since $Q(\bm f)=0$, one has $\bm f\in \bm f^\perp=E_\ell$. \\
	(3)
	Since $A\bm{c}_{\bm j}=\bm{e}_{\bm j}$ for $1\leq j\leq N-2$ from \eqref{eqnnormalvectors}, $\dim (E_{\ell})=N-3$. By (2), such a choice for $\mathcal{E}_\ell^{-}$ is possible. Any isotropic vector $\bm{u}\in E_\ell^{-}\subseteq \bm{f}^\perp$ gives an isotropic vector space $\langle\bm{f},\bm{u}\rangle$. Since $Q$ has signature $(1,N-3)$, the space is at most one-dimensional and thus $\bm{u}=\bm{0}$. Therefore $Q|_{E_\ell^-}$ is negative definite by (1). In particular, this gives that $Q(\bm{c}_{\bm j})<0$ for $\bm{c}_{\bm j}\in \mathcal{E}_\ell^{-}$.
\end{proof}

We have the following basic estimates basing on the particular expression of $Q$.

\begin{lemma}\label{lemeconeproperties}
	With the assumptions and notation as in Lemma \ref{lemnormalvectors}, the following hold:
	\begin{enumerate}[label=\rm{(\arabic*)}]
		\item We have $v_j>0$ ($1\le j\le\ell$), $w_j<0$ ($1\le j<\ell$), $w_\ell=0$, and $w_j>0$ ($\ell<j\le N-2$).
		
		\item For any $\g\in\R$, there exists $\d>0$, such that for any $\bm\l\in\xoverline\De\cap(\Z^{N-2}+\bm\a)\setminus(\Z+\a_\ell)\bm{e_\ell}$ and any $\bm\rho\in(\Z+\g)\bm{e_\ell}\setminus\{\bm0\}$ that satisfy $\l_k,\rho_k\ge0$ for $1\le k\le N$, we have $B(\bm\l,\bm\rho)\ge\d$.
	\end{enumerate}
\end{lemma}

\begin{proof}
	(1) By Proposition \ref{propisotricconestructure}, we have $\sin(\vf_{N-1}-\vf_\ell)=0$ and thus by \eqref{eqnpositivityonangles} $\vf_{N-1}-\vf_\ell=\pi$, $0<\vf_\ell<\pi<\vf_{N-1}<\vf_N=2\pi$. Explicit computations, using \eqref{eqndfnofNlambda}, give the desired claim.\\
	(2)	In the component of the region $\xoverline{\Delta}\setminus\{\bm{0}\}\subseteq\{\bm{\lambda}\in\R^{N-2}:Q(\bm{\lambda})\ge0\}$ that contains $\bm{f}=f_\ell\bm{e}_{\bm\ell}$, namely the one with $\l_k\ge0$ for $1\leq k\leq N$, we have
	\begin{equation*}\label{eqnfinerestimateonB}
		B(\bm{\l},\bm{\rho}) =  \sin(\th_N) \left(\rho_\ell w_\ell\sum_{j\leq\ell} v_j \l_j + \rho_\ell v_\ell\sum_{\ell<j\leq N-2} w_j \l_j\right)=\sin(\th_N)\cdot \rho_\ell\cdot v_\ell\cdot \sum_{\ell<j\leq N-2} w_j \l_j. \notag
	\end{equation*}
	Since $\rho_\ell>0$, we obtain $B(\bm{\l},\bm{\rho})\ge0$ with equality only if $\l_j=0$ for $\ell<j\le {N-2}$. In the case $\l_j=0$ for $\ell<j\le {N-2}$, recalling $\l_k\ge0$ for $1\leq k\leq N$, we have in particular
	\begin{equation*}\label{eqnlineardependenceonN}
		\l_{N} = \sum_{1\le j<\ell} w_j \l_j + w_\ell \l_\ell + \sum_{\ell+1\le j\le N-2} w_j \l_j = \sum_{1\le j<\ell} w_j \l_j \ge 0.
	\end{equation*}
	With (1) this gives $\l_j=0$ $(1\le j\le\ell-1)$ and thus $\bm\l\in\R\bm f$. We can thus assume that at least one of $\l_j$ $(\ell<j\le N-2)$ is non-zero. Letting $\k:=\min\bigcup_j(\Z+\a_j)\setminus\{0\}$, we have $\min\{\l_j\ne0:\ell<j\le N-2\}\ge\k$. Since $0<\rho_\ell\in\Z+\g$ implies that $\min(\rho_\ell)>0$, we obtain
	\begin{equation*}
		B(\bm\l,\bm\rho) \ge \sin(\th_N)\cdot  \min (\rho_\ell) \cdot v_\ell \cdot \min_{j>\ell} w_j \cdot \kappa  =: \d>0. \qedhere
	\end{equation*}
\end{proof}

Using Lemmas \ref{lemnormalvectors} and \ref{lemeconeproperties} we now give a slicing of $\xoverline\De\cap(\Z^{N-2}+\bm\a)$. We let $B_\e(\bm f):=\{\bm a\in \R^{N-2}:||\bm{a-f}||<\e\}$ and set $\R X:=\{t\bm a:t\in\R,\bm a\in X\}$.

\begin{lemma}\label{lemnbdofisotropicray}
Assume the conditions and notation as in Lemma \ref{lemnormalvectors}. Assume furthermore \eqref{eqnrelaxedconditionintermsangles}. Then there exist $\varepsilon,\delta>0$ sufficiently small and a vector $\bm{\eta}$ with $Q(\bm{\eta})\neq 0$, such that any
$\bm{\lambda}\in \R(B_{\varepsilon}(\bm{f})\cap \xoverline{\Delta}) \cap ( \mathbb{Z}^{N-2}+\bm{\alpha})$ lying in the same component as $\bm{f}$
admits a unique decomposition
\begin{align}\label{eqnuniquedecomposition}
\bm{\lambda}=
\bm{\rho}+\bm{\mu}\in \mathbb{L}_1\oplus \mathbb{L}_2\subseteq \mathbb{R}\bm{f}
\oplus \bm{\eta}^{\perp}.
\end{align}
Here
$\mathbb{L}_{1}, \mathbb{L}_{2}$ are  sets of  certain affine transformations of $\mathbb{Z},\mathbb{Z}^{N-3}$, respectively. Moreover
\begin{align}\label{eqnpropertiesofdecomposition}
B(\bm{\rho}, \bm{\mu} )\geq \delta\,,\quad  Q(\bm{\mu})>0\,\quad \text{for }\, \bm{\lambda}\notin (\mathbb{Z}+\alpha_{\ell})\bm{e}_{\bm \ell}\cup\{\bm{0}\}.
\end{align}
\end{lemma}

\begin{proof}
	Let $\bm\eta:=a\bm{c_\ell}-\bm f$ with $a>0$ sufficiently small such that
	\begin{equation*}\label{eqnQeta>0}
		Q(\bm{\eta})=a^2 Q(\bm{c}_{\bm \ell})-a B(\bm{f},\bm{c}_{\bm \ell})\neq 0
	\end{equation*}
	and (recalling that $v_{\ell}>0, w_j>0$ for $\ell<j\leq N-2$ as shown in Lemma \ref{lemeconeproperties})
	\begin{align}\label{eqnetachoice}
		av_{j}+f_{\ell}\sin (\th_{N})  v_{\ell}^2 w_{j}>0\qquad \text{for }\ell < j \leq N-2.
	\end{align}
	From \eqref{eqnnormalvectors} we also have
	\begin{align}\label{eqnBfeta}
		B(\bm{f},\bm{\eta})=aB(\bm{f},\bm{c_\ell})=af_{\ell}>0.
	\end{align}
	
	Take any $E_\ell^-$ provided in Lemma \ref{lemnormalvectors} (3). Since $\bm{c_1},\dots,\bm{c_{N-2}}$ are linearly independent we have $\bm{c_\ell}\notin E_\ell=\R\bm f\oplus E_\ell^-$, and thus
	\begin{equation*}\label{eqnVdecomposition1}
		V=\mathbb{R}\bm{f}\oplus E_\ell^-\oplus \mathbb{R}\bm{c_\ell}=\mathbb{R}\bm{f}\oplus E_\ell^-\oplus \mathbb{R}\bm{\eta}.
	\end{equation*}
	
	\begin{figure}[h]
		\centering
		\includegraphics[scale=0.6]{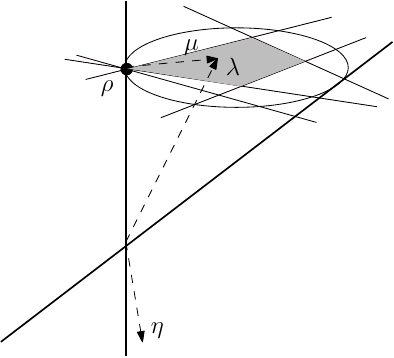}
		\caption{Slicing $\mathbb{R}^{N-2}$ using $\bm{\rho}$ and $\bm{\eta}^{\perp}$.}
		\label{figure:slicing}
	\end{figure}

	By the condition $Q(\bm\eta)\ne0$, the map $\bm x\mapsto B(\bm x,\bm\eta)$ is non-degenerate and thus $\bm\eta^\perp$ has dimension $N-3$. Since $B(\bm f,\bm\eta)\ne0$, we have  $\R\bm f\cap\bm\eta^\perp=\{\bm{0}\}$ and therefore a vector space decomposition
	\begin{align}\label{eqnslicing}
		V=\mathbb{R}\bm{f}\oplus \bm{\eta}^{\perp}.
	\end{align}
	In particular, any $\bm\l\in V$ admits a unique decomposition $\bm\l=\bm\rho+\bm\mu$ with $\bm\rho\in\R\bm f$, $\bm\mu\in\bm\eta^\perp$ as desired in \eqref{eqnuniquedecomposition}. Concretely,
	the component $\bm\mu$ is given by (see Figure \ref{figure:slicing} for an illustration)
	\begin{align}\label{eqnmucomponent}
		\bm{\mu}:=\bm{\lambda}-{B(\bm{\lambda},\bm{\eta})\over B(\bm{f},\bm{\eta}) }\bm{f}.
	\end{align}
	
	We next show that there exists $\e>0$ sufficiently small, such that for any $\bm\l\in\R(B_\e(\bm f)\cap\xoverline\De)\cap(\Z^{N-2}+\bm\a)$ that lies in the same component as $\bm f$, the resulting decomposition \eqref{eqnuniquedecomposition} satisfies \eqref{eqnpropertiesofdecomposition}. First, from \eqref{eqnmucomponent}, the set of the components $\bm\rho$ and $\bm\mu$ arising from $\bm\l\in\xoverline\De\cap(\Z^{N-2}+\bm\a)$ is given by affine transformations $\L_1,\L_2$ of the lattices $\Z,\Z^{N-3}$, respectively. To be more precise, we have
	\begin{align}\label{eqnlatticeL1}
		\L_1 &= \left\{\left(\l_\ell-a^{-1}\sin(\th_N)f_\ell v_\ell\sum_{\ell+1\le j\le N-2} w_j\l_j\right)\bm{e_\ell} : \bm\l\in\xoverline{\Delta}\cap\left(\Z^{N-2}+\bm\a\right)\right\},\\
		\nonumber
		\L_2 &= \left\{\sum_{\substack{j=1\\j\ne\ell}}^{N-2}\l_j\bm{e_j}+\left(a^{-1}\sin(\th_N)f_\ell v_\ell\sum_{\ell+1\le j\le N-2} w_j\l_j\right)\bm{e_\ell} : \bm\l\in\xoverline{\Delta}\cap\left(\Z^{N-2}+\bm\a\right)\right\}.
	\end{align}
	Under assumption \eqref{eqnrelaxedconditionintermsangles} which implies \eqref{eqnrelaxedconditionintermsofstandardform} by our choice of $\{\e_k\}_{1\le k\le N}$, $a$ can be chosen suitably such that indeed $\L_1,\L_2$ are certain affine transformations of the lattices $\Z,\Z^{N-3}$.
	Furthermore, from $B(\bm f,\bm\eta)>0$ in \eqref{eqnBfeta} and the continuity of $B$, there exists $\e>0$ sufficiently small, such that $B(\bm\l,\bm\eta)>0$ for any $\bm\l\in\R B_\e(\bm f)$. From $B(\bm\rho,\bm\eta)=B(\bm\l,\bm\eta)>0$, it follows that $\bm\rho$ lies in the same component as $\bm f$.
	Applying the reasoning of Lemma \ref{lemeconeproperties} (2) to $\bm\rho\in\L_1 $, there exists $\d>0$ such that $B(\bm\l,\bm\rho)\ge\d$ for those non-zero $\bm\l$ subject to the extra condition $\bm\l\notin(\Z+\a_\ell)\bm{e_\ell}$. Thus we have
	\begin{align*}
		B(\bm{\rho},\bm{\mu})=B(\bm{\rho}, \bm{\lambda}-\bm{\rho})=B(\bm{\lambda},\bm{\rho})\geq \delta\,,
	\end{align*}
	for $\bm{\lambda}\in \R(B_{\varepsilon}(\bm{f})\cap \xoverline{\Delta}) \cap (\mathbb{Z}^{N-2}+ \bm{\alpha})$ satisfying $\bm{\lambda}\notin (\mathbb{Z}+\alpha_{\ell})\bm{e}_{\bm \ell}\cup\{\bm{0}\}$.

	Finally
	we show that $Q(\bm\mu)>0$ for the resulting component $\bm\mu$. Since $\bm\l\in\xoverline\De$ we have $B(\bm{c_k},\bm\l)\ge0$ for $1\le k\le N$. Recalling that $B(\bm{c_j},\bm f)=0$ for $\bm{c_j}\in\calE_\ell$, we thus have
	\begin{align}\label{eqnboundmuinequality1}
		B(\bm{c}_{\bm j}, \bm{\mu})=B(\bm{c}_{\bm j}, \bm{\rho}+\bm{\mu})=B(\bm{c}_{\bm j},\bm{\lambda})\geq 0\,\quad \text{for } \bm{c}_{\bm j}\in \mathcal{E}_\ell.
	\end{align}
	By inserting the definitions of $\bm\mu$ and $\bm\eta$, we obtain that our choice of $\bm\eta$ ensures that 
	(note that $B(\bm\mu,\bm\eta)=0$ by construction)
	\begin{align}\label{eqnboundmuinequality2}
		B(\bm{\mu},\bm{c}_{\bm \ell})=a^{-1}B(\bm{\mu},\bm{f})=a^{-1}B(\bm{\lambda}-\bm{\rho},\bm{f})=a^{-1} B(\bm{\l},\bm{f})\geq 0
	\end{align}
	for $\bm\l\in\R(B_\e(\bm f)\cap\xoverline{\Delta})\cap(\Z^{N-2}+\bm\a)$ and
	the inequality is strict if $\bm{\lambda}\notin \mathbb{R}\bm{f}=\R\bm{e_\ell}$ by Lemma \ref{lemeconeproperties} (2).
	Furthermore, using \eqref{eqnnormalvectors}, \eqref{eqndefinitionofA}, \eqref{eqnetachoice}, and Lemma \ref{lemeconeproperties},  we obtain
	\begin{align}
		aB(\bm{f},\bm{c_\ell})B(\bm{\mu},\bm{c}_{\bm{N-1}})
		=a \sum_{1\leq j\leq  \ell-1}v_{j}\lambda_{j}
		+\sum_{\ell+1\leq j\leq  N-2}\left(av_{j}
		+ f_{\ell}\sin (\th_{N})  v_{\ell}^2 w_{j} \right)   \lambda_{j}\geq 0\notag\,,
	\end{align}
	where again the inequality is strict if $\bm{\lambda}\notin \mathbb{R}\bm{f}=\R\bm{e_\ell}$.
	Note that $B(\bm{f},\bm{c_\ell})>0$ by \eqref{eqnBfeta}. With \eqref{eqnboundmuinequality1} and \eqref{eqnboundmuinequality2} this shows that $\bm{\mu}\in \xoverline{\Delta}\subseteq\{\bm{x}\in \R^{N-2}: Q(\bm{x})\ge0\}$ and thus $Q(\bm{\mu})\geq 0$.
	If $\bm{\lambda}\in \mathbb{R}\bm{f}$, then $\bm{\mu}=0$ according to the direct sum structure in \eqref{eqnslicing}.
	If $\bm{\lambda}\notin \mathbb{R}\bm{f}$, then $Q(\bm{\mu})>0$ since otherwise
	 one has  $Q(\bm{\mu})= 0$ and by Proposition \ref{propisotricconestructure} one has $\bm{\mu}\in \mathbb{R}\bm{f}$ and thus $\bm{\lambda}\in \mathbb{R}\bm{f}$ by \eqref{eqnslicing},
	contradicting $\bm{\lambda}\notin \mathbb{R}\bm{f}$.
	This gives that  $Q(\bm{\mu})>0$ if $\bm{\lambda}\notin (\mathbb{Z}+\alpha_{\ell})\bm{e}_{\bm \ell}\cup\{\bm 0\}$ as desired, completing the proof.
\end{proof}

By the discussions about rationality in Remark \ref{remasymmetry},
it follows that under a different choice of independent parameters the results in Lemma \ref{lemnormalvectors} and \ref{lemeconeproperties},
and the lattice structure for $\mathbb{L}_1,\mathbb{L}_{2}$ in \eqref{eqnlatticeL1} still hold, with the domain for
$\bm{\lambda}$ in Lemma \ref{lemeconeproperties} (2) changed to an affine transformation of the sub-lattice mentioned there.
Choosing one $a$ for each isotropic ray, we obtain that
a decomposition in Lemma \ref{lemnbdofisotropicray} exists for each isotropic ray. We
use this fact in what follows without explicitly mentioning. We next prove the convergence for $\Theta_{Q,\chi}$.
\begin{theorem}\label{thmconvergenceofcountingfunction}
	Assume \eqref{eqnpositivityonangles}, $N\ge5$,
	and \eqref{eqnrelaxedconditionintermsangles}.
		Then $\Theta_{Q,\chi}$ converges compactly, i.e., uniformly on compact sets.

\end{theorem}

\begin{proof}
	We only need to consider one of the two connected components for the summation range, say, the one with $\l_k\ge0$. By Corollary \ref{corpositivityofarea}, $Q(\bm\l)\ge0$ for $\bm\l\in\xol\De$ which contains the summation range in \eqref{eqncountingfunctionwithparameterlambda}. Consider first the case $D_0\cap\xol\De=\{\bm0\}$. Let $S:=\{\bm\l\in\R^{N-2}:||\bm\l||=1\}$ be the unit sphere. Then $Q(\bm\l)\ge0$ for $\bm\l\in\xol\De\cap S$ and thus attains its minimum $m>0$. Compact convergence follows since for any $\bm\l\in\xol\De\setminus\{\bm0\}$, one has
	\[
		Q(\bm{\l}) = ||\bm\l||^2Q\left(\frac{\bm\l}{||\bm\l||}\right) \ge ||\bm\l||^2m > 0\,,\qquad \left|q^{Q(\bm\l)} e^{2\pi iB(\bm\l,\bm\b)}\right| \le e^{-2\pi ||\bm\l||^2my}.
	\]
	
	Next assume that $D_0\cap ( \xoverline{\Delta}\setminus\{\bm{0}\})\neq \emptyset$.
	By Proposition \ref{propisotricconestructure}, only finitely many rays are contained  in
	$D_0\cap ( \xoverline{\Delta}\setminus\{\bm{0}\})$ which are given by $D_0\cap (  F_{I}\setminus\{\bm{0}\}), |I|={N-2}$.
	By a relabeling of vertices if needed, any such ray $\mathbb{R}\bm{f}$  can be described as
	$
		\l_k = 0,\, k\in\{1,2,\dots,N\}\setminus\{\ell,N-1\},
	$
	with $1\leq \ell\leq N-3$. We take  $\l_1,\dots,\l_{N-2}$ as independent parameters for $\xoverline{\Delta}$ as before. Applying Lemma \ref{lemnbdofisotropicray} we have
	\begin{equation*}
	 	\sum_{\bm\l\in\mathbb{R}(B_\varepsilon(\bm f)\cap\Delta)\cap\left(\Z^{N-2}+\bm\a\right)} \left|q^{Q(\bm\l)}e^{2\pi iB(\bm\b,\bm\l)}\right|\le\sum_{\bm{\mu}\in \mathbb{L}_{2}} \left|q^{Q(\bm\mu)}\right|   \sum_{\bm\rho\in\mathbb{L}_{1}} \left|e^{2\pi i B(\bm\mu,\bm\rho)\t}\right|.
	\end{equation*}
	We only need to consider $\bm\l\ne\bm0$. By Lemma \ref{lemnbdofisotropicray}, $\rho_\ell>0$ as $\l_k\ge0$ and $B(\bm\mu,\bm\rho)\ge\d>0$ since $\bm\l\notin\R\bm f$ (note $\R\bm f\cap \Delta=\{\bm{0}\}$).
	Writing $\mathbb{L}_{1}$ in \eqref{eqnlatticeL1} as $(a_{2}\mathbb{N}+a_{1})\bm{e_\ell}$  for some $a_{1}\in \mathbb{R}, a_{2}\in \mathbb{R}^{+}$, the inner sum is bounded
	by the geometric series 
	$\sum_{k=0}^{\infty} e^{-2\pi  a_2\delta k v}$
	and thus converges compactly.
	The outer sum is absolutely convergent since $Q(\bm{\mu})>0$.
	We use the splitting
	\[
		\De = \bigcup_{\bm f} (\R(B_\e(\bm f)\cap\De)) \cup \left(\De\Big\backslash\bigcup_{\bm f} (\R(B_\e(\bm f)\cap\De))\right)
	\]
	and intersect it with $\Z^{N-2}+\bm\a$. For the
	summation over the former we apply the reasoning mentioned above to each isotropic ray $\mathbb{R}\bm{f}$,  for the summation over the latter
	we apply the compactness argument in the case $D_0\cap \xoverline{\Delta}=\{\bm{0}\}$. This proves the theorem.
\end{proof}

We also consider the following closely related series
\begin{equation}\label{eqnThetaseriesdivergentpart}
	\sum_{\bm\l\in\R\left(B_\e(\bm f)\cap\xoverline\De\right)\cap\left(\Z^{N-2}+\bm\a\right)}\hspace{-.25cm} q^{Q(\bm\l)}e^{2\pi iB(\bm\l,\bm\b)} = q^{Q(\bm\a)}q^{2\pi iB(\bm\a,\bm\b)}\hspace{-.25cm}\sum_{\bm\l\in\R\left(B_\e(\bm f)\cap\xoverline\De\right)\cap\left(\Z^{N-2}+\bm\a\right)}\hspace{-.25cm} q^{Q(\bm n)}e^{2\pi iB(\bm n,\bm z)}.
\end{equation}

Before proceeding we study $\R\bm f\cap(\Z^{N-2}+\bm\a)$ more closely. To investigate integrality we can not directly apply the cyclic symmetry anymore, since under a cyclic permutation the range $\Z^{N-2}+\bm\a$ changes in terms of a new set of consecutive $N-2$ parameters; see Remark \ref{remasymmetry}. We always work with a fixed set of independent parameters $\{\l_k\}_{1\le k\le N}$. We require the following.

\begin{lemma} \label{lemlatticeisotropicvector}
	Assume notation and assumptions as in Theorem \ref{thmconvergenceofcountingfunction}. Then $\R\bm f\cap(\Z^{N-2}+\bm\a)\ne\emptyset$ if and only if $\bm\a\in\R\bm f$. In this case $\R\bm f\cap(\Z^{N-2}+\bm\a)$ is a translated  rank-one sublattice of $\Z^{N-2}$.
\end{lemma}

\begin{proof}
	By Proposition \ref{propisotricconestructure} an isotropic ray $\R\bm f$ is given by $\l_k=0$, $k\in\{1,\dots,N\}\setminus\{a,b\}$, for some $a<b$ with $b\ne a+1$. In order for\footnote{Recall that here $\Z^{N-2}+\bm\a$ means $S(\Z^{N-2}+\bm\a)$ with $S$ given in Lemma \ref{lemrelaxedcondition}.} $\R\bm f\cap(\Z^{N-2}+\bm\a)\ne\emptyset$, a first condition on $\a_k$ ($1\le k\le N$) is
	\begin{equation}\label{eqnpoleconditionreal}
		\a_k\in\Z,\qquad k\in\{1,\dots,N\}\setminus\{a,b\}.
	\end{equation}
	Here we recall that a solution $\{\e_k\}_{1\le k\le N}$ to \eqref{eqnrelaxedprimitiveconditionintermsofstandardform} is fixed from the beginning in the definition of the counting function, as discussed in Subsection \ref{secintegralityofsignedarea}. If $b\in\{N-1,N\}$, then we obtain that
	\[
		\R\bm f \cap \left(\Z^{N-2}+\bm\a\right) = \R\bm{e_a} \cap \left(\Z^{N-2}+\bm\a\right) = (\Z+\a_a)\e_a\bm{e_a}\,,
	\]
	and \eqref{eqnpoleconditionreal} becomes
	\begin{equation*}\label{eqnpoleconditionrealcase1}
		\a_k\in\Z,\qquad k\in\{1,\dots,N-2\}\setminus\{a\},\qquad
		\begin{cases}
			\sum_{j=1}^{N-2} \frac{v_j\e_j}{\e_{N-1}}\a_j\in\Z &\text{if } b=N,\\
			\sum_{j=1}^{N-2} \frac{w_j\e_j}{\e_N}\a_j\in\Z &\text{if } b=N-1.
		\end{cases}
	\end{equation*}
	If $1\leq a<b\leq N-2$, then the set $\mathbb{R}\bm{f}\cap  (\mathbb{Z}^{N-2}+\bm{\alpha})$ consists of vectors in $\mathbb{Z}^{N-2}+\bm{\alpha}$ satisfying
	\begin{align*}
		\bm{v}=\lambda_a \bm{e_a}+\lambda_b \bm{e_b}\,,\quad
		\begin{pmatrix}
		v_{a} &v_{b} \\
		w_{a} & v_{b}
		\end{pmatrix}
		\begin{pmatrix}
		\lambda_a\\
		\lambda_b
		\end{pmatrix}=0.
	\end{align*}
	The requirement \eqref{eqnrelaxedprimitiveconditionintermsofstandardform} implies that $v_a=0$, $v_b=0$, $v_a\e_a=v_b\e_b$, or $v_a\e_a=-v_b\e_b$. By basic Linear Algebra, $\R\bm f\cap(\Z^{N-2}+\bm\a)$ is non-empty if and only if $\a_b\in\Z$, $\a_a\in\Z$, $\a_a+\a_b\in\Z$, $\a_a-\a_b\in\Z$. The condition \eqref{eqnpoleconditionreal} then becomes
	\begin{equation*}\label{eqnpoleconditionrealcase2part1}
		{\alpha_{k}}\in \mathbb{Z}\,,~k\in \{1,\dots,N-2\}\setminus\{a,b\}\,,\quad
		\sum_{j=1}^{N-2}{ v_{j}  \varepsilon_{j} \over \varepsilon_{N-1}}{\alpha_{j}}\in\mathbb{Z}\,,~
		\sum_{j=1}^{N-2} {w_{j} \varepsilon_{j}  \over \varepsilon_{N}}{\alpha_{j} }\in\mathbb{Z}.
	\end{equation*}
	In any of the above cases, the set $\mathbb{R}\bm{f}\cap  (\mathbb{Z}^{N-2}+\bm{\alpha})$  consists of an affine transformation of a rank-one sub-lattice
	of  $\mathbb{Z}^{N-2}$ that contains $\bm{\alpha}$ itself.
\end{proof}

We have the following meromorphicity result.

\begin{corollary}\label{corThetaseriespolarpart}
	Under the notation and assumptions as in Theorem \ref{thmconvergenceofcountingfunction}, the series
	\[
		\sum_{\bm\l\in\R\left(B_\e(\bm f)\cap\xoverline\De\right)\cap\left(\Z^{N-2}+\bm\a\right)} q^{Q(\bm n)}e^{2\pi iB(\bm n,\bm z)}
	\]
	is holomorphic everywhere, except along 
	$P_{\bm f}=\{\bm z:\bm\a\in\R\bm f,B(\bm f,\bm\b)=0\}$, where it is meromorphic with a simple pole.
\end{corollary}

\begin{proof}
	The series in \eqref{eqnThetaseriesdivergentpart} equals
	\begin{equation*}
		\sum_{\bm\l\in\mathbb{R}\left(B_\varepsilon(\bm f)\cap\Delta\right)\cap\left(\Z^{N-2}+\bm\a\right)}  q^{Q(\bm\l)}e^{2\pi iB(\bm\b,\bm\l)} + \sum_{\bm\l\in \mathbb{R}\bm{f}\cap\xoverline{\Delta}\cap\left(\Z^{N-2}+\bm\a\right)}  q^{Q(\bm\l)}e^{2\pi iB(\bm\b,\bm\l)}.
	\end{equation*}
	The first summand converges compactly by Theorem \ref{thmconvergenceofcountingfunction}. If $\R\bm f\cap\xoverline\De\cap(\Z^{N-2}+\bm\a)$ is non-empty, then, by Lemma \ref{lemlatticeisotropicvector} it is a translated rank-one lattice and $\bm\a\in\R\bm f$. As before we parametrize the isotropic ray $\R_0^+\bm f\cap\xoverline\De\cap(\Z^{N-2}+\bm\a)$ by $(a_2\N+a_1)\bm f$ for some $a_1\in\R,a_2\in\R^+$. Then the second summand equals
$
 \sum_{n\geq 0}  e^{2\pi i B(\bm{f},\bm{\beta})(na_2+a_1)}.
$
This series is conditionally convergent for those $\bm\b$ satisfying $B(\bm f,\bm\b)\ne0$ and the result is $\frac{e^{2\pi ia_1B(\bm f,\bm\b)}}{1-e^{2\pi ia_2B(\bm f,\bm\b)}}.$ The same analysis applies to the other half of the isotropic ray $\mathbb{R}\bm{f}$.
Thus
the original series is meromorphic with a simple pole at $\bm{z}=\bm{\alpha}\tau+\bm{\beta}$ if there exists an isotropic ray $\mathbb{R}\bm{f}$ such that $\bm{\alpha}\in\mathbb{R}\bm{f},\, B(\bm{f}, \bm{\beta})=0.$ This finishes the proof.
\end{proof}

The function $\TH_{Q,\chi}$ has the following elliptic properties.

\begin{corollary}\label{corThetaseriesellipticity}
	Let the notation and assumptions be the same as in Theorem \ref{thmconvergenceofcountingfunction}. Then
	\[
		\TH_{Q,\chi}(\bm z+\bm a\t+\bm b;\t) = q^{-Q(\bm a)} e^{-2\pi iB(\bm a,\bm z)} \TH_{Q,\chi}(\bm z;\t)\,,\qquad \bm a,\bm b\in\Z^{N-2}.
	\]
\end{corollary}

\subsection{Faces of $\xoverline\De$ and degenerations of polygons}\label{secfacesanddegenerations}

By Lemma \ref{lemlineardepenence}, one has $\lambda_{k}=0$ on the face $F_{K}$ for any $k\in K_{*}\,,$
where $K_{*}$ is the maximal set containing $K$ such that $F_{K}=F_{K_{*}}$ with $\dim (V_{K_{*}})=\dim( F_{K_{*}})$.
Each face $F_{K}$ can then uniquely be represented  by an index set $K_{*}\subseteq \{1,\dots,N\}$.
Denote by $\mathcal{F}(\xoverline{\Delta})$ the collection of faces of $\xoverline{\Delta}$, and by $\mathcal{F}_{0}(\xoverline{\Delta})$ the set of
faces which are isotropic rays.
Abusing notation we also use the notation $K_{*}\in \mathcal{F}(\xoverline{\Delta})$ if $F_{K_{*}}\in \mathcal{F}(\xoverline{\Delta})$.

\begin{lemma}\label{lemfacesofDeltaanddegenerations}
	For $K_*\in\FF(\xoverline{\Delta})\setminus\FF_0(\xoverline{\Delta})$, $Q|_{V_{K_*}}$ is non-degenerate of signature $(1,N-3-\lvert K_*\rvert)$.
\end{lemma}

\begin{proof}
	Assume that $Q|_{V_{K_*}}$ is degenerate, i.e., there exists $\bm h\in V_{K_*}$ such that $V_{K_*}\subseteq\bm h^\perp$, and in particular $\bm h\in D_0$. Since $(V,Q)$ has signature $(1,N-3)$, its largest degenerate subspaces are one-dimensional (see beginning of the proof of Proposition \ref{propisotricconestructure}), and thus $V_{K_*}\cap D_0=\R\bm h$. If $\dim(V_{K_*}\cap\De)\ge1$, then, by working with a cross section of $\xoverline\De$, $V_{K_*}\cap D_0$ can not be a line. Therefore $\dim(V_{K_*}\cap\De)\le0$. If $\dim(V_{K_*}\cap\De)=0$, then $V_{K_*}$ is an one-dimensional face contained in $\De$ since $V_{K_*}$ is a supporting subspace of $\xoverline\De$ by the maximality of $K_*$. In particular, $Q|_{V_{K_*}}$ is positive definite. This contradicts the degeneracy of $Q|_{V_{K_*}}$. Therefore $V_{K_*}\cap\De=\emptyset$ and $F_{K_*}=V_{K_*}\cap\xoverline\De=V_{K_*}\cap D_0=\R\bm h$ is an isotropic ray. This contradicts the assumption that $F_{K_*}\in\FF(\xoverline\De)\setminus\FF_0(\xoverline\De)$. Thus $Q|_{V_{K_*}}$ is non-degenerate. Since $Q$ has signature $(1,N-3)$ and $\emptyset\ne F_{K_{*}}\cap D_+\subseteq V_{K_*}\cap D_+$, $Q|_{V_{K_*}}$ has signature $(1,\dim(V_{K_*})-1)$, with $\dim(V_{K_*})=N-2-|K_*|$ by Lemma \ref{lemlineardepenence}.
\end{proof}

Restricting to $V_{\{k\}}:=\{\bm\l\in\R^{N-2}:\l_k=0\}$ preserves the integral lattice structure of the domain and the integral lattice structure for isotropic rays in Lemma \ref{lemlatticeisotropicvector}. Regarding the above series as a sub-series of the original one, Proposition \ref{propisotricconestructure}, Theorem \ref{thmconvergenceofcountingfunction}, Lemma \ref{lemlatticeisotropicvector}, and Corollary \ref{corThetaseriespolarpart} follow automatically.
These results can be alternatively obtained by using the geometric interpretation for $Q|_{F_{K_*}}$ whose extension gives $Q|_{V_{K_*}}$. By iteration, it suffices to consider points in $\xoverline\De$ that lie on the hyperplane $V_{\{k\}}$. If $\th_k+\th_{k+1}=\pi$, then $V_{\{k\}}\cap\xoverline\De$ is an isotropic ray representing those degenerated polygons which become lines. If $\th_k+\th_{k+1}>\pi$, then, by using $\l_j$, $j\in\{1,\dots,N\}\setminus\{k,k+1\}$, as independent parameters and computing $\l_k$ from \eqref{eqndfnofNlambda}, we obtain $V_{\{k\}}\cap\xoverline\De=\{\bm0\}$. If $\th_k+\th_{k+1}<\pi$, then such points represent degenerations of the original convex $N$-gon which are still convex, as discussed in Lemma \ref{la:quadFormsignature}. 
With the new angles, the signature of the resulting signed area follows from Lemma \ref{la:quadFormsignature}. The corresponding condition \eqref{eqnrelaxedconditionintermsangles} holds and solutions to \eqref{eqnrelaxedconditionintermsofstandardform} or to \eqref{eqnrelaxedprimitiveconditionintermsofstandardform} descend automatically. 

Consider the restrictions of the counting function on the faces. Denote by $\d_X(x)$ the characteristic function of the set $X$. Define, for $\bm z\notin\bigcup_{\bm f}P_{\bm f}$
\begin{align*}
	\TH_{Q,\d_{\xoverline{\Delta}}}(\bm z;\t) &:= q^{-Q(\bm\a)} e^{-2\pi iB(\bm\a,\bm\b)} 
	\hspace{-3mm} 
	\sum_{\bm\l\in\Z^{N-2}+\bm\a} \d_{\xoverline{\Delta}}(\bm\l) q^{Q(\bm\l)} e^{2\pi iB(\bm\l,\bm\b)}
	= 
	\hspace{-3mm} 
	\sum_{\bm n\in\Z^{N-2}} \d_{\xoverline{\Delta}}(\bm n+\bm\a) q^{Q(\bm n)} e^{2\pi iB(\bm n,\bm z)}.
\end{align*}
Let $\bm{\lambda}|_{V_{K_{*}}}$ be the restriction of $\bm{\lambda}$ to $V_{K_{*}}$. Then the
restriction of the  above sum to the face $F_{K_{*}}=\xoverline{\Delta}\cap V_{K_{*}}$ gives that $\TH_{Q,\d_{\xoverline{\Delta}\cap V_{K_*}}}$ equals
\[
	q^{-Q(\bm\a)} e^{-2\pi iB(\bm\a,\bm\b)}
	 \sum_{\bm\l\in\Z^{N-2}+\bm\a} \d_{F_{K_*}}(\bm\l) q^{Q|_{V_{K_*}}\left(\bm\l|_{V_{K_*}}\right)} e^{2\pi iB|_{V_{K_*}}\left(\bm\l|_{V_{K_*}},\bm\b|_{V_{K_*}}\right)} e^{2\pi iB\left(\bm\l|_{V_{K_*}},\bm\b-\bm\b|_{V_{K_*}}\right)}.
\]
Consider, for example, $K_*=\{k\}$, for some $1\le k\le N-2$ with $\th_k+\th_{k+1}<\pi$. The non-vanishing of $\d_{F_{K_*}}(\bm\l)$ in the above sum requires $\a_k\in\Z$. Further assume that $\bm\b$ satisfies $B(\bm\l|_{V_{K_*}},\bm\b-\bm\b|_{V_{K_*}})=0$, which can be achieved, for instance, by requiring $\b_k=0$ which in return implies $\bm\b=\bm\b|_{V_{K_*}}$. Then the above series is exactly the corresponding generating function of the degenerated polygons with angles $\{\th_1,\dots,\th_{k-1},\th_k+\th_{k+1},\th_{k+2},\dots,\th_N\}$ and signed area given in \eqref{eqninduction} in Lemma \ref{la:quadFormsignature}.

\section{Tools for indefinite theta functions}\label{sectools}

\subsection{Generalized error functions}

We follow the setup in \cite{Na}, except that we flip the signature of the quadratic form $Q$ (from $(r,s)$ to $(s,r)$) and that the relation between $Q$ and $B$ differs by a factor of $2$. Let $B$ be a symmetric, indefinite bilinear form of signature $(r,s)$ on $\R^n$ such that the associated quadratic form $Q(\bm x)=\frac12B(\bm x,\bm x)$ is integral on $\Z^n$. The bilinear form maps as $B(\bm x,\bm y)=\bm x^TA\bm y$ with a symmetric, indefinite matrix $A\in\R^{n\times n}$. For $C=(\bm{c_1},\dots,\bm{c_s})\in\R^{n\times s}$ spanning a negative definite $s$-dimensional subspace $\sangle C:=\sangle{\bm{c_1},\dots,\bm{c_s}}$ and $\bm x\in\C^n$, we define the {\it generalized error function}
\begin{align}\label{DefGenError}
	E_{Q,C} ( \bm{x} ) := \int_{\langle C\rangle} \sgn  (B (C,\bm{y}+\bm{x} ) ) e^{2 \pi Q\left( \bm{y}\right)  } \bm{d} \bm{y}
\end{align}
(introduced in \cite{ABMP})
with $B(C,\bm y):=C^TA\bm y$, $\sgn(\bm x):=\prod_{j=1}^s\sgn(x_j)$, 
and with the measure $\bm{dy}$ normalized such that $\int_{\sangle{C}}e^{2\pi Q(\bm y)}\bm d\bm y=1$. Note that $E_{Q,C}$ only depends on the projection of $\bm x$ to $\sangle C$. We write $E_{Q,\emptyset}(\bm x):=1$ where $\emptyset$ is the $0\times0$ matrix. For any matrix $C=(\bm{c_1},\dots, \bm{c_s})$, let $C_S:=(\bm{c_{k_1}},\bm{c_{k_2}},\dots, \bm{c_{k_{\lvert S\rvert}}})$ be the matrix consisting of the columns indexed by $S=\{k_1,k_2,\dots, k_{\lvert S\rvert} \}\subseteq\{1,\dots,s\}$ (with $k_1<k_2<\dots< k_{\lvert S\rvert}$) and $C_\emptyset:=\emptyset$ the $0\times 0$ matrix. For $S\subseteq \{1,\dots,s\}$, we define the orthogonal decomposition of $\bm{x}\in \R^n$ and $\bm{c_j}\in \R^n$ into $\langle C_S\rangle \oplus \langle C_S\rangle^\perp$ with respect to $B$ as $\bm{x}=:\bm{x_S}+\bm{x_S^\perp}$ and  $\bm{c_j}=:\bm{c_{j,S}}+\bm{c_{j\perp S}}$. Furthermore for $S=\{k_1,k_2,\dots, k_{\lvert S\rvert} \}$ and $T\subseteq\{1,\dots,s\}$ (with $k_1<k_2<\dots< k_{\lvert S\rvert}$), we let $C_{S\perp T}:=(\bm{c_{k_1\perp T}},\bm{c_{k_2\perp T}},\dots, \bm{c_{k_{\lvert S\rvert\perp T}}}).$ For the following definition, assume that $\langle C\rangle $ is negative definite of dimension $s$ and let $D=(\bm{d_1},\dots, \bm{d_s})\in \R^{n\times s}$ such that the its columns form a dual basis to $C$ for the quadratic space $(\langle C\rangle , Q|_{\langle C\rangle })$. Define for $\bm{x}\in \R^{r+s}$ with $\prod_{j=1}^{s} B(D_{j},\bm{x})\neq 0$ the {\it complementary generalized error function} as (this expression of $M_{Q,C}$ in \cite[Proposition (3.14)(e)]{Na} has a typo in the power of the determinant)
\begin{align*}
	M_{Q,C}(\bm{x}):=\left(\frac{i}{\pi}\right)^s \det\left(C^TAC\right)^{-\frac12}\int_{\langle C\rangle-i\bm{x_C}} \frac{e^{2\pi Q(\bm{w}) +2\pi i B(\bm{x},\bm{w}) }}{ \prod_{j=1}^{s}B(D_j,\bm{w})}\bm{dw}.
\end{align*}
Taking the vectors in $C$ as the basis for $\langle C\rangle$, any $\bm{y}\in \langle C\rangle$
can be uniquely written as $	\bm{y}=\sum_{1\leq j\leq s} t_{j}\bm{c_{j}}$.
Thus, one obtains
\[
	M_{Q,C}(\bm{x})=\left(\frac{i}{\pi}\right)^s\int_{\langle C\rangle-i\bm{x_C}} \frac{e^{2\pi Q(\bm{y}) +2\pi i B(\bm{x},\bm{y}) }}{ \prod_{j=1}^{s}t_{j}}\prod_{j=1}^{s}{dt_{j}}\,,
\]
which is defined if $\bm{x_C}=\sum_{j=1}^sb_j\bm{c_j}$ with $b_j\ne0$ for $1\le j\le s$. 

\subsection{Vign\'eras differential equation}
In \cite{Vi}, Vign\'eras showed that modularity of indefinite theta functions can be obtained by checking a differential equation.\footnote{It is this differential equation that is essential for the modularity.}

\begin{theorem}[Vign\'eras]\label{Vigneras}
	Suppose that $p:\R^n\to\R$ satisfies the following conditions: 
	\begin{enumerate}[label=\textnormal{(\roman*)},leftmargin=0.65cm, labelwidth=!, labelindent=0pt]
		\item For a differential operator $D$ of order $2$ and a polynomial $R$ of degree at most $2$, $\bm{w}\mapsto p(\bm{w})e^{-{\pi}Q(\bm{w})}$ and $\bm{w}\mapsto R(\bm{w})p(\bm{w})e^{-{\pi}Q(\bm{w})}$ belong to  $L^2\left(\R^n\right)\cap L^1\left(\R^n\right)$.
		
		\item Write the Euler and Laplace operator ($\bm x:=(x_1,\dots,x_n)^T$, $\bm{\del_x}:=(\frac{\del}{\del x_1},\dots\frac{\del}{\del x_n})^T$) $\calE:=\bm x^T\bm{\del_x}$ and $\De_{A^{-1}}:=\bm{\del_x}^TA^{-1}\bm{\del_x}$. For some $\l\in\Z$ the Vign\'eras differential equation holds:
		$$
			\left(\mathcal{E}-\frac{1}{2\pi}\Delta_{A^{-1}}\right)p=\lambda p.
		$$
	\end{enumerate}
	If furthermore
	\begin{align*}
		\Theta_{Q,p}(\bm{z};\t):=\sum_{\bm{m}\in \Z^{n}} p\left(\bm{m}+\frac{\bm{y}}{v}\right) q^{ Q(\bm{m})} e^{2\pi i B(\bm{m},\bm{z})}
	\end{align*} is absolutely locally convergent, then it transforms as a Jacobi form of weight $\lambda+\frac{n}{2}$ and index $A$:
	\begin{enumerate}[label=\textnormal{(\arabic*)},leftmargin=0.65cm, labelwidth=!, labelindent=0pt]
		\item For $\bm{m},\bm{\ell}\in \Z^n$ we have
		\begin{align*}
			\Theta_{Q,p}\left(\bm{z}+\bm{m}\tau +\bm{\ell};\tau \right) =  q^{-Q(\bm{m})} e^{-2\pi i B(\bm{m},\bm{z})} \Theta_{Q,p}(\bm{z};\tau).
		\end{align*}
		
		\item For $\pabcd$ in a suitable finite index subgroup of $\SL_2(\Z)$ we have
		\[
			\TH_{Q,p}\left(\frac{\bm z}{c\t+d};\frac{a\t+b}{c\t+d}\right) = (-i(c\t+d))^{\l+\frac n2}e^\frac{2\pi iQ(\bm z)}{c\t+d}\TH_{Q,p}(\bm z;\t).
		\]
	\end{enumerate}
\end{theorem}

We use the following properties of generalized error functions that were shown in \cite{Na}.

\begin{lemma}\label{ErrorProperties}
	\ \begin{enumerate}[leftmargin=*,label=(\arabic*)]
		\item If all occurring $M_{Q,C_S}$ are defined, then we have 
		\[
			E_{Q,C}(\bm x) = \sum_{S\subseteq\{1,\dots,s\}} M_{Q,C_S}(\bm x)\sgn\left(B\left(C_{\{1,\dots,s\}\setminus S\perp S},\bm x\right)\right),\qquad |M_{Q,C}(\bm x)| < s!e^{2\pi Q(\bm{x_C})}.
		\]
		
		\item We have
		\[
			\lim_{t\to\infty} E_{Q,C}(t\bm x) = \sum_{S\subseteq\{1,\dots,s\}} E_{Q,C_S}(\bm0)\d_{\scriptscriptstyle B(C_S,\bm x)=\bm0}\sgn\left(B\left(C_{\{1,\dots,m\}\setminus S},\bm x\right)\right).
		\]
		In particular, if all components of $B( C, \bm{x} )$ are non-zero, then
		\[
			\lim_{t\to\infty} E_{Q,C}(t\bm x) = \sgn(B(C,\bm x)).
		\]
		If $C=(C_1,C_2)$ with $C_1^TAC_2=0$, then we can factor
		\begin{align*}
			E_{Q,C}( \bm{x})=E_{Q,C_1}( \bm{x})E_{Q,C_2}( \bm{x}).
		\end{align*}
		
		\item The generalized error function $E_{Q,C}$ is a $\calC^\infty$-functions and satisfies Vign\'eras's differential equation with respect to $Q$ for $\l=0$. Moreover, $M_{Q,C}$ is a $\calC^\infty$-function, where it is defined.
	\end{enumerate}
\end{lemma}

\subsection{Generalized error functions for deformed quadratic forms}

For a set of vectors $C=(\bb{c_1},\dots,\bb{c_{k}})$, we define the cone given by their normal hyperplanes as
\begin{equation*}
	R(C):=\left\{\bm{x}\in\R^{r+s}: B(\bb{c}_{\bm{j}},\bm{x})\geq 0\text{ for } 1\leq j\leq k \text{ or } B(\bb{c}_{\bm{j}},\bm{x})\leq 0 \text{ for } 1\leq j\leq k \right\} .
\end{equation*}

Hereafter we assume that the normal vectors in $C$ are rational. We list some elementary properties about quadratic spaces easily obtained by diagonalization, generalizing Lemma \ref{lemnormalvectors}.

\begin{lemma}\label{lemsimplefacts}
	Let $(V,Q)$ be a quadratic space of signature  $(r_{+},r_{-},r_{0})$. Then the following hold:
	\begin{enumerate}[label=\rm{(\arabic*)}]
		\item Assume that $r_0=0$. Then the signature $(s_+,s_-,s_0)$ of any subspace satisfies $s_++s_0\le r_+$, $s_-+s_0\le r_-$. In particular, any isotropic subspace has dimension at most $\min\{r_+,r_-\}$.
		
		\item Assume that $r_{+}=0$. The set of isotropic vectors in V forms a vector space of dimension $r_{0}$ and is given by $V\cap V^{\perp}$.
		
		\item Assume that the signature of $(V,Q)$ is $(1,s,0)$. For any non-zero vector $\bm{c}\in V$, let $\bm{c}^{\perp}$ be the orthogonal complement of $\mathbb{R}\bm{c}$. If $Q(\bm{c})>0$, then $\bm{c}^{\perp}\cap( D_{+}\cup D_{0})=\{\bm{0}\}$. If $Q(\bm{c})=0$, then $\bm{c}^{\perp}\cap ( D_{+}\cup D_{0})= \bm{c}^{\perp}\cap D_{0}=\mathbb{R}\bm{c}$. If $Q(\bm{c})<0$, then $\bm{c}^{\perp}\cap  D_{+}\neq \{\bm{0}\} $.
	\end{enumerate}
\end{lemma}

Results on modularity of indefinite theta functions are established for the following special simplicial cones,
and then generalized to more general polyhedron cones.

\begin{assumption}\label{assumptionsI}
Let $(V,Q)$ have signature $(1,s)$ with $s\geq 1$.
Let $C$ be a collection of non-zero vectors.
Assume the following conditions: 
\begin{enumerate}[label=\textnormal{(\arabic*)}]
\item The cone $R(C)\subseteq D_{+}\cup D_{0}$ has dimension $s+1$.
\item The cone $R(C)$ is simplicial, i.e., its cross section is a simplex of dimension $s$.
\item The cone $R(C)\cap  D_{0}$ consists of at most one isotropic ray.
\end{enumerate}
\end{assumption}

We have the  following result regarding the signature of $\langle C_{T}\rangle$ for $T\subseteq C$.

\begin{lemma}\label{lemCTnegativesemidefinite}
	Suppose Assumption \ref{assumptionsI} {\rm(1)} holds. For $T\subseteq C$, exactly one of the following holds: 
	\begin{enumerate}[label=\textnormal{(\arabic*)}]
		\item We have $\langle C_{T}\rangle=V$.
		
		\item The space $\langle C_{T}\rangle$ is negative definite.
		
		\item The space $\langle C_{T}\rangle$ has signature $(0, \dim (\langle C_{T}\rangle)-1,1)$. This is equivalent to $\langle C_{T} \rangle\cap D_0\neq \{\bm{0}\}$ for $\langle C_{T}\rangle\neq V$.
	\end{enumerate}
	Assume furthermore $C=(\bm{c_1},\cdots,\bm{c_s},\bm{c_0})$. Then in {\rm(3)}, under the additional Assumption \ref{assumptionsI} {\rm(2)}, one has $\dim(\sangle{C_T})\ge2$. Under the additional Assumption \ref{assumptionsI} {\rm(2)}, {\rm(3)}, there is a unique index set $T$ having the above signature, with $\dim(\sangle{C_T})=s$.
\end{lemma}

\begin{proof}
	Consider the face $\sangle{C_T}^\perp\cap R(C)$ of $R(C)$. Assumption \ref{assumptionsI} (1) gives that this face is contained in $D_+\cup D_0$. Suppose that $\sangle{C_T}^\perp\cap R(C)\cap D_+\ne\emptyset$. Take any non-zero $\bm u\in\sangle{C_T}^\perp\cap D_+$. Then $\bm u^\perp$ has signature $(0,s,0)$. Thus $\sangle{C_T}\subseteq(\sangle{C_T}^\perp)^\perp\subseteq\bm u^\perp$ is negative definite. Suppose that $\sangle{C_T}^\perp\cap R(C)\subseteq D_{0}$. If $\sangle{C_T}^\perp\cap R(C)=\{\bm{0}\}$, then $\sangle{C_T}=V$. Otherwise, as in the proof of Proposition \ref{propisotricconestructure}, $\sangle{C_T}^\perp\cap D_0$ is a finite union of isotropic rays. Since $\sangle{C_T}^\perp\cap R(C)$ is a face of $R(C)$ contained in this union, it follows that the face $\sangle{C_T}^\perp\cap R(C)$ is a single isotropic ray $\R\bm{f_T}$, that is,
	\begin{equation}\label{eqnisotropicfaceinpolyhedron}
		\langle C_{T} \rangle^{\perp}\cap R(C)=\mathbb{R}\bm{f_{T}}
		.
	\end{equation}
	By Lemma \ref{lemsimplefacts} (1) we obtain that $\langle C_{T}\rangle $ is negative semidefinite, since otherwise any positive vector $\bm{u}\in \langle C_{T}\rangle $ would give rise to a rank two quadratic space
	$\langle \bm{f}, \bm{u}\rangle$ of signature $(1,0,1)$.
	Lemma \ref{lemsimplefacts} (2) then  gives that $\langle C_{T} \rangle\cap D_0$ is a vector space satisfying
	\begin{equation}\label{eqnisotropicfaceinpolyhedronasmaximal}
		\langle C_{T} \rangle\cap D_0=\langle C_{T} \rangle\cap \langle C_{T} \rangle^{\perp}.
	\end{equation}
	Combining   \eqref{eqnisotropicfaceinpolyhedron} and  \eqref{eqnisotropicfaceinpolyhedronasmaximal}, we have
	\begin{equation}\label{eqnCTisotropicfaceinpolyhedron}
		\langle C_{T} \rangle\cap D_0=\langle C_{T} \rangle\cap \langle C_{T} \rangle^{\perp}\cap D_0
		=\langle C_{T} \rangle\cap \mathbb{R}\bm{f_T}.
	\end{equation}
	If $\sangle{C_T}\cap D_0=\{\bm0\}$, then the negative semidefinite space $\sangle{C_T}$ is actually negative definite. Otherwise, $\sangle{C_T}\cap D_0=\R\bm{f_T}$. Combining  \eqref{eqnisotropicfaceinpolyhedron}, \eqref{eqnisotropicfaceinpolyhedronasmaximal}, and \eqref{eqnCTisotropicfaceinpolyhedron} we have
	\begin{equation*}
		\mathbb{R}\bm{f_T}=  \langle C_{T}\rangle^{\perp}\cap R(C)
		= \langle C_{T}\rangle\cap D_0
		=
		\langle C_{T}\rangle\cap \langle C_{T}\rangle^{\perp}
		.
	\end{equation*}
	The fact
	$\bm{f_{T}}\in  \langle C_{T} \rangle^{\perp}$ yields that the signature of $\langle C_{T}\rangle$ is $(0, \dim (\langle C_{T}\rangle)-1,1)$
	by Lemma \ref{lemsimplefacts} (1).
	
	Conversely, if $\langle C_{T} \rangle\neq V$  and
	$\langle C_{T} \rangle\cap D_0\neq \{\bm{0}\}$, then $\langle C_{T} \rangle$
	can not be negative definite and by the above discussions has the desired signature.
	
	Consider the case that $\langle C_{T}\rangle$ has signature $(0, \dim (\langle C_{T}\rangle)-1,1)$.
	First, the simplicial condition in Assumption \ref{assumptionsI} (2) gives that
	\begin{equation}\label{eqndimensionoffaceinsimplicialpolyhedron}
		\dim\,\left(
		\langle C_{T} \rangle^{\perp}\right)=\dim\,(\langle C_{T} \rangle^{\perp}\cap R(C)).
	\end{equation}
	By \eqref{eqnisotropicfaceinpolyhedron}, \eqref{eqndimensionoffaceinsimplicialpolyhedron}, and
	dimension reasons, we have
	\begin{equation*}
		\langle C_{T} \rangle^{\perp}=\langle C_{T} \rangle^{\perp}\cap R(C)=\mathbb{R}\bm{f_{T}}
		.
	\end{equation*}
	In particular, $\dim(\sangle{C_T}^\perp)=1$, $\dim(\sangle{C_T})=s$. Further imposing Assumption \ref{assumptionsI} (2), the vectors in $C$ are linearly independent and satisfy $\bm{c_j}^\perp\cap D_+\ne\{\bm0\}$, $0\le j\le s$. By Lemma \ref{lemsimplefacts} (3) we obtain that $Q(\bm{c_j})<0$. This gives $\dim(\sangle{C_T})\ge2$. Further imposing Assumption \ref{assumptionsI} (2), (3), there exists a unique subset $T$ of $\{1,\dots,s\}\cup\{0\}$ satisfying the above dimensionality condition and \eqref{eqnisotropicfaceinpolyhedron}.
\end{proof}

Under Assumption \ref{assumptionsI} (1), (2), (3) and $C=(\bm{c_1},\cdots,\bm{c_s},\bm{c_0})$, we can choose the unique subset in Lemma \ref{lemCTnegativesemidefinite} (3) to be $S=\{1,\dots,s\}$ by possibly relabelling indices. 

If $\sangle{C_T}$ fails to be negative definite, then we need substitutes for $E_{Q,C_T}$ and $M_{Q,C_T}$. For $S=\{1,\dots,s\}$, let $U_S:=\sangle{C_S}\cap D_0=\sangle{C_S}\cap\R\bm{f_S}$. Since $\bm{f_S}\in\sangle{C_S}^\perp$, $Q$ induces a negative definite quadratic form on the quotient vector space $\sangle{C_S}/U_S$. Taking a lift $W_S\subseteq\sangle{C_S}$ of it along the quotient map $\pi_S:\sangle{C_S}\to\sangle{C_S}/U_S$, one has an orthogonal splitting
\begin{equation}\label{eqnCSdecomposition}
	\langle C_{S}\rangle=
	U_{S}\oplus  W_{S}.
\end{equation}
A choice of $E_{-}$ in Lemma \ref{lemnormalvectors} provides an example for such a splitting.
We also take a splitting of 
\begin{equation}\label{eqnVdecomposition}
	V=\langle C_{S}\rangle\oplus \mathbb{R}\bm{\xi}\,,\quad \bm{ \xi}\in V\setminus \langle C_{S}\rangle.
\end{equation}
Combining the above, we obtain a vector space decomposition
\begin{equation}\label{eqnVfinerdecomposition}
	V=
	\langle C_{S}\rangle\oplus \mathbb{R}\bm{\xi}=
	U_{S}\oplus  W_{S}\oplus \mathbb{R}\bm{\xi}.
\end{equation}
Consider
a new quadratic form $Q_{t}$ on $V$, with $||\bm{f_S}||=1$
and $t>0$ sufficiently small
\begin{equation}\label{eqnQt}
	Q_{t}(\bm{\lambda})=Q(\bm{\lambda})
	- \frac t2||\bm{\lambda_{U_{S}}}||^2.
\end{equation}
The component $\bm{\lambda_{U_{S}}}$ is defined by using the decomposition provided in \eqref{eqnVfinerdecomposition}. By construction of $Q_{t}$,  $U_{S}\oplus W_{S}$ is an orthogonal decomposition of $\langle C_{S} \rangle$ and $(\langle C_{S} \rangle, Q_{t}|_{\langle C_{S} \rangle})$
is negative definite.

\begin{lemma}\label{lemCTnegativedefinite}
Suppose Assumption \ref{assumptionsI} holds with $C=(\bm{c_1},\cdots, \bm{c_{s}},\bm{c_{0}})$.
For $T\subsetneq \{0,\dots,s\}$, $Q_{t}|_{\langle C_T\rangle}$
is negative definite.
\end{lemma}

\begin{proof}
	We distinguish two cases. If $(\sangle{C_T},Q|_{\sangle{C_T}})$ is negative definite, then $Q(\bm\l)<0$ for $\bm\l\in\sangle{C_T}\setminus\{\bm0\}$. Thus by \eqref{eqnQt}, $Q_t|_{\sangle{C_T}}$ is negative definite.
	If $(\sangle{C_T},Q|_{\sangle{C_T}})$ has signature $(0,|T|-1,1)$, then, by Lemma \ref{lemCTnegativesemidefinite},
	$\sangle{C_T}\subseteq\sangle{C_S}$. Again $Q_t|_{\sangle{C_T}}$ is negative definite since $Q_t|_{\sangle{C_S}}$ is.
\end{proof}

Applying the same expressions as in the negative definite case, we obtain
$E_{Q_{t},C_{T}}, M_{Q_{t},C_{T}}$ for any $T\subsetneq \{0,\dots,s\}$.
Then all statements in Lemma \ref{ErrorProperties} hold.
In particular, we have
\begin{equation}\label{eqnMestimate}
|M_{Q_{t},C_T}(\bm{\lambda})|\leq
K\, e^{2\pi Q_{t}\left(\bm{\lambda_{C_{T}}}^{[t]}\right)},
\end{equation}
for some constant $K$ that only depends on $T, Q$ and $t$.
Hereafter the superscript $[t]$ represents a construction with respect to $Q_{t}$. 

For use below, we also discuss orthogonal decompositions and projection maps with respect to $Q$ and $Q_{t}$.
Projection maps
are defined using splittings, in particular projection maps
 to isotropic lines can not be defined directly using the quadratic forms.
One can choose
$\bm{\xi} $ in \eqref{eqnVdecomposition} such that
\begin{equation}\label{eqnxiforsplitting}
\bm{\xi}\in W_{S}^{\perp} \,,\quad
Q(\bm{\xi})=0\,,\quad B(\bm{\xi},\bm{f_S})=1.
\end{equation}
This is possible since $W_S^\perp$ has signature $(1,1)$ and $\bm{f_S}\in W_S^\perp$ is isotropic. Then, by construction, the orthogonal decompositions $W_S\oplus W_S^\perp$ in $Q$ and $Q_t$ are the same, only the orthogonal decompositions inside $W_S^\perp=\sangle{\bm{f_S},\bm\xi}$ are different. For use below, we study the projection to $\sangle{C_T}^\perp$.

\begin{lemma}\label{lemorthogonalprojectionCT}
	Suppose Assumption \ref{assumptionsI} holds with $C=(\bm{c_1},\cdots, \bm{c_{s}},\bm{c_{0}})$.
	Assume $T\subsetneq\{0,\dots,s\}$.
	\begin{enumerate}[label=(\arabic*)]
		\item If $\sangle{C_T}$ is negative definite, then, for any $\bm\l\in V$, one has $ \bm{\l_{C_T}}=\bm{\l_{C_T}}^{[t]}, \bm\l_{\perp\bm{C_T}}=\bm\l_{\perp\bm{C_T}}^{[t]}$.
		
		\item If $\sangle{C_T}$ is negative semi-definite, then, for any $\bm\l\in V$, one has
		$Q(\bm{\l_{C_T}}^{[t]})=Q(\bm{\l_{C_T}})$.
	\end{enumerate}
\end{lemma}

\begin{proof}
	\begin{enumerate}[label=(\arabic*),wide,labelindent=0pt]
		\item If $Q|_{\sangle{C_T}}$ is negative definite, then $V=\sangle{C_T}\oplus\sangle{C_T}^\perp$, $U_T:=\sangle{C_T}\cap D_0=\{\bm0\}$. By \eqref{eqnQt}
		\[
			B_t\left(\bm u,\bm\l_{\perp\bm{C_T}}\right) = B\left(\bm u,\bm\l_{\perp\bm{C_T}}\right) = 0 \text{ for } \bm u\in\sangle{C_T}.
		\]
		From the orthogonality we also have
		\[
			B_t\left(\bm{\l_{C_T}}^{[t]},\bm\l_{\perp\bm{C_T}}^{[t]}\right) = 0 = B_t\left(\bm{\l_{C_T}},\bm\l_{\perp\bm{C_T}}\right).
		\]
		Using $\bm\l=\bm{\l_{C_T}}+\bm\l_{\perp\bm{C_T}}=\bm{\l_{C_T}}^{[t]}+\bm\l_{\perp\bm{C_T}}^{[t]}$, the above two relations imply that
		\begin{align*}
			B_{t}\left(\bm{\lambda_{C_{T}}}^{[t]}-\bm{\lambda_{C_{T}}},\bm{\lambda_{C_{T}}}^{[t]}-\bm{\lambda_{C_{T}}}\right) 
			&= B\left(\bm{\lambda_{C_{T}}}^{[t]}-\bm{\lambda_{C_{T}}}, \bm{\lambda}_{\perp \bm {C_{T}}}\right) - B_{t}\left(\bm{\lambda_{C_{T}}}^{[t]}-\bm{\lambda_{C_{T}}}, \bm{\lambda}^{[t]}_{\perp \bm {C_{T}}}\right) = 0.
		\end{align*}
		Since $Q_{t}$ is negative definite on $\langle C_T\rangle $, this proves the desired claims.
		
		\item By Lemma \ref{lemCTnegativesemidefinite}, one has $T=S$ and $U_S\subseteq\sangle{C_T}$. By \eqref{eqnCSdecomposition} one has orthogonal decompositions
		\[
			\langle C_{T} \rangle=U_{S}\oplus (\langle C_{T}\rangle\cap W_{S})\,,\quad \langle C_{T} \rangle^{\perp}=U_{S}\oplus  (\langle C_{T}\rangle\cap W_{S})^{\perp}_{W_{S}}\,,
		\]
		where the notation $(\langle C_{T}\rangle\cap W_{S})^{\perp}_{W_{S}}$ stands for the orthogonal complement of $\langle C_{T}\rangle\cap W_{S}$ inside the negative definite space $W_{S}$. The orthogonal projections to $\langle C_{T}\rangle,\langle C_{T}\rangle^{\perp} $ are then provided by the following vector space decomposition induced from \eqref{eqnVfinerdecomposition}
		\begin{equation*}
			V= U_{S}\oplus (\langle C_{T}\rangle\cap W_{S}) \oplus  (\langle C_{T}\rangle\cap W_{S})^{\perp}_{W_{S}} \oplus \mathbb{R}\bm{\xi}.
		\end{equation*}
		Concretely, write
		\begin{equation}\label{eqnVevenfinercoordinates}
			\bm{\lambda}=a\bm{f_S}+\bm{v}+\bm{w}+b \bm{\xi}\in U_{S}\oplus (\langle C_{T}\rangle\cap W_{S}) \oplus  (\langle C_{T}\rangle\cap W_{S})^{\perp}_{W_{S}} \oplus \mathbb{R}\bm{\xi}.
		\end{equation}
		Then we have
		\begin{equation}
			\bm{\l_{C_T}} = a\bm{f_S} + \bm v,
			\quad
			\bm\l_{\perp\bm{C_T}} = a\bm{f_S} + \bm w,
			\quad
			\label{eqnVevenfinerdecompositioncoordinates}
			\bm{\l_{C_T}}^{[t]} = \left(a-\frac bt\right)\bm{f_S} + \bm v,
			\quad
			\bm\l_{\perp\bm{C_T}}^{[t]} = \bm w + b\!\left(\bm\xi+\frac1t\bm{f_S}\right)\!.
		\end{equation}
		The result follows. \qedhere
	\end{enumerate}
\end{proof}

The projection to  $\sangle{C_T}^{\perp^{[t]}}$ under $Q_t$ can be made more concrete. Employing Lemma \ref{lemCTnegativedefinite}, $Q_t(\bm{c_j})\ne0$ for $0\le j\le s$. Let $(\bm{d_j})_{0\le j\le s}$ be the dual basis of $(\bm{c_j})_{0\le j\le s}$ with respect to $Q_t$.

\begin{lemma}\label{lemCorthogonalprojection}
	Suppose Assumption \ref{assumptionsI} holds with $C=(\bm{c_1},\dots,\bm{c_s},\bm{c_0})$. For any $T\subsetneq\{0,\dots,s\}$, let $T^c:=\{0,\dots,s\}\setminus T$. The projection $\bm\l_{\perp\bm{C_T}}^{[t]}$, defined using $Q_t$, is the unique vector $\bm u$ satisfying
	\begin{equation}\label{eqnnormalcoordinatesoffaces}
		B_{t}(\bm{u}, \bm{c_j})=0\,,~j\in T\,,\quad
		B_{t}(\bm{u}, \bm{d_k})=B_{t}(\bm{\lambda}, \bm{d_k})\,, ~ k\in T^{c}.
	\end{equation}
\end{lemma}

\begin{proof}
First, using Lemma \ref{lemCTnegativedefinite},
$\bm{\lambda}^{[t]}_{\perp\bm{C_{T}}} $ can be defined and satisfies $\bm{\lambda}^{[t]}_{\perp\bm{C_{T}}} =\bm{\lambda}-\bm{\lambda_{C_{T}}}^{[t]}$.
We write $\langle D_{T^{c}}\rangle=\langle \bm{d_{k}}:{k\in T^{c}} \rangle $ and suppose that $\bm{\lambda}\in \langle C_{T}\rangle\cap \langle D_{T^{c}}\rangle$.
Then $\bm{\lambda}\in \langle C_{T}\rangle\cap \langle C_{T}\rangle^{\perp^{[t]}}$, in particular $Q_{t}(\bm{\lambda})=0$.
 By Lemma \ref{lemCTnegativedefinite} we have that $\bm{\lambda}=\bm{0}$. For dimension reasons, $C_{T}\cup  D_{T^{c}}$ forms a basis of $V$ and
$V=\langle C_{T}\rangle\oplus\langle D_{T^{c}}\rangle$.
Therefore, since $B_{t}$ is non-degenerate there exists a unique vector $\bm{u}$ satisfying \eqref{eqnnormalcoordinatesoffaces}.
Since $\bm{\lambda}^{[t]}_{\perp\bm{C_{T}}} $ satisfies \eqref{eqnnormalcoordinatesoffaces}, the desired claim follows.
\end{proof}

\section{Modularity of counting functions of convex planar polygons}\label{secmodularity}

In this section, we first establish some basic estimates
for functions appearing in indefinite theta functions of signature $(1,s)$.
Then we prove the modularity for them in the case of simplicial polyhedron cones and non-simplicial ones.

\subsection{Indefinite theta functions of signature $(1,s)$}

We define the theta function for a quadratic form $Q$ of signature $(1,s)$ and a cone given by $C=(\bm{c_1},\cdots, \bm{c_{k}})$  as
\begin{equation*}
	\Theta_{Q,p_C}(\bb{z};\tau):=\sum_{\bb{n}\in \Z^{s+1}} p_C\left(\bb{n}+\frac{\bb{y}}{v}\right)q^{Q(\bb{n})} e^{2\pi i B(\bb{n},\bb{z})}\,,
\end{equation*}
where
\begin{align}\label{eqnPCdefinition}
	p_C\left(\bb{x}\right):=&\prod_{j=1}^{k}{1+\sgn\left(B\left(\bb{c_j},\bb{x}\right)\right)\over 2}
	-(-1)^{k} \prod_{j=1}^{k}{1-\sgn\left(B\left(\bb{c_j},\bb{x}\right)\right)\over 2}\nonumber\\
	=&2^{-k}\sum_{S\subseteq \{1,\dots,k\}} \left(1-(-1)^{k+\lvert S\rvert} \right)  \sgn(B(C_S, \bm{x})).
\end{align}
We furthermore let
\begin{equation}\label{eqnPChatdefinition}
	\wh{p}_C(\bm x) := 2^{-k}\sum_{S\subseteq\{1,\dots,k\}} \left(1-(-1)^{k+\lvert S\rvert} \right) E_{Q,C_S}(\bm x).
\end{equation}

\begin{lemma}\label{completionProperties}
	Let $C:=(\bb{c_1},\dots,\bb{c_{s+1}})$ such that $R(C)\subseteq D_+\cup \{\bm{0}\}$ is simplical of dimension $s+1$. For $T\subsetneq \{1,\dots,s+1\}$,
we have:
	\begin{enumerate}[leftmargin=*,label=(\arabic*)]
		\item The subspace $\langle C_T\rangle $ is negative definite of dimension $\lvert T\rvert $.
		
		\item We have
		\begin{align*}
			\wh{p}_C(\bm{x})=\sum_{T\subsetneq \{1,\dots,s+1\}} 2^{-k} M_{Q,C_T}(\bm{x})p_{C_{T^{c}\perp C_{T}}}(\bm{x})\,,
		\end{align*}
		where $p_{C_{T^{c}\perp C_{T}}}(\bm{x})$ vanishes unless $\bm{x}\in R(C_{T^{c}\perp C_{T}})$ and  $\lvert p_{C_{T^{c}\perp C_{T}}}(\bm{x})\rvert \leq 2^{s+1-\lvert T\rvert}$.
		
		\item For some $\ell >0$, we have
		\begin{align*}
			\left\lvert \wh{p}_C(\bm{x})  e^{-\pi Q\left(\bm{x}\right)} \right\rvert \leq   s! \, e^{-\pi \ell \bm{x}^T\bm{x}},\qquad 	\left\lvert p_C(\bm{x})  e^{-\pi Q\left(\bm{x}\right)} \right\rvert \leq  e^{-\pi \ell \bm{x}^T\bm{x}}.
		\end{align*}
		
		\item The function $\widehat{p}_C$ is a $\mathcal{C}^\infty$-function and satisfies Vign\'eras differential equation with $\lambda=0$.
	\end{enumerate}
\end{lemma}

\begin{proof}
	(1)
	The claim follows from Lemma \ref{lemCTnegativesemidefinite}. \\
	(2) We apply Lemma \ref{ErrorProperties} and write $T^c:=\{1,\dots,s+1\}\setminus T$ to obtain
	\begin{align*}
		\wh{p}_C(\bm{x})= \sum_{T\subsetneq \{1,\dots,s+1\}} {1\over 2^{|T|}} M_{Q,C_T}(\bm{x})p_{C_{T^{c}\perp C_{T}}}(\bm{x}).
	\end{align*}	
	Noting that $B(\bb{a_{\perp C_T}}, \bb{b})=B(\bb{a_{\perp C_T}}, \bb{b_{\perp C_T}})=B(\bb{a}, \bb{b_{\perp C_T}})$ we obtain that $p_{C_{T^{c}\perp C_{T}}}(\bm{x})=p_{C_{T^{c}}}$ $(\bm{x}_{\perp C_{T}})=0$ unless $\bm{x}\in R(C_{T^{c}\perp C_{T}})$ (equivalently, $\bm{x_{\perp C_{T}}}\in R(C_{T^{c}})$).\\
	(3) For $\bm x$ such that $M_{Q,C_T}(\bm x)$ is well-defined, by Lemma \ref{ErrorProperties} we can bound
	\begin{align*}
		\left|M_{Q,C_T}(\bm{x}) e^{-\pi Q(\bm{x})} \right|   &< |T|! e^{2\pi Q\left(\bm{x}_{C_T}\right)} e^{-\pi Q\left(\bm{x}\right)} = \lvert T\rvert! e^{-\pi Q_T^+\left(\bm{x}\right)}\,,
	\end{align*}
	where the quadratic form with its sign flipped on the span of $C_T$ is
	\begin{align*}
		Q_T^+(\bm{x}):= Q\left(\bm{x_{\perp C_{T}}}\right)- Q(\bm{x_{C_T}})= Q(\bm{x})- 2Q(\bm{x_{C_T}}).
	\end{align*}

	Let $\bm x\in R(C_{T^c\perp T})\setminus\{\bm0\}$, which satisfies $\bm{x_{\perp C_T}}\in R(C_{T^c})$.  Since $B(\bm{c_j},\bm{x_{\perp C_T}})=0$ for $j\in T$, we further have $\bm{x_{\perp C_T}}\in R(C)\subseteq D_+\cup\{\bm0\}$. Since $Q$ is negative definite on the span of $C_T$ by (1), 
	\begin{align*}
		Q_T^+(\bm{x})=Q(\bm{x_{\perp C_T}})-Q(\bm{x_{C_T}})> 0\,,
	\end{align*}
	where the inequality is strict because $\bm{x_{C_T}}\ne0$ or $\bm{x_{\perp C_T}}\ne0$. Using compactness, we let $\ell_T:=\min(\{Q_T^+(\bm x):\bm x\in R(C_{T^c\perp T})\text{ with }\bm x^T\bm x=1\})>0$. Then for $\bm x\in R(C_{T^c\perp T})\setminus\{\bm0\}$ we have
	\begin{align*}
		Q_T^+(\bb{x}) =\bb{x}^T\bb{x} Q_T^+\left(\frac{\bb{x}}{\sqrt{\bb{x}^T\bb{x}}}\right)\geq \ell_T \bb{x}^T\bb{x}.
	\end{align*}
	Writing $\ell:=\min(\ell_T:T\subsetneq\{1,\dots,s+1\})>0$ and using part (2), we can bound
	\begin{align*}
		\left\lvert \wh{p}_C(\bm{x}) e^{-\pi Q\left(\bm{x}\right)} \right\rvert  \leq  \sum_{T\subsetneq\{1,\dots,s+1\}} {\lvert T\rvert! \over 2^{|T|}}
		\left\lvert  p_{C_{T^{c}\perp C_{T}}}(\bm{x})\right\rvert  e^{-\pi Q_T^+\left(\bm{x}\right)}
		\leq  \sum_{T\subsetneq\{1,\dots,s+1\}} {\lvert T\rvert! \over 2^{|T|}} e^{-\pi \ell \bm{x}^T\bm{x}}
		\leq   s!
		e^{-\pi \ell \bm{x}^T\bm{x}}
	\end{align*}
	for $\bm x$ such that all occurring $M_{Q,C_T}(\bm x)$ are well-defined. Since $\widehat{p}_C$ is continuous and we prove the bound for a dense subset, it holds for all $\bm{x}\in \R^{s+1}$. The bound for $p_C$ is obtained by considering only the term for $T=\emptyset $ in the above computation.\\
	(4) Since $\wh p_C$ is a linear combination of $E_{Q,C_S}$, this is a direct consequence of Lemma \ref{ErrorProperties} (3).
\end{proof}

\subsection{Modularity of indefinite theta functions of signature $(1,s)$: simplicial case}

In this subsection, we study modularity of $\TH_{Q,\wh p_C}$ if $R(C)$ is a simplicial polyhedron.

\subsubsection{Modularity of indefinite theta functions for simplicial cones}

\begin{theorem}\label{thm:Modularitysimplicialwithnoisotropicray}
	Let $C=(\bm{c_1},\dots,\bm{c_{s+1}})$ be a set of normal vectors such that $R(C)\subseteq D_+\cup\{\bm0\}\subseteq\R^{s+1}$ is simplical of dimension $s+1$. Then the series $\TH_{Q,p_C}$ is absolutely convergent and defines a meromorphic function in $\t$ which has the completion
	\begin{align*}
		\Theta_{Q,\wh{p}_{C}}(\bb z;\t) := \sum_{\bb n\in\Z^{s+1}} \wh{p}_{C} \left(\left(\bb n+\frac{\bb y}{v}\right)\sqrt{2v}\right)
		 q^{Q(\bb n)} e^{2\pi iB(\bb n,\bb z)}.
	\end{align*}
	This means that $\Theta_{Q,\wh{p}_C}$ transforms as a Jacobi form of weight $\frac{s+1}{2}$ and index $A$, and one can recover the ``holomorphic part'' $\TH_{Q,p_C}$ as (note that we obtain $\Theta_{Q,\wh{p}_C}$ by setting $w=\overline{\tau}$ instead of taking the limit)
	\begin{align*}
		\Theta_{Q,p_C}(\bm{z};\tau)=\sum_{\bb n\in\Z^{s+1}} \lim_{w \to -i\infty}\wh{p}_C \left(\left(\bb n+\frac{2\bb y}{\tau-w}\right)\sqrt{\tau-w}\right) q^{Q(\bb n)} e^{2\pi iB(\bb n,\bb z)}
	\end{align*}
	if $B(\bm{c_j},\bm n+\frac{\bm y}{v})\ne0$ for all $1\le j\le s+1$, $\bm n\in\Z^{s+1}$.
\end{theorem}

\begin{proof}
	By Lemma \ref{completionProperties} (3), $\TH_{Q,p_C}$ and $\TH_{Q,\wh p_C}$ are bounded by a theta function for a positive definite quadratic form and thus converge compactly in $\C^{s+1}\times\H$. By Lemma \ref{completionProperties} (3), (4), $\wh p_C$ satisfies the conditions of Theorem \ref{Vigneras} and thus $\TH_{Q,\wh p_C}$ transforms as a Jacobi form of weight $\frac{s+1}{2}$, index $A$.
\end{proof}

The rest of this subsubsection is devoted to proving the following  generalization of Theorem \ref{thm:Modularitysimplicialwithnoisotropicray}.

\begin{theorem}\label{thm:Modularitysimplicialwithisotropicray}
	Suppose Assumption \ref{assumptionsI} holds with $C=(\bm{c_1},\cdots,\bm{c_s},\bm{c_0})$. Then the series $\TH_{Q,p_C}$ is compactly convergent and defines a meromorphic function in $\t$ with modular completion
	\[
		\TH_{Q,\wh p_{C,t}}(\bm z;\t) := \sum_{\bm n\in\Z^{s+1}} \wh p_{C,t}\left(\left(\bm n+\frac{\bm y}{v}\right)\sqrt{2v}\right)q^{Q(\bm n)}e^{2\pi iB(\bm n,\bm z)}.
	\]
	Here $\wh p_{C,t}$ is defined as in \eqref{eqnPChatdefinition} with quadratic form $Q_t$ given in \eqref{eqnQt}. This means that $\TH_{Q,\wh p_{C,t}}$ transforms as a Jacobi form of weight $\frac{s+1}{2}$ and index $A$ and one can recover $\TH_{Q,p_C}$ as
	\begin{align*}
		\Theta_{Q,p_C}(\bm{z};\tau)=\lim_{t\rightarrow 0^+}\sum_{\bb n\in\Z^{s+1}} \lim_{w \to -i\infty}\wh{p}_{C,t} \left(\left(\bb n+\frac{2\bb y}{\tau-w}\right)\sqrt{\tau-w}\right) q^{Q(\bb n)} e^{2\pi iB(\bb n,\bb z)}.
	\end{align*}
\end{theorem}

We first investigate convergence of $\TH_{Q,\wh p_{C,t}}$. Following the computations in Lemmas \ref{ErrorProperties} and \ref{completionProperties}, we need to prove convergence of \begin{align}\label{eqnreducedseries}
\sum_{\bb \lambda\in\Z^{s+1}+\bm \alpha}
	e^{2\pi i B(\bb \lambda,\bb \beta)}e^{2\pi i u Q(\bm \lambda)}
	e^{-\pi   Q\left(  \sqrt{2v}\bm \lambda\right)}
	M_{Q_{t},C_T}\left( \sqrt{2v}\bb \lambda\right)
	p_{C_{T^{c}\perp^{[t]} C_{T}}}(\bb \lambda) .
\end{align}
In the above, one has
$
	p_{C_{T^{c}\perp^{[t]} C_{T}}}(\bb \lambda)=
	p_{C_{T^{c}}}(\bm{\lambda}^{[t]}_{\perp\bm{C_{T}}} )\,,
$
where the bilinear form in the sign function and orthogonal projection of the function  $ p_{C_{T^{c}\perp^{[t]} C_{T}}}$ is $B_{t}$.
We write
\begin{equation}\label{eqnMprewriting}
	e^{-\pi Q(\bm\l)}M_{Q_t,C_T}(\bm\l)p_{C_{T^c\perp^{[t]}C_T}}(\bm\l) = e^{-\pi Q(\bm\l)}e^{2\pi Q_t\left(\bm{\l_{C_T}}^{[t]}\right)}p_{C_{T^c\perp^{[t]}C_T}}(\bm\l)e^{-2\pi Q_t\left(\bm{\l_{C_T}}^{[t]}\right)}M_{Q_t,C_T}(\bm\l).
\end{equation}
According to \eqref{eqnMestimate} the quantity $|e^{-2\pi Q_t(\bm{\l_{C_T}^{[t]}})}M_{Q_t,C_T}(\bm\l)|$ is bounded from above by a constant.

We next estimate the function $e^{-\pi Q(\bm\l)}e^{2\pi Q_t(\bm{\l_{C_T}}^{[t]})}$ on the support of $p_{C_{T^c\perp^{[t]}C_T}}(\bm\l)$ whose closure is the polyhedron $R_t(C_{T^c\perp C_T})$. For any $\bm\l\in R_t(C_{T^c\perp C_T})$, one has $B_t(\bm{c_j},\bm\l^{[t]}_{\perp\bm{C_T}})=0$, $j\in T$, and
\begin{equation}\label{eqnequationforpolyhedron}
	B_{t}\left(\bm{c_k},\bm{\lambda}^{[t]}_{\perp\bm{C_{T}}} \right)\geq 0\text{ for } k\in T^{c}\,~~
	\text{or}~~
	B_{t}\left(\bm{c_k},\bm{\lambda}^{[t]}_{\perp\bm{C_{T}}} \right)\leq 0 \text{ for } k\in T^{c}.
\end{equation}

\begin{lemma}\label{lempositivityininterior}
	Suppose Assumption \ref{assumptionsI} holds with $C=(\bm{c_1},\cdots, \bm{c_{s}},\bm{c_{0}})$.
Let $Q_{T}^{+}(  \bm {\lambda}):=
Q(  \bm {\lambda})-2Q_{t}(  \bm {\lambda}^{[t]}_{\bm {C_{T}}})$.
For $t$ sufficiently small, $Q_{T}^{+}$ is positive
on $R^{\circ}_{t}(C_{T^{c}\perp C_{T}})\setminus \mathbb{R}\bm{f_S}$.
\end{lemma}

\begin{proof}
	Again we work with a cross section of the region. To be explicit, we take $\bm{\eta}=t^{-1}\bm{f_{S}}+\bm{\xi}$ which satisfies $Q_t(\bm{\eta})>0$. Then we consider intersection with the hyperplane $\{\bm{\lambda}\in \R^{s+1}\,:\,B_{t}(\bm{\eta},\bm{\lambda})=1\}$. Concretely, in terms of the coordinates for $\bm{\lambda}=a\bm{f_S}+\bm{v}+\bm{w}+b\bm{\xi}$ given in \eqref{eqnVevenfinercoordinates} one has
	\begin{equation}\label{eqncrosssectioncoordinateexpression}
		1=B_t(\bm{\eta},\bm{\lambda})=t^{-1}b.
	\end{equation}
	
	We distinguish two cases depending on whether $\sangle{C_T}$ is negative definite. If $\sangle{C_T}$ is not negative definite, then, by Lemma \ref{lemCTnegativesemidefinite}, we have $T=S$. A direct computation, using \eqref{eqnVevenfinerdecompositioncoordinates}, gives
	\begin{equation}\label{eqnQTfornotnegativedefiniteCT}
		Q_{T}^{+}(  \bm {\lambda})=ab+t \left(a-{b\over t}\right)^2-Q(\bm{v}) ={b^2\over t}-ab+ta^2-Q(\bm{v}).
	\end{equation}
	By the construction in \eqref{eqnCSdecomposition}, $W_{S}$ is negative definite and thus $Q(\bm{v})\leq 0$. Furthermore, by \eqref{eqncrosssectioncoordinateexpression}, we have that $t^{-1}b^2-ab+ta^2>0$ if $t>0$ is sufficiently small. This gives that $Q_{T}^{+}(  \bm {\lambda})>0$.
	
	If $\langle C_{T}\rangle$ is negative definite, then, by Lemmas \ref{lemCTnegativesemidefinite} and \ref{lemorthogonalprojectionCT}, we have $V=\langle C_{T} \rangle\oplus \langle C_{T} \rangle^{\perp}$, $\langle C_{T} \rangle\cap \mathbb{R}\bm{f_S}=\{\bm{0}\},\mathbb{R}\bm{f_S}\subseteq \langle C_{T} \rangle^{\perp}$, and $\bm{\lambda_{C_{T}}}=\bm{\lambda}^{[t]}_{\bm {C_{T}}}$. Thus we have
	\begin{align}\label{eqnQTexpression}
		Q_{T}^{+}\left(  \bm {\lambda}\right)&= Q\left(  \bm {\lambda}\right)-2Q_{t}\left(  \bm {\lambda}_{\bm {C_{T}}}\right) =Q\left(\bm {\lambda}_{\perp \bm {C_{T}}}\right)-  Q\left(  \bm {\lambda}_{\bm {C_{T}}}\right).
	\end{align}
	Furthermore,  Lemma \ref{lemorthogonalprojectionCT} gives $\bm {\lambda}_{\perp \bm {C_{T}}}=\bm {\lambda}^{[t]}_{\perp \bm {C_{T}}}$ and thus
	\begin{equation}\label{eqnBtsign}
		B_{t}\left(\bm{c_k},\bm {\lambda}^{[t]}_{\perp \bm {C_{T}}} \right)=B_{t}\left(\bm{c_k},\bm {\lambda}_{\perp \bm {C_{T}}} \right).
	\end{equation}
	Consider the interior $R_t^\circ(C_{T^c\perp C_T})$ of the polyhedron $R_t(C_{T^c\perp C_T})$ where the inequalities in \eqref{eqnequationforpolyhedron} are strict. For any $\bm\l\in R_t^\circ(C_{T^c\perp C_T})$ in the positive half of the polyhedron, $B_t(\bm{c_k},\bm\l_{\perp\bm{C_T}}^{[t]})>0$ for all $k\in T^c$. By \eqref{eqnBtsign} and continuity of $t\mapsto B_t$, we have that $B(\bm{c_k},\bm\l_{\perp\bm{C_T}})\ge0$ for all $k\in T^c$. The same reasoning applies to the negative half of the polyhedron. Thus for any $\bm\l\in R_t^\circ(C_{T^c\perp C_T})$, we have $\bm\l_{\perp\bm{C_T}}\in R(C_{T^c})$. It follows that $\bm\l_{\perp\bm{C_T}}\in R(C)\subseteq D_+\cup D_0$ and thus in \eqref{eqnQTexpression} we have $Q(\bm\l_{\perp\bm{C_T}})\ge0$. Since $\sangle{C_T}$ is negative semi-definite, one has $Q(\bm\l_{\bm{C_T}})\le0$. Therefore, $Q(\bm\l_{\perp\bm{C_T}})-Q(\bm{\l_{C_T}})\ge0$ and $Q(\bm\l_{\perp\bm{C_T}})-Q(\bm{\l_{C_T}})=0$ if and only if $\bm\l\in\R\bm{f_S}$. This finishes the proof.
\end{proof}

By continuity it follows that $Q_{T}^{+}$ is non-negative on the boundary of the polyhedron $R_{t}(C_{T^{c}\perp C_{T}})$.
We next consider the faces of the polyhedron $R_{t}(C_{T^{c}\perp C_{T}})$.

\begin{lemma}\label{lempositivity}
	Suppose Assumption \ref{assumptionsI} holds with $C=(\bm{c_1},\cdots, \bm{c_{s}},\bm{c_{0}})$.
If $t>0$ is sufficiently small,
then there exist constants (depending on $t$) $\kappa, \delta>0$ such that in a cross section of $R_{t}(C_{T^{c}\perp C_{T}})\setminus \mathbb{R}\overline{B_\varepsilon(\bm f)}$ we have

\begin{equation}\label{eqnsummandestimateonfull}
\left|
 e^{-\pi   Q\left(  \sqrt{2v}\bm \lambda\right)}
M_{Q_{t},C_T}\left( \sqrt{2v}\bb \lambda\right)  p_{C_{T^{c}\perp^{[t]} C_{T}}}\left(\bb \lambda\right)
\right|
\leq \kappa e^{-\pi \delta v}.
\end{equation}
\end{lemma}

\begin{proof}
	For any non-zero vector $\bm u\in R_t(C_{T^c\perp C_T})\setminus\R\ol{B_\e(\bm f)}$, consider $L:=\{0\le j\le s:B_t(\bm{c_j},\bm u_{\perp\bm{C_T}}^{[t]})=0\}$, which satisfies $T\subseteq L$ and
	\begin{equation}\label{eqnequationforpolyhedronstrictT}
		B_{t}\left(\bm{c_k},\bm{u}^{[t]}_{\perp\bm{C_{T}}} \right)> 0\text{ for } k\in L^{c}\,~~
		\text{or}~~
		B_{t}\left(\bm{c_k},\bm{u}^{[t]}_{\perp\bm{C_{T}}} \right)< 0 \text{ for } k\in L^{c}.
	\end{equation}

	If $L=\{0,\dots,s\}$, then by the non-degeneracy of $B_{t}$ we have
	$\bm{u}^{[t]}_{\perp \bm{ C_T}}=\bm{0}$ and thus $\bm{u}=\bm{u}^{[t]}_{\bm {C_{T}}}$.
	By \eqref{eqnQTfornotnegativedefiniteCT}, $Q_{T}^{+}$ is  strictly bounded below if $\langle C_{T}\rangle$ is not negative definite.
 	Consider next the case that $\langle C_{T}\rangle$ is negative definite.
	Using Lemma \ref{lemorthogonalprojectionCT}  and \eqref{eqnQTexpression}, we have that
	\[
		Q_{T}^{+}(\bm{u})=Q\left(\bm{u}_{\perp {\bm C_{T}}}\right)-Q\left(\bm{u}_{ {\bm C_{T}}}\right)=
		-Q\left(\bm{u}_{ {\bm C_{T}}}\right).
	\]
	This also has a strictly positive lower bound in a compact region of the cross section by the negative definiteness of $\sangle{C_T}$. In either case, the inequality follows from \eqref{eqnMprewriting}, \eqref{eqnMestimate}, and $e^{-\pi Q(\bm\l)}e^{2\pi Q_t(\bm{\l_{C_T}}^{[t]})}=e^{-\pi Q_T^+(\bm\l)}$. In the case $L\subsetneq\{0,\dots,s\}$, we claim that for any $\bm u$ satisfying \eqref{eqnequationforpolyhedronstrictT}, one has $\bm u\in R_t^\circ(C_{L^c\perp C_L})\setminus\R\ol{B_\e(\bm f)}$. That is,
	\begin{equation}\label{eqnequationforpolyhedronstrictL}
		B_{t}\left(\bm{c_k},\bm{u}^{[t]}_{\perp\bm{C_{L}}} \right)> 0\text{ for } k\in L^{c}\,~~
		\text{or}~~
		B_{t}\left(\bm{c_k},\bm{u}^{[t]}_{\perp\bm{C_{L}}} \right)< 0 \text{ for } k\in L^{c}.
	\end{equation}
	Indeed, by the orthogonal decomposition in Lemma \ref{lemCorthogonalprojection}, $\langle C_T\rangle^{\perp^{[t]}}=\langle D_{T^{c}}\rangle,\langle C_L\rangle^{\perp^{[t]}}=\langle D_{L^{c}}\rangle$
	and thus
	$
		\bm{u}^{[t]}_{\perp\bm{C_{T}}} =\bm{u}^{[t]}_{\perp\bm{C_{L}}} +\bm{v}$ for some $\bm{v}\in \langle D_{T^{c}\setminus L^{c}}\rangle.
	$
	Using the fact that $D_{\{0,\dots,s\}}$ is the dual basis of  $C_{\{0,\dots,s\}}$,  we have $B_{t}(\bm{c_k},\bm{v})=0, k\in L^{c}$.
	Therefore,
	\begin{equation}\label{eqnBtckurelation}
		B_{t}\left(\bm{c_k},\bm{u}^{[t]}_{\perp\bm{C_{T}}} \right)=
		B_{t}\left(\bm{c_k},\bm{u}^{[t]}_{\perp\bm{C_{T}}} \right)+B_{t}\left(\bm{c_k},\bm{v}\right)=B_{t}\left(\bm{c_k},\bm{u}^{[t]}_{\perp\bm{C_{L}}} \right)\,\text{ for } k\in L^{c}.
	\end{equation}
	Now \eqref{eqnequationforpolyhedronstrictL} follows from \eqref{eqnequationforpolyhedronstrictT}. Thus $\bm u$ lies in the interior of $R_t(C_{L^c\perp C_L})$ which is a face of $R_t(C_{T^c\perp C_T})$. Consider any vector $\bm{u}$ lying in the interior $R^{\circ}_{t}(C_{L^{c}\perp C_{L}})\setminus \mathbb{R}\overline{B_\varepsilon(\bm f)}$.
	From Lemma \ref{lempositivityininterior} and continuity, there exists a neighborhood $N(\bm{u})\subseteq R_{t}(C_{T^{c}\perp C_{T}})\setminus \mathbb{R}\overline{B_\varepsilon(\bm f)}$ and $\delta_{\bm{u}}>0$ such that $Q_{L}^{+}(\bm{\lambda})\geq \frac{\delta_{\bm{u}}}2$ for $\bm{\lambda}\in N({\bm{u}})$.
	In particular, using \eqref{eqnMestimate}, we have for $\bm{\lambda}\in N(\bm{u})$
	\begin{equation}\label{eqnsummandestimateonX}
		\left|
 		e^{-\pi   Q\left(  \sqrt{2v}\bm \lambda\right)}
		M_{Q_{t},C_L}\left( \sqrt{2v}\bb \lambda\right)  p_{C_{L^{c}\perp^{[t]} C_{L}}}\left(\bb \lambda\right)
		\right|
		\leq K e^{-\pi Q_{L}^{+}\left( \sqrt{2v}\bm{\lambda}\right)}\leq K e^{-\pi v\, \delta_{\bm{u}}}.
	\end{equation}
	For any such $\bm{\lambda}$, by \cite[Proposition 3.4, equation (51)]{Na} and  Lemma \ref{ErrorProperties},
	one has 
	\[
	\lim
	M_{Q_{t},C_L}\left( \sqrt{2v}\bb \lambda\right)
	=(-1)^{|L|+|T|} \sgn\left(B_t\left(\left(C_{L\setminus T}\right)_{\perp C_T}^{[t]},\bm{\lambda}\right)\right) M_{Q_{t},C_T}\left( \sqrt{2v}\bb \lambda\right)  \,,
	\]
	with the limit taken at the locus $B_t((C_{L\setminus T})_{\perp C_T}^{[t]},\bm\l)=\bm0$. On the other hand, from \eqref{eqnPCdefinition} and \eqref{eqnBtckurelation}, we obtain that along the same locus we have, for $m\in\N_0$ satisfying $0\leq m\leq |L|-|T|$,
	\[
		\lim \left\lvert p_{C_{T^{c}\perp^{[t]} C_{T}}}(\bb \lambda)\right\rvert= \frac{1}{2^{m}}\left\lvert p_{C_{L^{c}\perp^{[t]} C_{L}}}(\bb \lambda)\right\rvert.
	\]
	Here $m$ is the number of 
	vanishing components in $B_t((C_{L\setminus T})_{\perp C_T}^{[t]},\bm{\lambda})$ in the limit process.
	This gives
	\[
		\lim\,
		\left(\left\lvert
		M_{Q_{t},C_L}\left( \sqrt{2v}\bb \lambda\right)  p_{C_{L^{c}\perp^{[t]} C_{L}}}(\bb \lambda)\right\rvert\right)
		=  2^{m} \lim \left( \left\lvert M_{Q_{t},C_T}\left( \sqrt{2v}\bb \lambda\right)  p_{C_{T^{c}\perp^{[t]} C_{T}}}(\bb \lambda)\right\rvert\right).
	\]
	By shrinking the neighborhood $N(\bm{u})$ if necessary, according to \eqref{eqnsummandestimateonX} we obtain that inside $N(\bm{u})$  
	\begin{equation}\label{eqnsummandestimateonNX}
		\left|
		e^{-\pi   Q\left(  \sqrt{2v}\bm \lambda\right)}
		M_{Q_{t},C_T}\left( \sqrt{2v}\bb \lambda\right)  p_{C_{T^{c}\perp^{[t]} C_{T}}}\left(\bb \lambda\right)
		\right|
		\leq 2^{-m} \cdot 2 K e^{-\pi v\,\delta_{\bm{u}}}.
	\end{equation}
	
	\begin{figure}[h]
		\centering
		\includegraphics[scale=0.6]{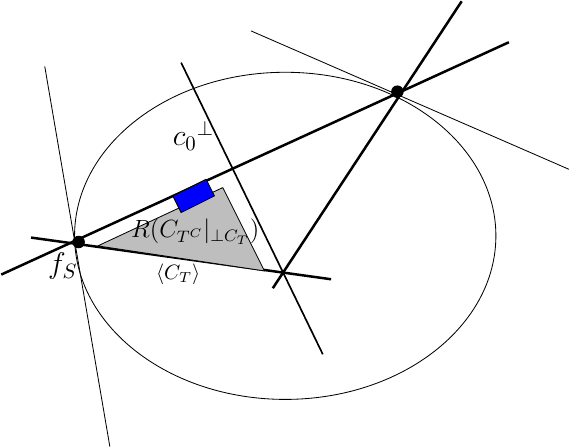}
		\caption{A cross section of the polyhedron $R_{t}(C_{T^{c}\perp C_{T}})$.}
		\label{figure:stratification}
	\end{figure}
	
	For any point $\bm{u}$ on the boundary of the cross section of $R_{t}(C_{T^{c}\perp C_{T}})\setminus \mathbb{R}\overline{B_\varepsilon(\bm f)}$, one
	constructs such a neighborhood $N(\bm{u})$. The resulting open sets then cover the boundary of the cross section. By compactness of the latter, we can find a finite cover
	within which we have the estimates given by \eqref{eqnsummandestimateonNX}.  After that, consider the complement (within $R_{t}(C_{T^{c}\perp C_{T}})\setminus \mathbb{R}\overline{B_\varepsilon(\bm f)}$) of the open finite cover, which is compact and covers the remaining cross section of $R_{t}(C_{T^{c}\perp C_{T}})\setminus \mathbb{R}\overline{B_\varepsilon(\bm f)}$.
	For the compact complement we also have an estimate of the form in
	\eqref{eqnsummandestimateonNX} according to Lemma \ref{lempositivityininterior};
	see Figure \ref{figure:stratification} for an illustration.
	We then obtain the desired estimate in \eqref{eqnsummandestimateonfull}
	everywhere in the cross section of the polyhedron $R_{t}(C_{T^{c}\perp C_{T}})\setminus \mathbb{R}\overline{B_\varepsilon(\bm f)}$.
\end{proof}

We are now ready to prove Theorem \ref{thm:Modularitysimplicialwithisotropicray}.
\begin{proof}[Proof of Theorem \ref{thm:Modularitysimplicialwithisotropicray}]
The compact convergence of $ \Theta_{Q,p_{C}}$ follows as in Theorem \ref{thmconvergenceofcountingfunction}.
Based on Lemma \ref{lempositivity},
we apply the compactness argument as in the proof of Theorem \ref{thm:Modularitysimplicialwithnoisotropicray}
which gives compact convergence of the series in \eqref{eqnreducedseries} and thus of $\Theta_{Q,\wh{p}_{C,t}}$, summed away from $\overline{B_\varepsilon(\bm f)}$.
For the summation over  $\overline{B_\varepsilon(\bm f)}$, we use \eqref{eqnMestimate}, Lemma \ref{lempositivityininterior},
and Theorem \ref{thmconvergenceofcountingfunction}. Similarly, $\wh{p}_{C,t}$  satisfies Theorem \ref{Vigneras} (i) according to Lemma  \ref{lempositivity}, \eqref{eqnMestimate}, and Lemma \ref{lempositivityininterior}.
   These give the compact convergence and meromorphicity of $\Theta_{Q,\wh{p}_{C,t}}$, and the desired modularity of $\Theta_{Q,\wh{p}_{C,t}}$.
 The relation between $\TH_{Q,p_C}$  and $\Theta_{Q,\wh{p}_{C,t}} $ follows from their compact convergences. \qedhere
\end{proof}

\subsubsection{Indefinite theta functions on faces of simplicial polyhedrons}\label{secfacesandthetafunctions}

We next consider the restrictions of theta functions to faces. Recall from \eqref{eqndfnface} that for any non-empty set $J\subseteq C$ we write $V_J=\sangle{C_J}^\perp$, $F_J=V_J\cap R(C)$. If $F_J$ is an isotropic ray $\R\bm f$, then by the simplicial condition and dimension reasons so is $V_J$. It follows that $\sangle{C_J}$ can not be $V$ nor negative definite by the signature of $V$. By Lemma \ref{lemCTnegativesemidefinite}, $J=\{1,\dots,s\}$. We take the splitting \eqref{eqnVfinerdecomposition}, with $\bm\xi\in D_0$ satisfying \eqref{eqnxiforsplitting}. A direct computation as in Theorem \ref{thmconvergenceofcountingfunction} and Corollary \ref{corThetaseriespolarpart} shows that the corresponding series $\TH_{Q,\d_{V_J}p_C}$ and $\TH_{Q,\d_{V_J}p_{C_{J^c}}}$ are divergent. 
We have the following result if $F_J$ is not an isotropic ray.

\begin{corollary}\label{corthetarestrictedonfaceinsimplex}
	Suppose Assumption \ref{assumptionsI} holds with $C=(\bm{c_1},\cdots,\bm{c_s},\bm{c_0})$. Assume $F_J$ is not an isotropic ray. The series $\TH_{Q,\d_{V_J}p_{C_{J^c}}}$ and $\TH_{Q,\d_{V_J}p_C}$ define meromorphic functions in $(\bm{z_{V_J}},\t)$ whose modular completions have weight $\frac{\dim(F_J)+1}{2}$.
\end{corollary}

\begin{proof}
	We only show the claim for $\TH_{Q,\d_{V_J} p_{C_{J^c}}}$, the claim for $\TH_{Q,\d_{V_J}p_C}$ follows analogously. Applying Lemma \ref{lemfacesofDeltaanddegenerations} to the polyhedron $R(C)$, $Q|_{V_J}$ has signature $(1,\dim(F_J)-1)$ if $F_J$ is not an isotropic ray. Take the splitting of the quadratic space $V$, 
	$V=V_J\oplus Z_J$. Note that $Z_J$ should not be confused with $\sangle{C_J}$ which may not give such a splitting (see \eqref{eqnVfinerdecomposition}). We write $\bm\l=\bm\mu+\bm\l_{\bm Z_J}$. The non-vanishing of $\d_{V_J}(\bm\l)$ requires that the orthogonal decomposition satisfies $\bm\l_{\bm Z_J}=0$. Since the face $V_J$ is cut out by hyperplanes with integral normal vectors, the quadratic space $V_J$ inherits the integral structure from $V$. Then $\bm\l_{\bm Z_J}=0$ implies $\bm\a_{\bm Z_J}\in\Z^{s+1}$. It follows that
	\begin{multline*}
		\sum_{\bm n\in\Z^{s+1}} \d_{V_J}(\bm\l)p_{C_{J^c}}(\bm\l)q^{Q(\bm n)}e^{2\pi iB(\bm n,\bm z)}\\
		= q^{-Q\left(\bm{\a_{Z_J}}\right)}e^{-2\pi iB\left(\bm\a_{\bm Z_J},\bm\b_{\bm Z_J}\right)}\sum_{\bm\a_{\bm V_J}+\bm m\in\left(\Z^{s+1}+\bm\a_{\bm V_J}\right)\cap V_J} p_{C_{J^c|_{V_J}}}\left(\bm\a_{\bm V_J}+\bm m\right)q^{Q(\bm m)}e^{2\pi iB\left(\bm m,\bm z_{\bm V_J}\right)}.
	\end{multline*}
Convergence and modularity follows from Theorems \ref{thmconvergenceofcountingfunction} and \ref{thm:Modularitysimplicialwithisotropicray} if $\dim(F_J)\ge2$ and from standard facts on Jacobi theta functions if $\dim(F_J)=1$.
\end{proof}


\subsection{Modularity of indefinite theta functions of signature $(1,s)$: non-simplicial case}

 Next we generalize Theorems \ref{thm:Modularitysimplicialwithnoisotropicray} and \ref{thm:Modularitysimplicialwithisotropicray}
  from simplicial cones to polyhedral cones arising in the counting functions.
The idea is to decompose polyhedral cones into simplicial cones.

\subsubsection{Simplicial decomposition}

We first mention some elementary results regarding simplicial decompositions of general polyhedrons.
Then we restrict to the ones arising in the counting functions.

\begin{lemma}\label{DecomposePolyhedron}
	One can decompose a convex $r$-dimensional polytope $P$ into $r$-dimensional simplices $S_1,\dots,S_\ell$ which only overlap on their boundaries.
\end{lemma}

\begin{proof}
	We prove the lemma by induction on $r$. Pick $\bb{w_0}\in P$ and label the faces of $P$ not containing $\bb{w_0}$ as $F_1,\dots, F_k$. These are convex $(r-1)$-dimensional polytope. By induction hypothesis, we can split each of these into simplices, labeling all of these simplices as $G_1,\dots, G_\ell$.
	Writing $\mathrm{conv}(M)$ for the convex hull of the set $M$, the polytope $P$ decomposes into the simplices $S_1:=\mathrm{conv}(G_1\cup\{\bb{w_0}\}),\dots, S_\ell:=\mathrm{conv}(G_\ell\cup\{\bb{w_0}\})$ overlapping only on the boundaries.
\end{proof}

We apply this to decompose polyhedral cones into simplicial cones.

\begin{lemma}\label{DecomposeSimplicialGeneral}
	Let $Q$ be a quadratic form of signature $(1,s)$ and consider the cone $R(C)$ with $C:=(\bm{c_0},\dots,\bm{c_k})$ and $R(C)\subseteq D_+\cup\{\bm0\}$.
	\begin{enumerate}
		\item One can decompose $R$ into simplicial cones $R(C_1),\dots, R(C_\ell) \subseteq D_+\cup \{\bm{0}\}$ intersecting only on the boundary.
		
		\item There exist $W_j\subset\R^{s+1}$ and $\chi_{W_j}$ satisfying $\{\bm0\}\ne W_j\cap R(C_j)\subseteq D_+\cup\{\bm0\}$, and additional simplicial cones denoted by $R(C_{\ell+1}),\dots,R(C_h)\subseteq D_+\cup\{\bm0\}$, and $d_j\in\R$,  such that
		\begin{align*}
			p_C(\bm{x}) = \sum_{j=1}^{\ell} p_{C_j}(\bm{x})+\sum_{j=\ell+1}^{h} \chi_{W_j}(\bm{x}) d_j  p_{C_j}(\bm{x}).
		\end{align*}
	\end{enumerate}
\end{lemma}

\begin{proof}
	(1) Note that $\sangle C=\R^{s+1}$ and $k\ge s$ since for any $\bm w\in \R^{s+1}\setminus\sangle C^\perp$, we have that $\bm w\oplus\sangle C^\perp\subseteq R(C)\subseteq D_+\cup D_0$ is at most one-dimensional due to the signature $(1,s)$ of $Q$. Fix a vector $\bm d\in\R^{s+1}$ with $Q(\bm d)>0$. We begin by intersecting the cone $R(C)$ with the hyperplane
	\begin{equation*}
		H_{\bb d}:=\left\{\bm{x}\in \R^{s+1}: B(\bb d,\bm{x})=1 \right\}.
	\end{equation*}
	Since $R(C)\setminus\{\bm{0}\}$ has two connected components and $H_{\bb{d}}$ intersects only one of them, we can write 
	\begin{equation*}
		R(C)\cap H_{\bm d}=\left\{ \bm{x}\in\R^{s+1}: B(\bb d,\bm{x})=1, B(\bb{c_j},\bm{x})\geq 0\text{ for } 0\leq j\leq k
		\right\},
	\end{equation*}
	possibly after changing $\bm d\mapsto-\bm d$ which is a $s$-dimensional convex polytope. By Lemma \ref{DecomposePolyhedron} we can decompose it into $s$-dimensional simplices $S_1,\dots,S_\ell$. Then there are $\ell$ sets $C_j\,(1\le j\le\ell)$ consisting of $s+1$ vectors each such that $R(C_j):=\R S_j$ and
	\begin{equation*}
		\bigcup_{j=1}^\ell R(C_j) =\R \bigcup_{j=1}^\ell S_j =\R(R(C)\cap H_{\bm{d}})= R(C)
	\end{equation*}
	since each simplex can be described by $s+1$ conditions of the shape $B(\bb{c},\bm{x})>0$. Thus the $R(C_j)$ form a decomposition of $R(C)$ into simplicial cones intersecting only at the boundaries.
	
	\noindent (2) By the proof of (1),
	$
	p_C(\bm{x})-\sum_{j=1}^{\ell} p_{C_j}(\bm{x})
	$
	is supported on the boundaries of $R(C_j)$. By the inclusion-exclusion principle, it consists of a linear combination of terms of the shape (see 
	\eqref{eqnpCexpanded} for details)
	$
		\sum_{j=\ell+1}^{h}\chi_{W_j}(\bm{x}) p_{C_j}(\bm{x}).
	$
	Here $W_j\subset\R^{s+1}$ are subspaces with $\{\bm0\}\ne W_j\cap R(C_j)\subseteq D_+\cup\{\bm0\}$ (where for $\ell+1\le j\le h$ there exists some $1\le k\le\ell$ such that $C_j=C_k$). Since $W_j$ contains a positive vector, $W_j$ has signature $(1,\dim(W_j)-1)$. Note that $R(C_j)\cap W_j\subseteq(D_+\cup\{\bm0\})\cap W_j$ satisfies the conditions of the Lemma in the lower-dimensional subspace $W_j$. Since the claim is true for $\dim(W_j)=1$ (with $\ell=h=1$, $C_1=C$), this proves the claim by induction on $\dim(W_j)$.
\end{proof}

In the rest of this section, let $(V,Q)$ be the quadratic space arising from the counting function $\Theta_{Q,\chi}$ of convex polygons given in \eqref{eqncountingfunctionwithparameterlambda}
that satisfy the assumptions\footnote{In fact, the discussions below also apply to the case $N=4$.} of Theorem \ref{thmconvergenceofcountingfunction}.
Let $C=(\bm{c_{k}}:\,1\leq k\leq N), N=s+3$ be the set of normal vectors therein.
By Proposition \ref{propisotricconestructure} and Lemma \ref{lemlineardepenence}, one has
$\bm{c_{k}}^{\perp}\cap (D_{+}\cup D_{0})\neq \{\bm{0}\}$.
By Lemma \ref{lemsimplefacts} (3),
we have $Q(\bm{c_k})\leq 0$.
Furthermore, any hyperplanes whose intersection with $R(C)$ has dimension $s$ (i.e., the supporting hyperplanes) satisfies
\begin{equation}\label{eqnsupportinghyperplaneQ}
	\bm{c_{k}}^{\perp}\cap D_{+}\neq \{\bm{0}\}
\end{equation}
and thus $Q(\bm{c_k})<0$ by Lemma \ref{lemsimplefacts} (3). We next consider simplicial decompositions of $R(C)\subseteq D_+\cup D_0$ into simplicial polyhedrons with each of them satisfying Assumption \ref{assumptionsI} using the construction in Lemmas \ref{DecomposePolyhedron} and \ref{DecomposeSimplicialGeneral}. We follow the following procedure:
\begin{enumerate}[label=\textnormal{(\arabic*)}]
	\item We first decompose the polyhedron $R(C)$ into small polyhedrons $R_{\bm{f}}$ each containing a small neighborhood of one isotropic ray $\mathbb{R}\bm{f}$ and the complement $R^{[c]}$ of the union of these small polyhedrons. The cross section of the small polyhedron $R_{\bm{f}}$ is a cone over some polytope $P_{\bm{f}}$, with the tip $\bm{w_0}(\bm{f})$ (as in the proof of Lemma \ref{DecomposePolyhedron}) given by the cross section of $\mathbb{R}\bm{f}$.
	
	\item We apply the triangulation procedure in Lemma \ref{DecomposePolyhedron} with the cross section $\bm{w_0}(\bm{f})$ of $\mathbb{R}\bm{f}$ chosen as the reference vertex, to get a triangulation $X_{\bm{f}}$ for $R_{\bm{f}}$. Any triangulation of $P_{\bm{f}}$ gives such a triangulation for $R_{\bm{f}}$.
	
	\item The simplicial triangulations of the small polyhedrons $R_{\bm{f}}$ above also create faces for the complement polyhedron  $R^{[c]}$. Apply the triangulation procedure in Lemma \ref{DecomposePolyhedron} to get a triangulation for the resulting complement polyhedron  $R^{[c]}$.
\end{enumerate}
Denote the resulting simplicial decomposition of $R(C)$ by $X$. It is obtained by inserting hyperplanes with normal vectors $\bm{c_\ell}$ ($N+1\le\ell\le N+M$) for $M\in\N_0$. We can arrange the newly inserted hyperplanes such that all of the normal vectors are rational by perturbing them if necessary. Denote as before $\l_j=B(\bm{c_j},\bm\l)$ and $\l_{-j}=B(-\bm{c_j},\bm\l)$, $1\le j\le N+M$. Let $J\subseteq\{N+1,\dots,N+M\},J^c=\{N+1,\dots,N+M\}\setminus J$. It follows that
\begin{align}\label{eqnpCexpanded}\nonumber
	p_C(\bm\l) &= p_C(\bm\l)\prod_{\ell=N+1}^{N+M} \left(\frac{1+\sgn(\l_\ell)}{2}+\frac{1+\sgn(\l_{-\ell})}{2}\right)\\
	\nonumber
	&= \sum_{J\subseteq\{N+1,\dots, N+M\}} \left(\prod_{k=1}^{N} \frac{1+\sgn(\l_k)}{2}\prod_{\ell\in J} \frac{1+\sgn(\l_\ell)}{2}\prod_{\ell\in J^c} \frac{1+\sgn(\l_{-\ell})}{2}\right.\\
	&\hspace{3.5cm}\left.- (-1)^N\prod_{k=1}^{N} \frac{1+\sgn(\l_{-k})}{2}\prod_{\ell\in J} \frac{1+\sgn(\l_{-\ell})}{2}\prod_{\ell\in J^c} \frac{1+\sgn(\l_\ell)}{2}\right).
\end{align}
Let $C(J):=(\bm{c_j}:1\le j\le N\cup J\cup (-J^c))$. For each summand $p_{C(J)}$ determined by $J$, the support is contained in the intersection of the half-spaces
\[
	R(C(J))=\{\bm{\lambda}~:~ {\lambda}_{j}\geq 0\,, \bm{c_{j}}\in C(J)  \,~ \mathrm{or}~ {\lambda}_{j}\leq 0\,, \bm{c_{j}}\in C(J) \}.
\]
Since $X$ is a simplicial decomposition, $F_{J}:=R(C(J))\subseteq R(C)$ is a simplicial polyhedron. Take any simplicial polyhedron $\xoverline{\Delta}_{J}$ of dimension $s+1$ containing $R(C(J))$. Denote the set of normal vectors of the codimension-one faces of $\xoverline{\Delta}_{J}$ by $C(\xoverline{\Delta}_{J})$.
If $F_{J}\subsetneq \xoverline{\Delta}_{J}$, then the corresponding summand $p_{C(J)}$ is a rational multiple of $\delta_{F_{J}} p_{C(\xoverline{\Delta}_J)}$. Otherwise, we have  $F_{J}=\xoverline{\Delta}_{J}$ and $p_{C(J)}-p_{C(\xoverline{\Delta}_{J})}$ is a linear combination of functions of the form $\delta_{G_{K}} p_{C(\xoverline{\Delta}_J)}$, where the $G_{K}$ are proper faces of $\xoverline{\Delta}_J$ obtained by intersecting hyperplanes $\langle \bm{c_{j}}\rangle^{\perp}, \bm{c_{j}}\in C(J)\setminus C(\xoverline{\Delta}_{J})$, with $\xoverline{\Delta}_J$.

We next deal with the case where $F_J$ or $G_K$ is an isotropic ray on which the theta functions are divergent as discussed in Subsubsection \ref{secfacesandthetafunctions}. By the construction of simplicial decomposition, it is enough to focus on the small polyhedron $R_{f}$ that contains the isotropic ray $\R\bm f$. We exclude those normal vectors $\bm{c_k}$, $1\le k\le N$, satisfying 
\begin{equation}\label{eqnexclusionofisotropicray}
	\bm{c_{k}}^{\perp}\cap R(C)=\bm{c_{k}}^{\perp}\cap  R(C)\cap D_{0}\,,
\end{equation}
which in the current case is equivalent to 
$
\dim\left(\bm{c_{k}}^{\perp}\cap R(C)\right)=1\,
$
by  Proposition \ref{propisotricconestructure} and Lemma \ref{lemlineardepenence}. The corresponding hyperplanes are non-supporting hyperplanes that intersect $R(C)$ only at the isotropic ray (see \eqref{eqnsupportinghyperplaneQ}), and removing them from $C$ does not change $R(C)$.
Define
\[
	C^{\mathrm{reg}}:=\left\{\bm{c_{k}}\in C~:~\bm{c_{k}}^{\perp}\cap R(C)\neq \bm{c_{k}}^{\perp}\cap  R(C)\cap D_{0}\right\}.
\]

By  Lemma \ref{lemsimplefacts} (3), we then have
$
	C^{\mathrm{reg}}=C\setminus  D_0.
$

\begin{lemma}\label{lemnoisotrpicrayinsimplicialdecomposition}
	Replacing $C$ by $C^{\reg}$ in \eqref{eqnpCexpanded}, the resulting faces $F_J$ and $G_K$ are not isotropic rays.
\end{lemma}

\begin{proof}
	For any isotropic ray $\R\bm f$, denote the normal vector of the face of $R_{\bm f}$ that does not contain $\bm{w_0}(\bm f)$ by $\bm{c_0}$. By Proposition \ref{propisotricconestructure} and Lemma \ref{lemlineardepenence}, there exists at most one non-supporting hyperplane that passes through the isotropic ray $\R\bm f$, and thus the set $\calC_{\bm f}$ of normal vectors from $C\cup\bm{c_0}$ whose hyperplanes intersect $R_{\bm f}$ non-trivially consists of $s+2$ elements.
	
	Suppose that $\calC_{\bm f}\cap D_0\ne\emptyset$. That is, there exists a non-supporting hyperplane for $R_{\bm f}$ whose normal vector is contained in $C\cap D_0$. After the exclusion of $C\cap D_0$ from $C$, the small polyhedron $R_{\bm f}$ is simplicial near $\R\bm f$ (in the sense that the cross section of $R_{\bm f}$ is a simplex). Thus no decomposition of the small polyhedron $R_{\bm f}$ is needed. In this case, for any $J$ such that $\xoverline\De_J=R_{\bm f}$, one has $R(C(J))\cap\R\bm f=\{\bm0\}$ or $R(C(J))=R_{\bm f}$ depending on the sign of $B(\bm{c_0},\bm\l)$. Otherwise, the set $\calC_{\bm f}$ consists of vectors from  $C^{\reg}\cup\bm{c_0}$.
	Suppose	there exists a non-supporting hyperplane whose normal vector $\bm{c_a}\in C^{\reg}$ satisfies $\dim(\bm{c_a}^\perp\cap R_{\bm f})<s$. Since $\dim(\bm{c_a}^\perp\cap R(C))=\dim(\bm{c_a}^\perp\cap R_{\bm f})$, by Lemma \ref{lemlineardepenence}, we have that the face is  an isotropic ray\footnote{For general polyhedrons, the face could be either an isotropic ray or a face of dimension greater than one containing the isotropic ray.}. Thus $\bm{c_a}\in D_0$ by Lemma \ref{lemsimplefacts} (3), contradicting the fact that $C^{\reg}\cap D_0=\emptyset$. It follows that any hyperplane from $\calC_{\bm f}$ is a supporting hyperplane.

	We then consider the cross section of $R_{\bm f}$. This is a cone over the polytope $P_{\bm f}$ obtained by intersecting $s+1$ half-spaces in $\bm{c_0}^\perp\cong\R^{s-1}$, with the tip $\bm{w_0}(\bm f)$ being the cross section of the isotropic ray $\R\bm f$. The newly inserted hyperplanes in the simplicial triangulation $X_{\bm f}$ of $R_{\bm f}$ arise from cones (first over $\bm{w_0}(\bm f)$ then over $\bm0\in\R^{s+1}$) over newly inserted hyperplanes in $P_{\bm f}$ that gives rise to a simplicial triangulation $Y_{\bm f}$ of $P_{\bm f}$. Hyperplanes in $Y_{\bm f}$ fall into two groups: the newly added ones each containing the reference vertex $\bm{w_0}$ in the construction using Lemma \ref{DecomposePolyhedron} and the old $s+1$ hyperplanes used in forming $P_{\bm f}\subseteq\bm{c_0}^\perp$. In particular, each of the new hyperplanes and old half-spaces contains $\bm{w_0}$. By the construction based on \eqref{eqnpCexpanded}, the faces $F_J$ of $G_K$ above arise from cones over intersections within $P_{\bm f}$ of hyperplanes from the former group and half-spaces including the latter group. Any such intersection contains $\bm{w_0}$ and is thus non-empty. Since each cone construction increases the dimensions of intersections by one\footnote{We use the uncommon convention that the dimension of the empty set is $-1$.}, the cone of any such intersection, now inside $X_{\bm f}$, can not be the isotropic ray $\R\bm f$.
\end{proof}

In sum, excluding the normal vectors subject to \eqref{eqnexclusionofisotropicray} from $C$
eliminates the faces $F_{J},G_{K}$ that are isotropic rays.
Correspondingly $p_{C}$ is replaced by $p_{C^{\mathrm{reg}}}$,
with
$\mathrm{supp}\,(p_{C}-p_{C^{\mathrm{reg}}})\subseteq R(C)\cap D_{0}$.

\subsubsection{Modularity of indefinite theta functions for polyhedron cones}

Combining Lemmas \ref{DecomposeSimplicialGeneral} and \ref{lemnoisotrpicrayinsimplicialdecomposition}, Theorem \ref{thm:Modularitysimplicialwithisotropicray}, and Corollary \ref{corthetarestrictedonfaceinsimplex} we obtain.

\begin{theorem}\label{thm:Modularitynonsimplicial}
		Assume \eqref{eqnpositivityonangles}, $N\ge5$,
	and \eqref{eqnrelaxedconditionintermsangles}.
	Let $C=(\bm{c_k}:1\le k\le N),N=s+3$ be the set of normal vectors of $R(C)\subseteq D_+\cup D_0$. Then $\TH_{Q,p_{C^{\reg}}}$ defines a meromorphic function and its completion is a sum of modular forms of (positive) weights at most $\frac{s+1}{2}$ and depth at most $s$.
\end{theorem}

\begin{remark}
	If $s+1\le4$, then, using elementary geometry, it is easy to see that at most one hyerplane is needed to construct a triangulation for $R_{\bm f}$. In general, the simplicial decomposition for $R(C)$ satisfying the desired properties is far from being unique. It seems interesting to investigate the dependence of the modular completion of $\TH_{Q,p_{C^{\reg}}}$ on the simplicial decomposition. 
\end{remark}

We finally discuss modularity of $\Theta_{Q,\chi_{C}}$.
Recall from
\eqref{eqnchifunction} that $\chi(\bm{\lambda})=
\chi_{C}(\bm{\lambda})$, where
\[
	\chi_{C}(\bm{\lambda})
	:=\prod_{k=1}^{N}
	\frac{1+\sgn({\lambda}_{k})-\delta({\lambda}_{k})}2
	-(-1)^{N}\prod_{k=1}^{N}
	\frac{1-\sgn({\lambda}_{k})-\delta({\lambda}_{k})}2	.
\]
A direct computation shows that 
\begin{equation*}\label{eqndecompositionofChi}
	\chi_{C}(\bm{\lambda})=\sum_{J\subseteq \{1,\dots,N\}} (-1)^{|J|} \prod_{k\in J}\delta({\lambda}_{k})  p_{C}({\lambda})
=\sum_{J\subseteq \{1,\dots,N\}} (-1)^{|J|} \delta_{\langle C_{J}\rangle^{\perp}}  p_{C}(\bm{\lambda}).
\end{equation*}
Consider the following regularized version of $\chi_C(\bm\l)$, excluding the contributions of isotropic rays

\begin{align}\label{RegularizedChi}
	\chi_{C}^{\mathrm{reg}}(\bm{\lambda}):=\sum_{\substack{J\subseteq \{1,\dots,N\}\\\langle C_{J}\rangle^{\perp}\cap R(C)\neq \langle C_{J}\rangle^{\perp}\cap D_0}} \delta_{\langle C_{J}\rangle^{\perp}}(\bm{\lambda})  p_{C^{\mathrm{reg}}}(\bm{\lambda}).
\end{align}

Note that solely removing $C\cap D_{0}$ from $C$ through the $p$-functions is not enough, since isotropic vectors can arise from faces of $R(C^{\mathrm{reg}})$. By Lemma \ref{lemlineardepenence}, the above summation over index sets $J$ is equivalent to the summation over faces of $R(C)$ that are not isotropic rays.

We finally arrive at the main modularity result stated in Theorem \ref{ModularityOfCoutingFunction}.

\begin{proof}[Proof of Theorem \ref{ModularityOfCoutingFunction}]
	By Remark \ref{remexceptionalcases} we only need to consider $N\ge5$. Using the construction of $\chi_C^{\reg}$ and Lemma \ref{lemnoisotrpicrayinsimplicialdecomposition}, the desired result follows from the same proof as Theorem \ref{thm:Modularitynonsimplicial}.
\end{proof}

\begin{remark}
	We can  prove the  modularity with other versions of $\chi_C$ similarly using the same approach. For example, consider
	\[
		\chi(\bm{\lambda})
		:=\prod_{k=1}^{N} \frac{1+\sgn({\lambda}_{k})+\delta({\lambda}_{k})}2
		-(-1)^{N}\prod_{k=1}^{N}  \frac{1-\sgn(\bm{\lambda}_{k})-\delta({\lambda}_{k})}2.
	\]
	A direct computation shows that 
	\begin{equation*}\label{eqndecompositionofChi}
		\chi(\bm{\lambda})=\sum_{J\subseteq \{1,\dots,N\}}2^{-|J|} \prod_{k\in J}\delta({\lambda}_{k})  p_{C_{J^{c}}}({\lambda})
		=\sum_{J\subseteq \{1,\dots,N\}} 2^{-|J|} \delta_{\langle C_{J}\rangle^{\perp}}  p_{C_{J^{c}}}(\bm{\lambda})
 		.
	\end{equation*}
	Again applying Corollary \ref{corthetarestrictedonfaceinsimplex} proves modularity of $\TH_{Q,\chi_C^{\reg}}$.
\end{remark}

\begin{example}
	Consider the pentagon with the following angles $\th_1=\th_2=\th_3=\frac\pi3$, $\th_4=\th_5=\frac{2\pi}{3}$.
	\begin{figure}[h]
		\centering
		\includegraphics[scale=0.6]{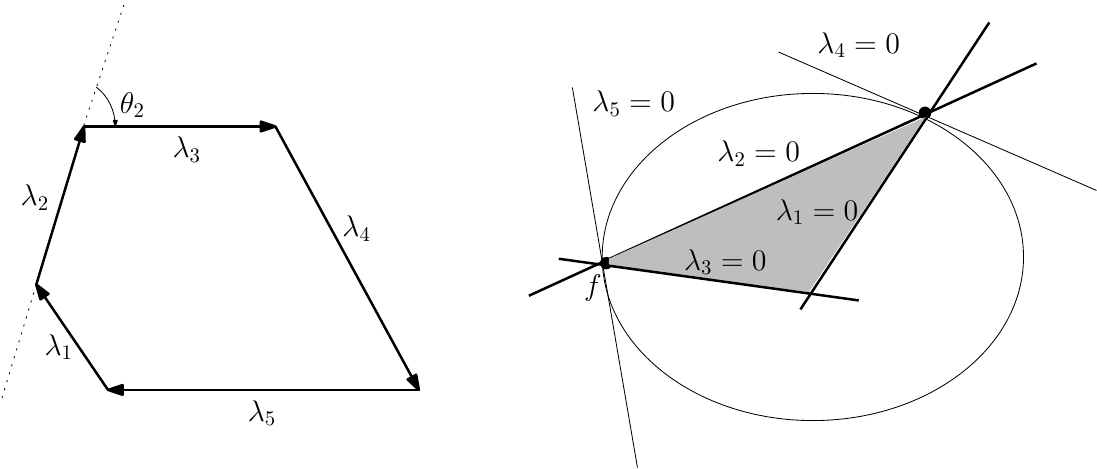}
		\caption{The figure on the left shows the pentagon, the one on the right shows a cross section of the domain $\xoverline{\Delta}$ in Proposition \ref{propisotricconestructure}.}
		\label{figure:pentagon}
	\end{figure}

	\noindent Choosing $\lambda_{j}\,(1\leq j \leq 3)$ as independent parameters, the vectors $\bm{v}$ and $\bm{w}$ in \eqref{eqndfnofNlambda} are $\bm{v}=(1, 1, 0)^{T}$ and $\bm{w}=(0, 1, 1)^{T}$. The matrix in \eqref{eqndefinitionofA} is
	\[
		A = \sin\left(\frac\pi3\right)\mat{0&1&1\\1&1&1\\1&1&0}
	\]
	and correspondingly the quadratic form given in Lemma \ref{la:quadForm} is
	\[
		Q(\bm\l) = \sin\left(\frac\pi3\right)\left(\l_1\l_2+\l_1\l_3+\l_2\l_3+\frac{\l_2^2}2\right).
	\]	
	The positivity condition in \eqref{eqnpositivityonangles} is satisfied, so are the conditions in Lemma \ref{lemconditiononangles} and \eqref{eqnrelaxedconditionintermsangles} by the computations in Example \ref{exconditiononangles}. In fact, we take $P=(\sin(\frac\pi3))^{-\frac12}\Id$ and $S=(\sin(\frac\pi3))^{-\frac12}\Id$, so that
	\[
		S^{T}A S=\begin{pmatrix}
		0 & 1 & 1\\
		1 & 1 & 1\\
		1 & 1 & 0
		\end{pmatrix}.
	\]
 	The normal vectors in \eqref{eqnnormalvectors} are then
 	\[
 		\bm{c}_{\bm 1}=
 		\begin{pmatrix}
 		-1 \\
 		1\\
 		0
 		\end{pmatrix}\,,\quad
  		\bm{c}_{\bm 2}=
 		\begin{pmatrix}
		1 \\
 		-1\\
		1
 		\end{pmatrix}\,,\quad
  		\bm{c}_{\bm 3}=
 		\begin{pmatrix}
		0 \\
		1\\
		-1
 		\end{pmatrix}\,,\quad
  		\bm{c}_{\bm 4}=
 		\begin{pmatrix}
		0 \\
		0\\
		1
 		\end{pmatrix}=  \bm{e}_{\bm 3}\,,\quad
  		\bm{c}_{\bm 5}=
 		\begin{pmatrix}
		1 \\
		0\\
		0
 		\end{pmatrix}= \bm{e}_{\bm 1}.
 	\]	
	They satisfy $Q(\bm{c}_{\bm 1})=Q(\bm{c}_{\bm 2})=Q(\bm{c}_{\bm 3})=-{1\over2},Q(\bm{c}_{\bm4})=Q(\bm{c}_{\bm 5})=0$.
	In particular, $\bm{c_{k}}^{\perp}\cap R(C)=\bm{c_{k}}^{\perp}\cap R(C)\cap D_{0}$ for $k\in\{4,5\}$. Consider for example the isotropic ray $\mathbb{R}\bm{f}$ given by $\lambda_{2}=\lambda_{3}=0$, that is, $\bm{f}= \bm{e}_{\bm 1}$. For this ray, the subset $\mathcal{E}_{1}$  in Lemma \ref{lemnormalvectors} is $\mathcal{E}_{1}=\{\bm{c}_{\bm 2}, \bm{c}_{\bm 3}\}$ and one can take $\mathcal{E}_{1}^{-}$ there to be either $\{\bm{c}_{\bm 2}\}$ or $\{ \bm{c}_{\bm 3}\}$. By definition, we have $C^{\mathrm{reg}}=\{\bm{c_1},\bm{c_2},\bm{c_{3}}\}$.
	
	A simplicial decomposition satisfying Assumption \ref{assumptionsI} can be obtained by inserting the hyperplane with normal vector $\bm{c_0}=\bm{c_1}-r\bm{c_3}$ with $r\in\Q^+$. Correspondingly we have $\l_0=\l_1-r\l_3$. The sets of normal vectors of the resulting two small simplicial polyhedrons are given by $C_{023}:=\{\bm{c_0},\bm{c_2},\bm{c_3}\}$ and $C_{021}:=\{\bm{c_0},\bm{c_2},\bm{c_1}\}$, respectively. We have
	\begin{equation*}
		p_{C}(\bm{\lambda})=\prod_{k=1}^{5}\tfrac{1+\sgn({\lambda}_{k})}2-(-1)^{5}\prod_{k=1}^{5} \tfrac{1-\sgn({\lambda}_{k})}2,\quad
		p_{C^{\mathrm{reg}}}(\bm{\lambda}) =\prod_{k=1}^{3}\tfrac{1+\sgn({\lambda}_{k})}2-(-1)^{3}\prod_{k=1}^{3} \tfrac{1-\sgn({\lambda}_{k})}2.
	\end{equation*}
	It follows that
	\begin{align*}
		p_{C^{\reg}}(\bm\l) &= p_{C^{\reg}}(\bm\l)\left(\frac{1+\sgn(\l_0)}{2}+\frac{1-\sgn(\l_0)}{2}\right)\\
		&= \left(\prod_{k=0}^3 \frac{1+\sgn(\l_k)}{2} - (-1)^3\prod_{k=0}^3 \frac{1-\sgn(\l_k)}{2}\right)\\
		&\hspace{2.5cm}+ \left(\prod_{k=1}^3 \frac{1+\sgn(\l_k)}{2}\frac{1-\sgn(\l_0)}{2} - (-1)^3\prod_{k=1}^3 \frac{1-\sgn(\l_k)}{2}\frac{1+\sgn(\l_0)}{2}\right).
	\end{align*}
	The first summand is
	\begin{align*}
		&\prod_{k\in \{0,2,3\}} \frac{1+\sgn(\l_k)}{2}\frac{1+\sgn(\l_1)}{2} - (-1)^3\prod_{k\in \{0,2,3\}} \frac{1-\sgn(\l_k)}{2}\frac{1-\sgn(\l_1)}{2}\\
		 &= p_{C_{023}}\left(1-\frac12\d(\l_1)\right)
		= p_{C_{023}} - \frac12p_{C_{023}}\d(\l_3)\d(\l_0).
	\end{align*}
	Similarly, the second summand is $p_{C_{021}}-\frac12p_{C_{021}}\d(\l_1)\d(\l_0)$. Combining them, we have
	\begin{equation*}
		p_{C^{\mathrm{reg}}} = p_{C_{023}}+p_{C_{021}} - \frac12p_{C_{023}}\delta({\lambda}_{3})\delta({\lambda}_{0}) - \frac12p_{C_{021}}\delta({\lambda}_{1})\delta({\lambda}_{0}) = p_{C_{023}}+p_{C_{021}}- p_{C^{\mathrm{reg}}}\delta({\lambda}_{3})\delta({\lambda}_{1}).
	\end{equation*}
	The contribution of $p_{C^{\mathrm{reg}}}\delta({\lambda}_{3})\delta({\lambda}_{1})$ to $\Theta_{Q,p_{C^{\mathrm{reg}}}}$ gives the Jacobi theta function and is independent of the choice of $r\in \mathbb{Q}^{+}$.
\end{example}

\appendix
\section{deformations of polygons with fixed angles}\label{secappenedix}

The condition $\theta_{k}\neq 0,\lambda_{k}\neq 0, 1\leq k\leq N$
and
\eqref{eqnrankcondition}
provide necessary and sufficient conditions for $\theta_{k}, \lambda_{k}\,\,(1\leq  k\leq N)$ to yield an $N$-gon which can be a non-convex,
or even a complex polygon (namely one for which non-consecutive edges intersect)  with the outer angles defined appropriately.
From elementary geometry, the necessary and sufficient conditions to yield a convex $N$-gon is
$$
	\sum_{m=0}^{\ell-1} \lambda_{k+m}e^{-i\sum_{a=0}^m \theta_{k+a}} \times e^{-0\cdot i} > 0,
\quad\quad \left(1\leq k\leq N, 1\leq \ell\leq N-2\right),
$$
or
$$
	\sum_{m=0}^{\ell-1} \lambda_{k+m}e^{-i\sum_{a=0}^m \theta_{k+a}} \times e^{-0\cdot i} < 0,
\quad\quad \left(1\leq k\leq N, 1\leq \ell\leq N-2\right).
$$

We next discuss the conditions such that the $N$-gon is not simple.
For $N=3$, all $N$-gons are simple. Hence we assume $N\geq 4$.
Again, by elementary geometry, it is easy to see that an $N$-gon is not simple if and only if
there exists $1\leq k\leq N$ and $t\in (0,1]$ such that 
\begin{align*}
	\left(t\lambda_{k-1} +\lambda_{k}e^{-i\theta_{k}}\right)\times \left(\lambda_{k+1} e^{-i \left(\theta_{k} +\theta_{k+1}\right)}\right)=0.
\end{align*}
This can be simplified into the following:
\begin{align*}
	\text{there exist } 1\le k\le N, ~t\in (0,1]\,,~\mathrm{such\, that}\,~
	t\lambda_{k-1}\sin (\theta_{k}+\theta_{k+1})+\lambda_{k}\sin ( \theta_{k+1})=0.
\end{align*}

The formula derived in Lemma \ref{la:quadForm}
is the signed area for the $N$-gon. While it is always non-zero for simple polygons, it could be zero for complex polygons.
Just as the area is defined to be a signed area (as the cross product), the length $\lambda_{k}$ can also be defined as the signed one with respect to the original orientation $e^{-i\varphi_{k}}$.

\begin{figure}[h]
	\centering
	\includegraphics[scale=0.5]{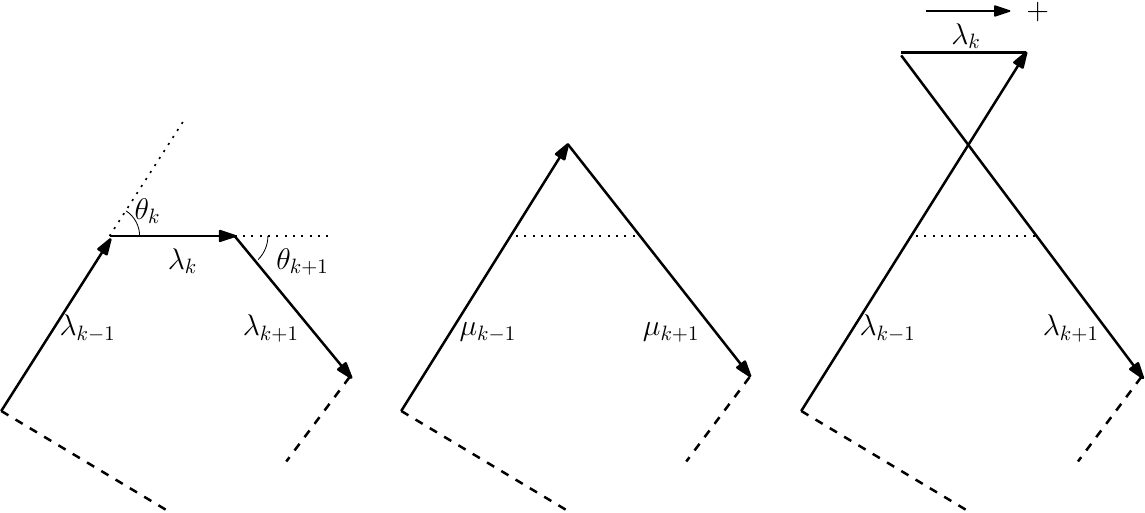}
	\caption{Deformations of convex $N$-gons. The middle one is a convex $(N-1)$-polygon with area $Q_{0}$, while the rightmost one is a complex $N$-gon.}
	\label{figure:deformationpolygon}
\end{figure}

Consider the deformations of a convex polygon
by
fixing the angles $\theta_{k},1\leq k\leq N$ and varying the signed side lengths $\lambda_{k}\,( 1\leq k\leq N)$.
Assume that
\begin{align}\label{eqndegenerationcondition}
	\theta_{k}>0\,,\quad \theta_{k+1}>0\,,\quad \theta_{k}+\theta_{k+1}<\pi\,,
\end{align}
for certain $k$. As $\lambda_{k}\to 0$, the convex $N$-gon degenerates into an $(N-1)$-polygon with
angles given by $\theta_{1},\cdots ,\theta_{k-1}, \theta_{k}+\theta_{k+1},\theta_{k+2},\cdots, \theta_{N}$, as
described in Lemma \ref{la:quadFormsignature}.
The signed (say positive) area $Q(\lambda_{1}, \cdots ,\lambda_{N};\theta_{1},\cdots ,\theta_{N})$ for the original convex $N$-gon
becomes the positive number
\begin{align*}
	Q_{0}:=Q(\mu_{1},\cdots,\mu_{k-1},\mu_{k+1}, \cdots ,\mu_{N};\theta_{1},\cdots, \theta_{k-1}, \theta_{k}+\theta_{k+1},\theta_{k+2},\cdots ,\theta_{N} ).
\end{align*}
As the signed length $\lambda_{k}$ keeps decreasing to a negative value, the
$(N-1)$-polygon becomes a non-simple one, with the same angles  $\theta_{k},1\leq k\leq N$.
The signed area becomes
\begin{align*}
	-{\sin (\theta_{k}) \sin (\theta_{k+1})\over 2\sin (\theta_{k}+\theta_{k+1})}\lambda_k^2
	+Q_{0}\,;
\end{align*}
see Figure \ref{figure:deformationpolygon} for illustration.
One can then extend the expression $Q(\l_1,\cdots,\l_N;\th_1,\cdots,\th_N)$ for the area of a convex $N$-gon to a quadratic from with variables $\l_1,\cdots,\l_N$ and parameters $\th_1,\cdots,\th_N$. The above geometric consideration gives that the sign of this quadratic form jumps at those values for $\l_k$, with $k$ subject to the conditions in \eqref{eqndegenerationcondition}, such that
\[
	-\frac{\sin(\th_k)\sin(\th_{k+1})}{2\sin(\th_k+\th_{k+1})}\l_k^2 + Q_0 = 0.
\]
Any complex polygon for which non-consecutive edges intersect at points which are not vertices can be deformed to a convex one through inverting the above deformation process iteratively. Furthermore starting from a convex $N$-gon the value of the corresponding quadratic form at any $\l_k$ ($1\le k\le N$) can be realized as the signed area for a deformed polygon which has no more than $N$ sides. The non-convex case can be treated similarly.

\begin{figure}[h]
 	\centering
 \includegraphics[scale=0.5]{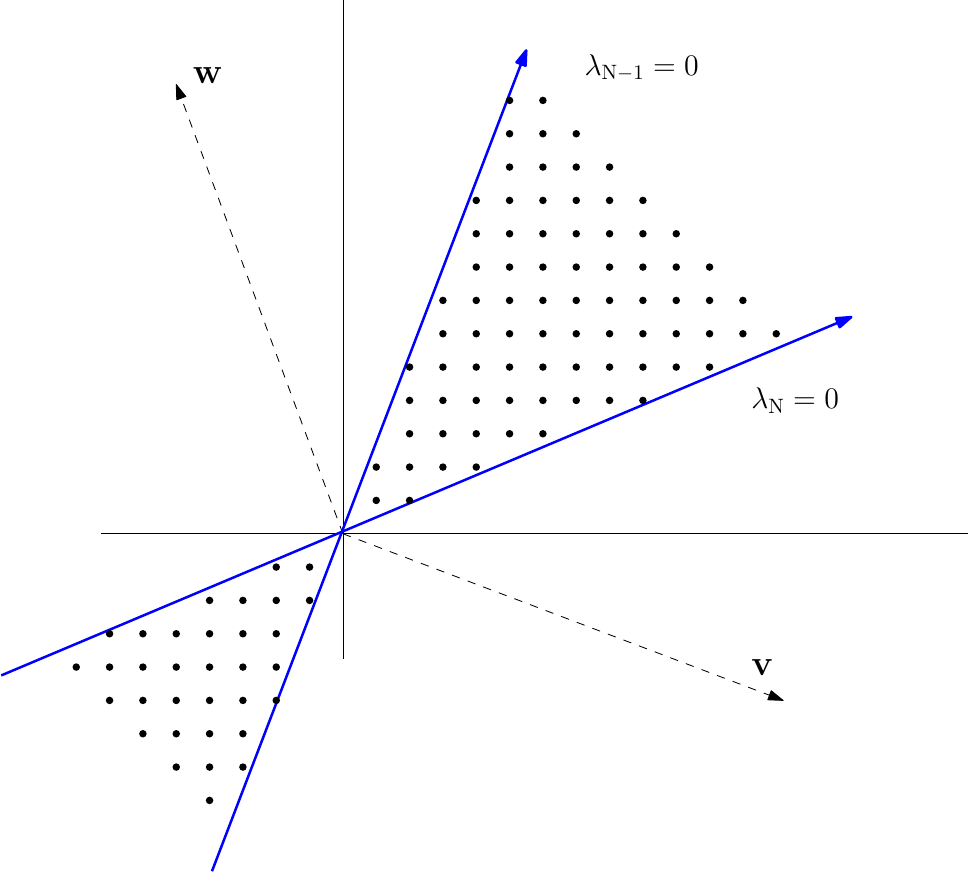}
\caption{Different regions on $\mathbb{R}^{{N-2}}$ for fixed angles.
The interior of the cone consists of simple $N$-gons and possibly complex ones (depending on the convexity), the boundary of the cone represents $n$-gons with $n\leq N-1$,
and the complement of the closure of the cone consists of complex polygons and ones with negative orientations.
}
 	\label{figure:sumoverpolygons}
 \end{figure}
 
\begin{example}\label{extrapzoid}
Consider the enumeration of isosceles trapezoids as  in Example
\ref{extrapzoidintro}.
Comparing the expressions for $f(\bm{z})$ and $\wh{f}(\bm{z})$, we see that the modular completion can be regarded as
a weighted count of trapezoids with the weight specified  by both the sign function and the  error function $E$,
in which not only simple trapezoids (i.e., those $\bb{n}$ satisfying $\chi(\bm{n}+\frac{\bm{y}}{v})\neq 0$) contribute but all possible $\bb{n}$ contribute.
The difference is a count weighted  by the complementary error function
\begin{align*}\label{eqncomplementaryEfunctionweightedcount}
	\wh f(\bm z)-f(\bm{z}) = \frac12\sum_{n_1\in\Z} q^\frac{n_1^2}{2} e^{2\pi in_1z_1} \sum_{n_2\in\Z} \left(\sgn(-n_2)-E\left(-n_2\sqrt{2v}\right)\right) q^{-\frac{n_2^2}{2}} e^{-2\pi in_2z_2}.
\end{align*}
Recall that the area of the polygon is $\frac12{(n_1^2-n_2^2)}$; intuitively the above difference is the contribution of the complex polygons and those with negative orientations; see Figures \ref{figure:deformationpolygon} and \ref{figure:sumoverpolygons} for visualizations.
We expect that similar phenomenon exists for the counting functions of more general polygons.

\end{example}


\end{document}